\documentclass[11pt]{article}

\input diagram.tex
\usepackage{amsmath} % eg. for \begin{equation}
\usepackage[mathscr]{eucal}
\usepackage{amssymb} % eg. for \leqslant
\usepackage{theorem}
\usepackage{tikz-cd}
\usetikzlibrary{positioning}

\usepackage{enumerate}

\theoremstyle{change}  % puts numbers IN FRONT of "Theorem"
\newtheorem{theorem}{Theorem}[section] % defines environment "Theorem".
\newtheorem{lemma}[theorem]{Lemma}  % defines environment "Lemma", that
\newtheorem{proposition}[theorem]{Proposition}
\newtheorem{corollary}[theorem]{Corollary}
\theorembodyfont{\rmfamily}  % has the effect that the content of Remark,
\newtheorem{Remark}[theorem]{Remark}

\newtheorem{definition}[theorem]{Definition}
\newtheorem{notation}[theorem]{Notation}
\newtheorem{nothing}[theorem]{} % empty Theoremumgebung.

\newenvironment{proof}{\noindent{\bf Proof}\ }{\qed\bigskip}

\renewcommand{\le}{\leqslant} % needs amssymb-Paket

\newcommand{\Aut}{\mathrm{Aut}}

\newcommand{\Bl}{\mathrm{Bl}}

\newcommand{\calA}{\mathcal{A}}

\newcommand{\calC}{\mathcal{C}}
        
\newcommand{\calD}{\mathcal{D}}               

\newcommand{\calF}{\mathcal{F}}

\newcommand{\calH}{\mathcal{H}}

\newcommand{\calL}{\mathcal{L}}

\newcommand{\calO}{\mathcal{O}}
\newcommand{\calP}{\mathcal{P}}

\newcommand{\calQ}{\mathcal{Q}}

\newcommand{\calU}{\mathcal{U}}
\newcommand{\calV}{\mathcal{V}}
\newcommand{\calW}{\mathcal{W}}

\newcommand{\calY}{\mathcal{Y}}
\newcommand{\calZ}{\mathcal{Z}}

\newcommand{\catfont}{\mathsf}

\newcommand{\Def}{\mathrm{Def}}

\newcommand{\FF}{\mathbb{F}}
\newcommand{\KK}{\mathbb{K}}

\newcommand{\Fun}{\mathrm{Fun}}

\newcommand{\Hom}{\mathrm{Hom}}

\newcommand{\id}{\mathrm{id}}

\newcommand{\Ind}{\mathrm{Ind}}

\newcommand{\Inf}{\mathrm{Inf}}
\newcommand{\Inn}{\mathrm{Inn}}

\newcommand{\Isom}{\mathrm{Iso}}

\newcommand{\lexp}[2]{\setbox0=\hbox{$#2$} \setbox1=\vbox to \ht0{}\,\box1^{#1\,}\!#2}

\newcommand{\lmod}[1]{\llap{\phantom{|}}_{#1}\catfont{mod}}
\newcommand{\lMod}[1]{\llap{\phantom{|}}_{#1}\catfont{Mod}}

\newcommand{\Out}{\mathrm{Out}}

\newcommand{\qed}{\nobreak\hfill
                  \vbox{\hrule\hbox{\vrule\hbox to 5pt
                  {\vbox to 8pt{\vfil}\hfil}\vrule}\hrule}}
\newcommand{\Res}{\mathrm{Res}}

\newcommand{\ZZ}{\mathbb{Z}}

\newcommand{\Br}{\mathrm{Br}}

\newcommand{\DD}{D^{\Delta}}

\newcommand{\FFT}{\mathbb{F}T}
\newcommand{\FFTD}{\mathbb{F}T^{\Delta}}
\newcommand{\RTD}{RT^{\Delta}}

\newcommand{\fracb}[2]{\frac{\raisebox{-.7ex}{$\scriptstyle #1$}}{\raisebox{.7ex}{$\scriptstyle #2$}}}
\newcommand{\Proj}{\mathrm{Proj}}
\newcommand{\Fppk}[1]{\calF_{#1pp_k}^\Delta}
\title{Diagonal $p$-permutation functors, semisimplicity, and functorial equivalence of blocks}
\author{Serge Bouc and Deniz Y\i lmaz}
\date{}
\providecommand{\keywords}[1]
{
  \small\smallskip\par	
  \hspace{2ex}\textbf{Keywords:} #1
}
\providecommand{\msc}[1]
{
  \small\smallskip\par	
  \hspace{2ex}\textbf{MSC2020:} #1
}

\begin{document}
\sloppy

\maketitle
\begin{abstract} Let $k$ be  an algebraically closed field of characteristic $p>0$, let $R$ be a commutative ring, and let $\FF$ be an algebraically closed field of characteristic 0. We consider the $R$-linear category $\Fppk{R}$ of diagonal $p$-permutation functors over $R$. We first show that the category $\Fppk{\FF}$ is {\em semisimple}, and we give a parametrization of its simple objects, together with a description of their evaluations. \par
Next, to any pair $(G,b)$ of a finite group $G$ and a block idempotent $b$ of $kG$, we associate a diagonal $p$-permutation functor $\RTD_{G,b}$ in $\Fppk{R}$. We find the {\em decomposition} of the functor $\FFTD_{G,b}$ as a direct sum of simple functors in $\Fppk{\FF}$. This leads to a characterization of {\em nilpotent blocks} in terms of their associated functors in $\Fppk{\FF}$. \par
Finally, for such pairs $(G,b)$ of a finite group and a block idempotent, we introduce the notion of {\em functorial equivalence} over $R$, which (in the case {$R=\nolinebreak\ZZ$}) is slightly weaker than $p$-permutation equivalence, and we prove a corresponding {\em finiteness theorem}: for a given finite $p$-group $D$, there is only a finite number of pairs $(G,b)$, where $G$ is a finite group and $b$ a block idempotent of $kG$ with defect isomorphic to $D$, up to functorial equivalence over~$\FF$.\par
%{\bf Keywords:} diagonal $p$-permutation functor, semisimple, block, functorial equivalence.
\end{abstract}
\keywords{diagonal $p$-permutation functor, semisimple, block, functorial equivalence.}
\msc{16S34, 18B99, 20C20, 20J15.}
\section{Introduction}
In the past decades, various categories have been considered, where objects are finite groups and morphisms are obtained from various types of double group actions. The linear representations of these categories give rise to interesting functor categories: Examples include biset functors (\cite{Bouc2010a}), $p$-permutation functors (\cite{Ducellier2015}), simple modules over Green biset functors (\cite{romero-simple}), modules over shifted representation functors (\cite{shifted}), fibered biset functors (\cite{boltje-coskun2018}, and generalizations of these (\cite{barker-ogut2020}, \cite{barker2016}). \par
In the present paper, we consider another example of a similar context: For an algebraically closed field $k$ of positive characteristic $p$, and a commutative ring $R$, we define the following category $Rpp_k^\Delta$:
\begin{itemize}
\item The objects of $Rpp_k^\Delta$ are the finite groups.
\item For finite groups $G$ and $H$, the set of morphisms $\Hom_{Rpp_k^\Delta}(G,H)$ from $G$ to $H$ in $Rpp_k^\Delta$ is equal to $R\otimes_{\ZZ} T^\Delta(H,G)$, where $T^\Delta(H,G)$ is the Grothendieck group of the category of {\em diagonal $p$-permutation $(kH,kG)$-bimodules}. These are $p$-permutation bimodules which admit only indecomposable direct summands with twisted diagonal vertices (or equivalently, $p$-permutation bimodules which are projective when considered as left or right modules).
\item The composition in $Rpp_k^\Delta$ is induced by $R$-linearity from the usual tensor product of bimodules: if $G$, $H$, and $K$ are finite groups, if $M$ is a diagonal $p$-permutation $(kH,kG)$-bimodule and $N$ is a diagonal $p$-permutation $(kK,kH)$-bimodule, then $N\otimes_{kH}M$ is a diagonal $p$-permutation $(kK,kG)$-bimodule. The composition of (the isomorphism class of) $N$ and (the isomorphism class of) $M$ is by definition (the isomorphism class of) $N\otimes_{kH}M$.
\item The identity morphism of the group $G$ is the (isomorphism class of the) $(kG,kG)$-bimodule $kG$.
\end{itemize}
The category $Rpp_k^\Delta$ is an $R$-linear category. The $R$-linear functors from $Rpp_k^\Delta$ to the category $\lMod{R}$ of $R$-modules are called {\em diagonal $p$-permutation functors} over $R$. These functors, together with natural transformations between them, form an $R$-linear abelian category $\Fppk{R}$. \par
These diagonal $p$-permutation functors have been introduced in \cite{BoucYilmaz2020}, in the case $R$ is a field $\FF$ of characteristic 0. Even though this will also be our assumption in most of the present paper, we give the above more general definition, as we will also need to consider the case $R=\ZZ$ in various places.\par
The main motivation for considering diagonal $p$-permutation functors comes from block theory, and in particular the notion of $p$-permutation equivalence of blocks of finite groups, introduced in \cite{boltje-xu} and developed in \cite{BoltjePerepelitsky2020}. This notion inserts in the following chain of equivalences of blocks of group algebras, namely
$$\hbox{Puig's equiv.}\implies \hbox{Rickard's splendid derived equiv.}\implies \hbox{$p$-permutation equiv.},$$
and these equivalences are related to important {\em structural conjectures}, such as Brou{\'e}'s abelian defect group conjecture (Conjecture 9.7.6 in~\cite{linckelmann2018}), and {\em finiteness conjectures}, such as Puig's conjecture (Conjecture 6.4.2 in~\cite{linckelmann2018}), or Donovan's conjecture (Conjecture 6.1.9 in~\cite{linckelmann2018}).\par
In this paper, we introduce yet another equivalence, weaker than $p$-permutation equivalence, between blocks of group algebras, which we call {\em functorial equivalence} over $R$: to each pair $(G,b)$ of a finite group $G$ and a block idempotent $b$ of the group algebra $kG$, we associate a canonical diagonal $p$-permutation functor over $R$, denoted by $RT^\Delta_{G,b}$. This functor is a direct summand of the representable functor $RT^\Delta_G$ at $G$, obtained from the $(kG,kG)$-bimodule $kGb$, viewed as an idempotent endomorphism of $G$ in the category $Rpp_k^\Delta$. When $(H,c)$ is a pair of a finite group $H$ and a block idempotent $c$ of $kH$, we say that $(G,b)$ and $(H,c)$ are {\em functorially equivalent over $R$} if the functors $RT^\Delta_{G,b}$ and $RT^\Delta_{H,c}$ are isomorphic in $\Fppk{R}$.\medskip\par
We obtain the following main results:
\begin{itemize}
\item The category $\Fppk{\FF}$ of diagonal $p$-permutation functors over $\FF$ is a {\em semisimple} $\FF$-linear abelian category (Theorem~\ref{Fppk semisimple}).
\item The {\em simple} diagonal $p$-permutation functors over $\FF$ are parametrized by triples $(L,u,V)$, where $(L,u)$ is a $D^\Delta$-pair (see Section~\ref{sec pairs}), and $V$ is a simple $\FF\Out(L,u)$-module (see Notation~\ref{OutLu}).
\item The {\em evaluation} $S_{L,u,V}(G)$ of the simple functor $S_{L,u,V}$ parametrized by the triple $(L,u,V)$, at a finite group $G$, is explicitly computed in Corollary~\ref{evaluation simple}.
\item The {\em multiplicity} of any simple functor $S_{L,u,V}$ in the functor $\FFTD_{G,b}$ associated to a block idempotent $b$ of a finite group $G$ is explicitly given by three equivalent formulas (Theorem~\ref{thm multiplicityformulas}): One in terms of fixed points of some subgroups of $\Out(L,u)$ on $V$, the second one in terms of the ``$u$-invariant'' $(G,b)$-Brauer pairs $(P,e)$, and the third one in terms of the ``$u$-invariant'' local pointed subgroups $P_\gamma$ of $G$ on $kGb$.
\item We give two characterizations (Theorem~\ref{thm nilpotentcharacterization}) of {\em nilpotent blocks} in terms of their associated diagonal $p$-permutation functors: let $b$ be a block idempotent of the group algebra $kG$. Then $b$ is nilpotent if and only if one of the equivalent following conditions holds:
\begin{itemize} 
\item If $S_{L,u,V}$ is a simple summand of $\FFTD_{G,b}$, then $u=1$.
\item The functor $\FFTD_{G,b}$ is isomorphic to the representable functor $\FFTD_D$ for some $p$-group $D$.
\end{itemize} 
\item We show (Proposition~\ref{fun equivalence}) that if $(G,b)$ and $(H,c)$ are functorially equivalent blocks over $\FF$, then $kGb$ and $kHc$ have the {\em same number of simples modules}, and $b$ and $c$ have {\em isomorphic defect groups}. 
\item We show (Proposition~\ref{fun equivalence}) that if $b$ and $c$ are {\em nilpotent} blocks of $G$ and~$H$, respectively, then $(G,b)$ and $(H,c)$ are functorially equivalent over $\FF$ {\em if and only} if $b$ and $c$ have isomorphic defect groups.
\item We prove a {\em finiteness theorem} (Theorem~\ref{finiteness}) for functorial equivalence of blocks: for a given finite $p$-group $D$, there are only finitely many pairs $(G,b)$, where $G$ is a finite group, and $b$ is a block idempotent of $kG$ with defect isomorphic to $D$, up to functorial equivalence over~$\FF$.
\item We give a sufficient condition (Theorem~\ref{abelian defect}) for two pairs $(G,b)$ and $(H,c)$ to be functorially equivalent over $\FF$ in the situation of {\em Brou{\'e}'s abelian defect group conjecture}.
\end{itemize}
The paper is organized as follows: Sections 2 to 5 are devoted to technical tools used in the proof of our semisimplicity Theorem~\ref{Fppk semisimple}. Section 2 deals with Brauer quotients of tensor products of diagonal $p$-permutation bimodules, Section 3 recalls the definitions of pairs, $\DD$-pairs, and idempotents of $p$-permutation rings. The main theorem of this section is Theorem 3.7. In Section~4 we state some results from Clifford theory, and in Section~5 a theorem on some equivalences of abelian categories. Section~6 is devoted to the proof of our semisimplicity theorem (Corollary~\ref{Fppk semisimple}). In Section~7, we compute the evaluations of the simple diagonal $p$-permutation functors (Corollary~\ref{evaluation simple}). In Section~8, we introduce the diagonal $p$-permutation functors associated to blocks of finite groups, and we describe their decomposition as a direct sum of simple functors (Theorem~\ref{thm multiplicityformulas}). In Section~9, we apply these results to the characterization of nilpotents blocks in terms of the associated functors (Theorem~\ref{thm nilpotentcharacterization}). In Section~10, we introduce functorial equivalence of blocks, and state some of its basic consequences (Proposition~\ref{fun equivalence}); next we prove our finiteness theorem for functorial equivalence over $\FF$ (Theorem~\ref{finiteness}). Finally, Section~11 considers the case of blocks with abelian defect groups (Theorem~\ref{abelian defect}).

\section{Brauer character formula}\label{sec Brauercharacterformula}

Throughout $G$, $H$ and $K$ denote finite groups. Let $\FF$ denote an algebraically closed field of characteristic $0$ and let $k$ denote an algebraically closed field of positive characteristic $p$. We assume that all the modules considered are finitely generated.

Let $M$ be a $(kG,kH)$-bimodule. For $m\in M$ and $(g,h)\in G\times H$, the formula $(g,h)m=gmh^{-1}$ induces an isomorphism between categories $\lmod{kG}_{kH}$ and $\lmod{k[G\times H]}$. In what  follows, we will often use this isomorphism to identify left $k[G\times H]$-modules via $(kG,kH)$-bimodules.

Let $M$ be a $kG$-module and let $P$ be a $p$-subgroup of $G$. The Brauer construction of $M$ at $P$ will be denoted by $M[P]$.   

\begin{lemma}\label{lem deflationcharacter}
Let $N\unlhd G$ be a normal subgroup of $G$ and let $V$ be an $\FF G$-module with character~$\chi$. Then the character of $\Def^G_{G/N}V$ is given by
\begin{align*}
gN\mapsto \frac{1}{|N|}\sum_{n\in N} \chi(gn)
\end{align*}
for $gN\in G/N$. 
\end{lemma}
\begin{proof}
By \cite[Lemma 7.1.3]{Bouc2010a} the character of $\Def^G_{G/N}V$ is given by the formula
\begin{align*}
gN\mapsto \frac{1}{|G|}\sum_{\substack{u\in G/N, h\in G\\ gN\cdot u=u\cdot h}}\chi(h)=\frac{1}{|G|} |G/N|\sum_{n\in N} \chi(gn)=\frac{1}{|N|}\sum_{n\in N}\chi(gn)\,,
\end{align*}
as desired.
\end{proof}

Let $X$ be a subgroup of $G\times H$ and let $L$ be a finite dimensional $\FF X$-module. Let also $Y$ be a subgroup of $H\times K$, and $M$ a finite dimensional $\FF Y$-module. Since $k_1(X)\times k_2(X)$ is a subgroup of~$X$, the module $L$ can be viewed as an $\big(\FF k_1(X),\FF k_2(X)\big)$-bimodule. Similarly, $M$ can be viewed as an $\big(\FF k_1(Y),\FF k_2(Y)\big)$-bimodule.  Set $S=k_2(X)\cap k_1(Y)$. Then the tensor product $L\otimes_{\FF S}M$ is an $\big(\FF k_1(X), \FF k_2(Y)\big)$-bimodule. For $(a,b)\in X*Y$, choose $h\in H$ such that $(a,h)\in X$ and $(h,b)\in Y$. Then
$$(a,b)(l\otimes m)=(a,h)l\otimes (h,b)m$$
is a well defined element of $L\otimes_{\FF S}M$, and this defines a structure of $\FF(X*Y)$-module on $L\otimes_{\FF S}M$. This construction is first used in \cite{Bouc2010b}.

Let $X\times_H Y:=\{((g,h), (\tilde{h}, k))\in X\times Y \mid h=\tilde{h}\}$. Consider the surjective group homomorphism
\begin{align*}
\nu: X\times_H Y\to X*Y\,, \quad \left((g,h),(h,k)\right)\mapsto (g,k)\,.
\end{align*}
The kernel of $\nu$ is $\{\left((1,h),(h,1)\right)\mid h\in S\}$ and we denote by $\bar{\nu}: (X\times_H Y)/\ker(\nu)\to X*Y$ the isomorphism induced by $\nu$. 

\begin{lemma}\label{lem characterofstar}
Let $X\le G\times H$ and $Y\le H\times K$ be subgroups. Let also $L$ be a finite dimensional $\FF X$-module and $M$ a finite dimensional $\FF Y$-module. Then the character $\chi_{L\otimes_{\FF S}M}$ of the $\FF (X*Y)$-module $L\otimes_{\FF S} M$ is given by
$$\chi_{L\otimes_{\FF S}M}(a,b)=\frac{1}{|S|}\sum_{\substack{h\in H\\(a,h)\in X\\(h,b)\in Y}}\chi_L(a,h)\,\chi_M(h,b),$$
for all $(a,b)\in X*Y$, where $\chi_L$ and $\chi_M$ are the characters of $L$ and $M$, respectively.
\end{lemma}
\begin{proof}
Let $(a,b)\in X*Y$ be an arbitrary element and $c\in H$ such that $(a,c)\in X$ and $(c,b)\in Y$. By \cite[Proposition~2.8]{Boltje20}, we have
\begin{align*}
\chi_{L\otimes_{\FF S}M}=\left(\Isom(\bar{\nu})\circ \Def^{X\times_H Y}_{(X\times_H Y)/\ker(\nu)}\circ \Res^{X\times Y}_{X\times_H Y}\right)(\chi_L\times \chi_M)\,.
\end{align*}
Hence by Lemma \ref{lem deflationcharacter} we have
\begin{align*}
\chi_{L\otimes_{\FF S}M}(a,b)&=\frac{1}{|S|}\sum_{h\in S}\chi_L(a, hc)\chi_M(hc, b)\\&
=\frac{1}{|S|}\sum_{h\in Sc}\chi_L(a,h)\chi_M(h,b)\\&
=\frac{1}{|S|}\sum_{\substack{h\in H\\(a,h)\in X\\(h,b)\in Y}}\chi_L(a,h)\,\chi_M(h,b)\,,
\end{align*}
as desired.
\end{proof}

\begin{notation}
{\rm (a)} Let $U\le G$ and $W\le K$ be subgroups and $\gamma: W\to U$ a group isomorphism. We set
\begin{align*}
\Delta(U,\gamma,W)=\big\{\big(\gamma(w),w\big):w\in W\big\}
\end{align*}
for the corresponding twisted diagonal subgroup of $G\times K$.

\smallskip
{\rm (b)}
Let $\Delta(U,\gamma, W)$ be a twisted diagonal subgroup of $G\times K$. 
Let $\Gamma_{H}(U,\gamma, W)$ denote the set of triples $(\alpha, V,\beta)$ where $V$ is a subgroup of $H$, and $\alpha: V\to U$ and $\beta: W\to V$ are group isomorphisms with the property that $\gamma=\alpha \circ \beta$. The group $N_{G\times K}(\Delta(U,\gamma, W))\times H$ acts on $\Gamma_H(U,\gamma, W)$ in the following way: if $\big((x,z),h\big)\in N_{G\times K}(\Delta(U,\gamma, W))\times H$ and $(\alpha,V,\beta)\in \Gamma_H(U,\gamma, W)$, then
$$\big((x,z),h\big)\cdot (\alpha,V,\beta)=(i_x\circ\alpha\circ i_h^{-1},{^hV},i_h\circ \beta\circ i_z^{-1}),$$
where $i_u$ denotes the conjugation by an element $u$.
\end{notation}

\begin{proposition}\cite[Corollary~7.4(b)]{BoltjePerepelitsky2020} \label{Brauer quotient}
Let $L$ be a $p$-permutation $(kG, kH)$-bimodule and $M$ a $p$-permutation $(kH, kK)$-bimodule. Suppose that all of the indecomposable summands of $L$ and $M$ have twisted diagonal vertices. Let $\tilde{\Gamma}_{H}(U,\gamma,W)$ denote a set of representatives of $N_{G\times K}(\Delta(U,\gamma, W))\times H$-orbits of $\Gamma_H(U,\gamma, W)$.    Then the $kN_{G\times K}(\Delta(U,\gamma,W))$-module $(L\otimes_{kH} M)[\Delta(U,\gamma,W)]$ is isomorphic to
\begin{align*}
\bigoplus_{(\alpha, V,\beta)}\Ind_{N_{G\times H}(\Delta(U,\alpha, V)) \ast N_{H\times K}(\Delta(V,\beta, W))} ^{N_{G\times K}(\Delta(U,\gamma, W))} \left( L\big[\Delta(U,\alpha, V)\big]\otimes_{kC_H(V)} M\big[\Delta(V,\beta, W)\big]\right)\,,
\end{align*}
where $(\alpha, V,\beta)\in \tilde{\Gamma}_{H}(U,\gamma,W)$.
\end{proposition}

We say that $(P,u)$ is a {\em pair} of a finite group $S$, if $P$ is a $p$-subgroup of $S$ and $u$ is a $p'$-element of $N_S(P)$.  Let $(P,u)$ be a pair of $S$ and let $X$ be a $p$-permutation $kS$-module. Then $\tau_{P,u}^S(X)\in \FF$ is defined as the value at $u$ of the Brauer character of $X[P]$ (see~\cite{BoucThevenaz2010} Notation~2.1 and Notation~2.15 for details). \par
In the next proposition, we will consider three finite groups $G$, $H$ and $K$, and we will need to lift $p$-permutation $(kG,kH)$-bimodules and $(kH,kK)$-bimodules from characteristic $p$ to characteristic~0. To do this, we can choose a $p$-modular system $(\KK,\calO,k)$ containing~$k$, where $\calO$ is a complete discrete valuation ring with residue field $k$ and field of fractions $\KK$ of characteristic 0. For a finite group $S$ and a $p$-permutation $kS$-module $X$, there is a unique $p$-permutation $\calO S$-lattice $X^{\calO}$, up to isomorphism, such that $X^{\calO}/J(\calO)X^{\calO}\cong X$. We denote by $X^0$ the $\KK S$-module $\KK\otimes_{\calO} X^{\calO}$. For a $p$-subgroup $P$ of $S$ and a (possibly $p$-singular) element $u$ of $N_S(P)$, we define $\widehat{\tau}_{P,u}^S(X)$ as the value at $u$ of the (ordinary) character of $X[P]^0$. As $\widehat{\tau}_{P,u}^S(X)$ is a sum of roots of unity, we may assume that $\widehat{\tau}_{P,u}^S(X)$ lies in our algebraically closed field $\FF$ of characteristic~0. Moreover $\widehat{\tau}_{P,u}^S(X)=\tau_{P,u}^S(X)$ when $u$ is $p$-regular. With this notation: 
\begin{proposition} \label{prop brauercharacterofinduction}Let $(s,t)\in N_{G\times K}\big(\Delta(U,\gamma,W)\big)$. Let $L$ be a $p$-permutation $(kG,kH)$-bimodule and $M$ a $p$-permutation $(kH,kK)$-bimodule. Suppose that all of the indecomposable summands of $L$ and $M$ have twisted diagonal vertices. Then
	$$\widehat{\tau}_{\Delta(U,\gamma,W),(s,t)}^{G\times K}(L\otimes_{kH}M)\!=\!\frac{1}{|H|}\!\!\!\!\!\sum_{\substack{\rule{0ex}{2ex}(\alpha,V,\beta)\in \Gamma_H(U,\gamma,W),\,h\in H\\\rule{0ex}{1.6ex}(s,h)\in N_{G\times H}(\Delta(U,\alpha,V))\\\rule{0ex}{1.6ex}(h,t)\in N_{H\times K}(\Delta(V,\beta,W))}}\!\!\!\!\!\widehat{\tau}_{\Delta(U,\alpha,V),(s,h)}^{G\times H}(L)\,\widehat{\tau}_{\Delta(V,\beta,W),(h,t)}^{H\times K}(M).$$
\end{proposition}
\begin{proof}
We use Proposition~\ref{Brauer quotient}, where we set for simplicity
\begin{align*}
	N_{U,\gamma,W}&=N_{G\times K}\big(\Delta(U,\gamma,W)\big)\\
	N_{U,\alpha,V}&=N_{G\times H}\big(\Delta(U,\alpha,V)\big)\\
	N_{V,\beta,W}&=N_{H\times K}\big(\Delta(V,\beta,W)\big)\\
	N_{U,\alpha,V,\beta,W}&=N_{G\times H}\big(\Delta(U,\alpha,V)\big)*N_{H\times K}\big(\Delta(V,\beta,W)\big).
\end{align*}
We write $\tilde{\Gamma}_{U,\gamma,W}=\tilde{\Gamma}_H(U,\gamma,W)$ for short, for a set of representatives of the orbits of $N_{U,\gamma,W}\times H$-action on $\Gamma_{U,\gamma,W}=\Gamma_H(U,\gamma,W)$. For $(\alpha,V,\beta)\in \Gamma_{U,\gamma,W}$, we denote by $S_{U,\alpha,V,\beta,W}$ the stabilizer of $(\alpha,V,\beta)$ in $N_{U,\gamma,W}\times H$, i.e.,
$$S_{U,\alpha,V,\beta,W}=\big\{\big((x,z),h\big)\in N_{U,\gamma,W}\times H\mid i_x\circ\alpha=\alpha\circ i_h,\;i_h\circ\beta=\beta\circ i_z\big\}.$$
With this notation, we have an isomorphism of $kN_{U,\gamma,W}$-modules
$$(L\otimes_{kH}M)\big[\Delta(U,\gamma,W)\big]\cong \!\!\!\bigoplus_{(\alpha,V,\beta)\in\tilde{\Gamma}_{U,\gamma,W}}\!\!\!\Ind_{N_{U,\alpha,V\beta,W}}^{N_{U,\gamma,W}}\left(L\big[\Delta(U,\alpha,V)\big]\otimes_{k C_H(V)}M\big[\Delta(V,\beta,W)\big]\right).$$
This isomorphism can be lifted to $\calO$, to give an isomorphism of $\FF N_{U,\gamma,W}$-modules
$$(L\otimes_{kH}M)\big[\Delta(U,\gamma,W)\big]^0\cong\!\!\! \bigoplus_{(\alpha,V,\beta)\in\tilde{\Gamma}_{U,\gamma,W}}\!\!\!\Ind_{N_{U,\alpha,V\beta,W}}^{N_{U,\gamma,W}}\left(L\big[\Delta(U,\alpha,V)\big]^0\otimes_{\FF C_H(V)}M\big[\Delta(V,\beta,W)\big]^0\right),$$
and $\widehat{\tau}_{\Delta(U,\gamma,W),(s,t)}^{G\times K}(L\otimes_{kH}M)$ is equal to the trace of $(s,t)$ acting on the right hand side. This implies that
\begin{align*}
	\widehat{\tau}_{\Delta(U,\gamma,W),(s,t)}^{G\times K}(L\otimes_{kH}M)&=\sum_{(\alpha,V,\beta)\in\tilde{\Gamma}_{U,\gamma,W}}\frac{1}{|N_{U,\alpha,V,\beta,W}|}\sum_{\substack{(a,b)\in N_{U,\gamma,W}\\(^as,^bt)\in N_{U,\alpha,V,\beta,W}}}\theta_{U,\alpha,V,\beta,W}(^as,^bt)\\
=&\!\!\!\sum_{(\alpha,V,\beta)\in \Gamma _{U,\gamma,W}}\frac{|S_{U,\alpha,V,\beta,W}|}{|N_{U,\gamma,W}||H||N_{U,\alpha,V,\beta,W}|}\hspace{-3ex}\sum_{\substack{(a,b)\in N_{U,\gamma,W}\\(^as,^bt)\in N_{U,\alpha,V,\beta,W}}}\hspace{-3ex}\theta_{U,\alpha,V,\beta,W}(^as,^bt)
\end{align*}
where $ \theta_{U,\alpha,V,\beta,W}$ is the character of $L\big[\Delta(U,\alpha,V)\big]^0\otimes_{\FF C_H(V)}M\big[\Delta(V,\beta,W)\big]^0$.

Now we observe that $p_1(S_{U,\alpha,V,\beta,W})=N_{U,\alpha,V,\beta,W}$ and $k_2(S_{U,\alpha,V,\beta,W})=C_H(V)$. It follows that 
$$|S_{U,\alpha,V,\beta,W}|=|N_{U,\alpha,V,\beta,W}||C_H(V)|,$$
so
$$\widehat{\tau}_{\Delta(U,\gamma,W),(s,t)}^{G\times K}(L\otimes_{kH}M)=\sum_{(\alpha,V,\beta)\in \Gamma_{U,\gamma,W}}\frac{|C_H(V)|}{|N_{U,\gamma,W}||H|}\hspace{-3ex}\sum_{\substack{(a,b)\in N_{U,\gamma,W}\\({^as},{^bt})\in N_{U,\alpha,V,\beta,W}}}\hspace{-3ex}\theta_{U,\alpha,V,\beta,W}({^as},{^bt}).$$
But for $(a,b)\in N_{U,\gamma,W}$, saying that $({^as},{^bt})\in N_{U,\alpha,V,\beta,W}$ amounts to saying that $(s,t)\in N_{U,\alpha,V,\beta,W}^{(a,b)}=N_{U,i_a^{-1}\circ\alpha,V,\beta\circ i_b,W}$. Moreover $\theta_{U,\alpha,V,\beta,W}({^as},{^bt})=\theta_{U,i_a^{-1}\circ\alpha,V,\beta\circ i_b,W}(s,t)$. So
\begin{align*}
\widehat{\tau}_{\Delta(U,\gamma,W),(s,t)}^{G\times K}(L\otimes_{kH}M)&=\fracb{|C_H(V)|}{|N_{U,\gamma,W}||H|}\sum_{\substack{(a,b)\in N_{U,\gamma,W}\\(i_a^{-1}\circ\alpha,V,\beta\circ i_b)\in \Gamma _{U,\gamma,W}\\(s,t)\in N_{U,i_a^{-1}\circ\alpha,V,\beta\circ i_b,W}}} \theta_{U,i_a^{-1}\circ\alpha,V,\beta\circ i_b,W}(s,t)\\
&=\fracb{|C_H(V)|}{|H|}\sum_{\substack{(\alpha,V,\beta)\in \Gamma_{U,\gamma,W}\\(s,t)\in N_{U,\alpha,V,\beta,W}}} \theta_{U,\alpha,V,\beta,W}(s,t).
\end{align*}
Now by Lemma~\ref{lem characterofstar}
$$\theta_{U,\alpha,V,\beta,W}(s,t)=\fracb{1}{|C_H(V)|}\sum_{\substack{h\in H\\(s,h)\in N_{U,\alpha,V}\\(h,t)\in N_{V,\beta,W}}}\widehat{\tau}_{\Delta(U,\alpha,V),(s,h)}^{G\times H}(L)\,\widehat{\tau}_{\Delta(V,\beta,W),(h,t)}^{H\times K}(M).$$
It follows that
\begin{align*}
\widehat{\tau}_{\Delta(U,\gamma,W),(s,t)}^{G\times K}(L\otimes_{kH}M)&=\frac{1}{|H|}\sum_{\substack{(\alpha,V,\beta)\in\Gamma_{U,\gamma,W}\\h\in H\\(s,h)\in {N_{U,\alpha,V}}\\(h,t)\in {N_{V,\beta,W}}}} \widehat{\tau}_{\Delta(U,\alpha,V),(s,h)}^{G\times H}(L)\,\widehat{\tau}_{\Delta(V,\beta,W),(h,t)}^{H\times K}(M)\,. 
\end{align*}

\end{proof}

\begin{corollary}\label{cor brauercharacteroftensor}  Let $L$ be a $p$-permutation $(kG,kH)$-bimodule and let $M$ be a $p$-permutation $(kH,kK)$-bimodule. Suppose that all of the indecomposable summands of $L$ and $M$ have twisted diagonal vertices. Then, for any diagonal pair $\big(\Delta(U,\gamma,W),(s,t)\big)$ of $G\times K$
$$\tau_{\Delta(U,\gamma,W),(s,t)}^{G\times K}(L\otimes_{kH}M)\!=\!\frac{1}{|H|}\!\!\!\!\!\sum_{\substack{\rule{0ex}{2ex}(\alpha,V,\beta)\in \Gamma_H(U,\gamma,W),\,h\in H_{p'}\\\rule{0ex}{1.6ex}(s,h)\in N_{G\times H}(\Delta(U,\alpha,V))\\\rule{0ex}{1.6ex}(h,t)\in N_{H\times K}(\Delta(V,\beta,W))}}\!\!\!\!\!\tau_{\Delta(U,\alpha,V),(s,h)}^{G\times H}(L)\,\tau_{\Delta(V,\beta,W),(h,t)}^{H\times K}(M),$$
where $H_{p'}$ is the set of $p'$-elements of $H$.
\end{corollary}
\begin{proof}
For $(\alpha,V,\beta)\in\Gamma_{H}(U,\gamma,W)$, the indecomposable direct summands of the $\FF N_{U,\alpha,V}$-module $L[\Delta(U,\alpha,V)]^0$ have twisted diagonal vertices. Similarly, the indecomposable direct summands of the $\FF N_{V,\beta, W}$-module $M[\Delta(V,\beta, W)]^0$ have twisted diagonal vertices. Now assume that $(s,h)\in N_{U,\alpha,V}$ is such that $\widehat{\tau}_{\Delta(U,\alpha,V),(s,h)}^{G\times H}(L)\neq 0$. Then by \cite[Theorem~4.7.4]{NagaoTsushima1989}, the $p$-part of the element $(s,h)$ is contained in a twisted diagonal $p$-subgroup of $G\times H$. Since $s$ is a $p'$-element, this means that $h$ is a $p'$-element as well. Therefore the summation over $h\in H$ in Proposition~\ref{prop brauercharacterofinduction} reduces to a summation over $h\in H_{p'}$, and then both $(s,h)$ and $(h,t)$ are $p'$-elements, so we can replace $\widehat{\tau}$ by $\tau$ throughout.
\end{proof}

\section{Pairs}\label{sec pairs}
 
Recall that $k$ denotes an algebraically closed field of characteristic $p$ and $\FF$ denotes an algebraically closed field of characteristic zero.  Also, $G$, $H$ and $K$ denote finite groups.  

\begin{nothing}
{\rm (a)} We denote by $T(G)$ the Grothendieck group of $p$-permutation $kG$-modules.  Let $\calQ_{G,p}$ denote the set of pairs $(P,u)$ where $P$ is a $p$-subgroup of $G$ and $u$ is a $p'$-element of $N_G(P)$. The group $G$ acts on the set $\calQ_{G,p}$ via conjugation and we write $[\calQ_{G,p}]$ for a set of representatives of $G$-orbits on $\calQ_{G,p}$. \par
Let $M$ be a $p$-permutation $kG$-module. Recall from Section~\ref{sec Brauercharacterformula} that $\tau^G_{P,u}(M)$ is defined as the value at $u$ of the Brauer character of $M[P]$. We extend this $\FF$-linearly to obtain a map $\tau^G_{P,u}$ from $\FF T(G):=\FF\otimes_{\ZZ} T(G)$ to $\FF$. It is well known that the set of maps $\tau_{P,u}$, for $(P,u)\in [\calQ_{G,p}]$,  is the set of all distinct $\FF$-algebra homomorphisms from $\FF T(G)$ to $\FF$ (see e.g. Proposition 2.18 of \cite{BoucThevenaz2010} for details), and hence the primitive idempotents $F^G_{P,u}$ of $\FF T(G)$ are indexed by $[\calQ_{G,p}]$. The idempotent $F^G_{P,u}$ is defined by the property that for any $(Q,t)\in [\calQ_{G,p}]$,
\begin{equation*}
\tau^G_{Q,t}(F^G_{P,u}) = \left \{
  \begin{aligned}
    &1 && \text{if}\ (Q,t)=_G (P,u),\\
    &0 && \text{otherwise}\,.
  \end{aligned} \right.
\end{equation*}

\smallskip
{\rm (b)} More generally, we often consider pairs $(P,s)$ where $P$ is a $p$-group and $s$ is a generator of a $p'$-group acting on $P$. We write $P\langle s\rangle := P\rtimes \langle s\rangle$ for the corresponding semi-direct product. We say that two pairs $(P,s)$ and $(Q,t)$ are {\em isomorphic} and write $(P,s)\cong (Q,t)$, if there is a group isomorphism $f: P\langle s\rangle \to Q\langle t\rangle$ that sends $s$ to a conjugate of $t$.  The following type of pairs will play a crucial role in this paper.

\begin{definition}\cite{BoucYilmaz2020}
A pair $(P,s)$ is called a {\em $\DD$-pair}, if $C_{\langle s\rangle}(P)=1$.
\end{definition}

See \cite[Proposition~5.6]{BoucYilmaz2020} for more properties of $\DD$-pairs. Note that for an arbitrary pair $(P,s)$, the pair $(\tilde{P},\tilde{s}):=(PC_{\langle s\rangle}(P)/C_{\langle s\rangle}(P), sC_{\langle s\rangle}(P))$ is a $\DD$-pair.  
\end{nothing}

\begin{lemma}\label{lem pairsofdirectproduct}
{\rm (i)} Let $\big(\Delta(U,\gamma, W), (s,t)\big)$ be a pair of $G\times K$. Then we have $(\tilde{U},\tilde{s})\cong (\tilde{W},\tilde{t})$.

\smallskip

{\rm (ii)} Let $(U,s)$ be a pair of $G$ and let $(W, t)$ be a pair of $K$. Suppose that $(U,s)\cong (W,t)$. Then $\big(\Delta(U,\gamma, W), (s,t)\big)$ is a pair of $G\times K$ for some group isomorphism $\gamma: W\to U$.

\smallskip

{\rm (iii)} Let $(U,s)$ be a pair of $G$ and let $(W, t)$ be a pair of $K$. Assume that we have $(\tilde{U},\tilde{s})\cong (\tilde{W},\tilde{t})$. Then $\left(\Delta(U,\gamma, W), (s,t)\right)$ is a pair of $G\times K$ for some group isomorphism $\gamma: W\to U$.
\end{lemma}
\begin{proof}
{\rm(i)} The element $(s,t)$ normalizes the group $\Delta(U,\gamma,W)$ means that $s$ normalizes $U$, $t$ normalizes $W$, and $\gamma(\lexp{t}u)=\lexp{s}\gamma(u)$ for all $u\in U$. Set $X:=N_{\langle s\rangle\times \langle t\rangle}(\Delta(U,\gamma,W))\le \langle s\rangle\times\langle t\rangle$. Then we have
\begin{align*}
p_1(X)=\langle s\rangle, \quad p_2(X)=\langle t\rangle, \quad k_1(X)=C_{\langle s\rangle}(U), \quad k_2(X)=C_{\langle t\rangle}(W)\,.
\end{align*}
Hence we have a group isomorphism
\begin{align*}
\eta:  \langle t\rangle/C_{\langle t\rangle}(W)\to \langle s\rangle /C_{\langle s\rangle}(U)
\end{align*}
that sends $tC_{\langle t\rangle}(W)$ to $sC_{\langle s\rangle}(U)$. The map
\begin{align*}
\theta: W\langle t\rangle/C_{\langle t\rangle}(W)\to U\langle s\rangle/C_{\langle s\rangle}(U)
\end{align*}
defined as $\theta(w t^iC_{\langle t\rangle}(W)):=\gamma(t)\eta(t^iC_{\langle t\rangle}(W))$ is an isomorphism that sends  $tC_{\langle t\rangle}(W)$ to $sC_{\langle s\rangle}(U)$. Hence the pairs $\left(UC_{\langle s\rangle}(U)/C_{\langle s\rangle}(U), sC_{\langle s\rangle}(U)\right)$ and $\left(WC_{\langle t\rangle}(W)/C_{\langle t\rangle}(W), tC_{\langle t\rangle}(W)\right)$ are isomorphic. This proves the claim. 

\smallskip

{\rm (ii)} Let $f:W\langle t\rangle \to U\langle s\rangle$ be a group isomorphism that sends $t$ to a conjugate of $s$. Let $u\in U$ be an element with the property that $f(t)=usu^{-1}$. Let $i_{u^{-1}}$ denote the automorphism of $U$ induced by conjugation with $u^{-1}$ and define $\gamma= i_{u^{-1}}\circ f :W\to U$. We have
\begin{align*}
\gamma(\lexp{t}w)=u^{-1}f(\lexp{t}w)u=u^{-1}f(t)f(w)f(t^{-1})u=su^{-1}f(w)us^{-1}=\lexp{s}\gamma(w)
\end{align*}
for all $w\in W$. This shows that $\left(\Delta(U,\gamma, W), (s,t)\right)$ is a pair of $G\times K$.

\smallskip

{\rm (iii)} By part (ii) there exists a group isomorphism 
\begin{align*}
\tilde{\gamma}: WC_{\langle t\rangle}(W)/C_{\langle t \rangle}(W)\to UC_{\langle s\rangle}(U)/C_{\langle s\rangle}(U)
\end{align*}
with the property that $\tilde{\gamma}(\lexp{\bar{t}}{\bar{w}})=\lexp{\bar{s}}{\tilde{\gamma}(\bar{w})}$ for all $\bar{w}\in WC_{\langle t\rangle}(W)/C_{\langle t \rangle}(W)$. The map $\tilde{\gamma}$ induces an isomorphism $\gamma:W\to U$ with the property that  $\gamma(\lexp{t}w)=\lexp{s}\gamma(w)$ for all $w\in W$. This means that $\left(\Delta(U,\gamma, W), (s,t)\right)$ is a pair of $G\times K$.
\end{proof}

\begin{notation}
 For any $p$-permutation $kG$-module $W$, we set $\widetilde{W}:=\Ind_{\Delta G}^{G\times G} W$. This defines an algebra homomorphism from $\FFT(G)$ to $\FFTD(G,G)$. 
\end{notation}

\begin{Remark}\label{rem conjugatepairs}
Let $(P,r)$ be a pair of $G$ and let $\big(\Delta(U,\alpha, V), (s,z)\big)$ be a diagonal pair of $G\times G$. Then we have
\begin{equation*}
\tau^{G\times G}_{\Delta(U,\alpha, V), (s,z)}(\widetilde{F^G_{P,r}}) = \left \{
  \begin{aligned}
    &\vert C_G\left(P\langle r\rangle\right)\vert, && \text{if}\ \big(\Delta(U,\alpha, V), (s,z)\big)=_{G\times G} \big(\Delta(P), (r,r)\big) \\
    &0, && \text{otherwise}\,.
  \end{aligned} \right.
\end{equation*}
Indeed, after identifying the group $G$ with $\Delta G$, we can also identify $F^G_{P,r}\in \FFT(G)$ via $F^{\Delta G}_{\Delta P, (r,r)} \in \FFT(\Delta G)$. Therefore by \cite[Proposition~3.2]{BoucThevenaz2010} we have
\begin{align*}
\Ind^{G\times G}_{\Delta G} F^{\Delta G}_{\Delta P, (r,r)}= | C_G\left(P\langle r\rangle\right)| \cdot F^{G\times G}_{\Delta P, (r,r)}
\end{align*}
which implies the equality above.
\end{Remark}

\begin{lemma}\label{lem scalarmultiple}
Let $\big(\Delta(P,\gamma, Q), (r,u)\big)$ be a pair of $G\times K$. Then 
\begin{align*}
\widetilde{F^G_{P,r}}\otimes_{kG} F^{G\times K}_{\Delta(P,\gamma, Q), (r,u)}\otimes_{kK} \widetilde{F^K_{Q,u}}= F^{G\times K}_{\Delta(P,\gamma, Q), (r,u)} \quad \text{in} \quad \FFTD(G,K)\,.
\end{align*}
\end{lemma}
\begin{proof}
Let $\big(\Delta(P',\gamma', Q'), (r',u')\big)$ be a pair of $G\times K$. Then by Corollary \ref{cor brauercharacteroftensor}, 
\begin{align*}
\tau^{G\times K}_{\Delta(P',\gamma', Q'), (r',u')}\big(\widetilde{F^G_{P,r}}\otimes_{kG} F^{G\times K}_{\Delta(P,\gamma, Q), (r,u)}\otimes_{kK} \widetilde{F^K_{Q,u}}\big)
\end{align*}
is equal to
\begin{align*}
\frac{1}{|G|}\!\!\!\!\!\sum_{\substack{\rule{0ex}{2ex}(\alpha,V,\beta)\in \Gamma_G(P',\gamma',Q'),\,g\in G_{p'}\\\rule{0ex}{1.6ex}(r',g)\in N_{G\times G}(\Delta(P',\alpha,V))\\\rule{0ex}{1.6ex}(g,u')\in N_{G\times K}(\Delta(V,\beta,Q'))}}\!\!\!\!\!\tau_{\Delta(P',\alpha,V),(r',g)}^{G\times G}(\widetilde{F^G_{P,r}})\,\tau_{\Delta(V,\beta,Q'),(g,u')}^{G\times K}(F^{G\times K}_{\Delta(P,\gamma, Q), (r,u)}\otimes_{kK} \widetilde{F^K_{Q,u}})\,.
\end{align*}
By Remark \ref{rem conjugatepairs}, the evaluation $\tau_{\Delta(P',\alpha,V),(r',g)}^{G\times G}(\widetilde{F^G_{P,r}})$ is non-zero if and only if $\big(\Delta(P',\alpha,V),(r',g)\big)=_{G\times G} \big(\Delta P, (r,r)\big)$ if and only if there exists $(g_1, g_2)\in G\times G$ such that $\big(\Delta(\lexp{g_1}P', i_{g_1}\alpha i_{g_2}^{-1}, \lexp{g_2}V),(\lexp{g_1}r',\lexp{g_2}g)\big)= \big(\Delta P, (r,r)\big)$. Hence the above sum can be written as
\begin{align*}
\frac{1}{n|G|}\!\!\!\!\!\sum_{\substack{\rule{0ex}{2ex}(g_1, g_2)\in G\times G, \\\rule{0ex}{1.6ex}\lexp{g_1}r'=r \\\rule{0ex}{1.6ex}(\lexp{g_2^{-1}}{r}, u')\in N_{G\times K}(\Delta(V,\beta,Q'))}}\!\!\!\!\! |C_G(P\langle r\rangle)| \,\tau_{\Delta(\lexp{g_2^{-1}}{P}, i_{g_2}^{-1} i_{g_1}\gamma',Q'),(\lexp{g_2^{-1}}r, u')}^{G\times K}(F^{G\times K}_{\Delta(P,\gamma, Q), (r,u)}\otimes_{kK} \widetilde{F^K_{Q,u}})\,,
\end{align*}
where $n$ is the number of pairs $(g_1,g_2)$ which satisfy the conditions above when $\alpha, V,\beta$ and $g$ are fixed. 
Now again by Corollary \ref{cor brauercharacteroftensor}, the element
\begin{align*}
\tau_{\Delta(\lexp{g_2^{-1}}{P}, i_{g_2}^{-1} i_{g_1}\gamma',Q'),(\lexp{g_2^{-1}}r, u')}^{G\times K}(F^{G\times K}_{\Delta(P,\gamma, Q), (r,u)}\otimes_{kK} \widetilde{F^K_{Q,u}})
\end{align*}
is equal to
\begin{align*}
\frac{1}{|K|}\!\!\!\!\!\sum_{\substack{\rule{0ex}{2ex}(\phi,Y,\psi)\in \Gamma_K(\lexp{g_2^{-1}}{P},  i_{g_2}^{-1} i_{g_1}\gamma', Q') \\\rule{0ex}{1.6ex}k\in K_{p'}\\\rule{0ex}{1.6ex}(\lexp{g_2^{-1}}r,k)\in N_{G\times K}(\Delta(\lexp{g_2^{-1}}{P},\phi,Y))\\\rule{0ex}{1.6ex}(k,u')\in N_{K\times K}(\Delta(Y,\psi,Q'))}}\!\!\!\!\!\tau_{\Delta(\lexp{g_2^{-1}}{P},\phi,Y),(\lexp{g_2^{-1}}r,k)}^{G\times K}(F^{G\times K}_{\Delta(P,\gamma, Q), (r,u)})\,\tau_{\Delta(Y,\psi,Q'),(k,u')}^{K\times K}(\widetilde{F^K_{Q,u}})\,.
\end{align*}
Therefore, the element
\begin{align*}
\tau_{\Delta(\lexp{g_2^{-1}}{P}, i_{g_2}^{-1} i_{g_1}\gamma',Q'),(\lexp{g_2^{-1}}r, u')}^{G\times K}(F^{G\times K}_{\Delta(P,\gamma, Q), (r,u)}\otimes_{kK} \widetilde{F^K_{Q,u}})
\end{align*}
is non-zero if and only if there exists $(\phi,Y,\psi)\in \Gamma_K(\lexp{g_2^{-1}}{P}, i_{g_2}^{-1} i_{g_1}\gamma',Q')$ such that 
\begin{align}\label{eqn condition1}
\big(\Delta(\lexp{g_2^{-1}}{P},\phi,Y),(\lexp{g_2^{-1}}r,k)\big)=_{G\times K} \big(\Delta(P,\gamma, Q),(r,u)\big) 
\end{align}
and such that
\begin{align*}
\big(\Delta(Y,\psi, Q'), (k,u')\big)=_{K\times K} \big(\Delta Q, (u,u)\big)\,.
\end{align*}
The latter implies that there exists $(k_1, k_2)\in K\times K$ such that
\begin{align*}
\big(\Delta(\lexp{k_1}Y,i_{k_1}\psi i_{k_2}^{-1}, \lexp{k_2}Q'), (\lexp{k_1}k,\lexp{k_2}u')\big)= \big(\Delta Q, (u,u)\big)\,.
\end{align*}
This, in particular, implies that $\psi=i_{k_1}^{-1}i_{k_2}$ and hence that $\phi=i_{g_2}^{-1}i_{g_1}\gamma' i_{k_2}^{-1}i_{k_1}$. Therefore the statement (\ref{eqn condition1}) is equivalent to
\begin{align*}
\big(\Delta(\lexp{g_2^{-1}g_1}{P'},i_{k_1}^{-1}i_{k_2}\gamma' i_{g_1}^{-1}i_{g_2}, \lexp{k_1^{-1}k_2}Q'),(\lexp{g_2^{-1}g_1}r',\lexp{k_1^{-1}k_2}u')\big)=_{G\times K} \big(\Delta(P,\gamma, Q),(r,u)\big) 
\end{align*}
which is equivalent to 
\begin{align*}
\lexp{(g_2^{-1}g_1, k_1^{-1}k_2)}{\big(\Delta(P',\gamma',Q'), (r',u')\big)}=_{G\times K} \big(\Delta(P,\gamma, Q),(r,u)\big)\,.
\end{align*}
This is clearly equivalent to 
\begin{align*}
\big(\Delta(P',\gamma',Q'), (r',u')\big)=_{G\times K} \big(\Delta(P,\gamma, Q),(r,u)\big)\,.
\end{align*}
This shows that the element $\widetilde{F^G_{P,r}}\otimes_{kG} F^{G\times K}_{\Delta(P,\gamma, Q), (r,u)}\otimes_{kK} \widetilde{F^K_{Q,u}}$ is a non-zero scalar multiple of the idempotent $F^{G\times K}_{\Delta(P,\gamma, Q), (r,u)}$, i.e.,
\begin{align*}
\widetilde{F^G_{P,r}}\otimes_{kG} F^{G\times K}_{\Delta(P,\gamma, Q), (r,u)}\otimes_{kK} \widetilde{F^K_{Q,u}}= \lambda\cdot F^{G\times K}_{\Delta(P,\gamma, Q), (r,u)}
\end{align*}
for some non-zero $\lambda\in\FF$. But multiplying by $\widetilde{F^K_{Q,u}}$ from the right, and by $\widetilde{F^G_{P,r}}$ from the left implies that $\lambda^2=\lambda$. Hence $\lambda=1$ and the result follows.
\end{proof}

The following result will be used in Section \ref{sec semisimplicity}. 
\begin{theorem}\label{thm keypoint}
Let $(P,r)$ be a pair of $G$ and $(Q,u)$ a pair of $K$. Then there exists a $p$-permutation $(kG, kK)$-bimodule $M$ all of whose indecomposable direct summands have twisted diagonal vertices, with the property that 
\begin{align*}
\widetilde{F^G_{P,r}}\otimes_{kG} M\otimes_{kK} \widetilde{F^K_{Q,u}} \neq 0
\end{align*}
if and only if  $(\tilde{P},\tilde{r})\cong (\tilde{Q},\tilde{u})$. 
\end{theorem}
\begin{proof}
Suppose that $M$ is a diagonal $p$-permutation $(kG, kK)$-bimodule with the property that 
\begin{align*}
\widetilde{F^G_{P,r}}\otimes_{kG} M\otimes_{kK} \widetilde{F^K_{Q,u}} \neq 0\,.
\end{align*}
Then there exists a pair $\left(\Delta(U,\gamma, W), (s,t)\right)$ of $G\times K$ such that 
\begin{align*}
\tau^{G\times K}_{\Delta(U,\gamma, W), (s,t)} (\widetilde{F^G_{P,r}}\otimes_{kG} M\otimes_{kK} \widetilde{F^K_{Q,u}})\neq 0\,.
\end{align*}
By Corollary \ref{cor brauercharacteroftensor}, there exists $(\alpha, V,\beta)\in\Gamma_G(U,\gamma,W)$ and $g\in G$ such that 
\begin{align*}
\tau^{G\times G}_{\Delta(U,\alpha, V), (s,g)}(\widetilde{F^G_{P,r}})\neq 0 \quad \text{and} \quad \tau^{G\times K}_{\Delta(V,\beta, W), (g,t)}(M\otimes_{kK} \widetilde{F^K_{Q,u}})\neq 0\,.
\end{align*}
Similarly, there exists $(\phi, Y,\psi)\in\Gamma_K(V,\beta, W)$ and $l\in K$ such that
\begin{align*}
\tau^{G\times K}_{\Delta(V,\phi, Y), (g,l)}(M)\neq 0 \quad \text{and} \quad \tau^{K\times K}_{\Delta(Y,\psi, W), (l,t)}(\widetilde{F^K_{Q,u}})\neq 0\,.
\end{align*}
Since $\tau^{G\times G}_{\Delta(U,\alpha, V), (s,g)}(\widetilde{F^G_{P,r}})\neq 0$, Remark \ref{rem conjugatepairs} implies that the pair $(P,r)$ is $G$-conjugate to the pair $(V,g)$. Similarly $\tau^{K\times K}_{\Delta(Y,\psi, W), (l,t)}(\widetilde{F^K_{Q,u}})\neq 0$ implies that the pair $(Q,u)$ is $K$-conjugate to the pair $(Y,l)$. Moreover since $\left(\Delta(V,\phi, Y), (g,l)\right)$ is a pair of $G\times K$, Lemma \ref{lem pairsofdirectproduct} implies that $(\tilde{V},\tilde{g})\cong (\tilde{Y},\tilde{l})$. Hence we also have $(\tilde{P},\tilde{r})\cong (\tilde{Q},\tilde{u})$, as desired. 

Now assume that $(\tilde{P},\tilde{r})\cong (\tilde{Q},\tilde{u})$. Then by Lemma \ref{lem pairsofdirectproduct}, there exists a group isomorphism $\gamma:Q\to P$ such that $\left(\Delta(P,\gamma, Q), (r,u)\right)$ is a pair of $G\times K$. By Lemma \ref{lem scalarmultiple}, $\widetilde{F^G_{P,r}}\otimes_{kG} F^{G\times K}_{\Delta(P,\gamma, Q), (r,u)}\otimes_{kK} \widetilde{F^K_{Q,u}}$ is nonzero and we are done. 
\end{proof}

\section{Some Clifford theory}\label{sec CliffordTheory}
The results in this section will be used in Section \ref{sec Blocksasfunctors}.  Let $k$ be an algebraically closed field of positive characteristic $p$. Let also $N$ be a finite group with a normal subgroup $C$ such that $N/C$ is a cyclic $p'$-group $\langle u\rangle$. 

\begin{theorem}\label{thm cliffordthm}
{\rm (i)} Any $N$-invariant projective indecomposable $kC$-module $F$ can be extended to a projective indecomposable $kN$-module $E$. Moreover, $\Ind_C^N F\cong \bigoplus_{\lambda:\langle u\rangle \to k^\times} E_\lambda$
where $E_\lambda= E\otimes_k \Inf^N_{N/C} k_\lambda$.

\smallskip
{\rm (ii)} If $M$ is a projective indecomposable $kN$-module such that $\Res^N_C M$ admits an $N$-invariant indecomposable summand $F$, then $F=\Res^N_C M$, that is, $\Res^N_C M$ is indecomposable. 
\end{theorem}
\begin{proof}
{\rm (i)} Let $F$ be an $N$-invariant indecomposable projective $kC$-module and set $S=F/J(F)$.  By \cite[Theorem~2.14]{Feit1982}, we can extend $S$ to a simple $kN$-module $\widehat{S}$, with projective cover $E$.  For a group isomorphism $\lambda:\langle u\rangle \to k^\times$, we set $\widehat{S}_\lambda = \widehat{S} \otimes_k \Inf^N_{N/C}k_\lambda$.  Then $\widehat{S}_\lambda$ is a simple $kN$-module with projective cover $E_\lambda= E\otimes_k \Inf^N_{N/C} k_\lambda$. 

The simple $kN$-modules $\widehat{S}_\lambda$ are all distinct. Equivalently, if $\lambda:\langle u\rangle\to k^\times$ is a non-trivial character, then $\widehat{S}$ and $\widehat{S}_\lambda$ are not isomorphic. Indeed, if $\varphi: \widehat{S}\to \widehat{S}_\lambda$ is an isomorphism of $kN$-modules, then the restriction of $\varphi$ to $kC$ is an automorphism of $S$ as a $kC$-module, hence it is a scalar multiple of the identity of $S$, by Schur's lemma.  So there exists $\mu\in k^\times$ such that $\varphi(v)=\mu v$ for any $v\in S$.  Now since $\varphi: \widehat{S}\to \widehat{S}_\lambda$ is a morphism of $kN$-modules, for any $x\in N$ and $v\in S$, we have
\begin{align*}
\varphi(x\cdot v)= \lambda(\overline{x})x\cdot \varphi(v)\,,
\end{align*}
where $\overline{x}$ is the image of $x$ in $\langle u\rangle$. So we get that $\mu x\cdot v=\lambda(\overline{x})\mu x\cdot v$, hence $\lambda(\overline{x})=1$, so $\lambda=1$. 

Now $\Hom_{kN}(\Ind_C^N F, \widehat{S}_\lambda) \cong \Hom_{kC}(F, S)$ is one dimensional.  Hence there is a non-zero morphism of $kN$-modules $\psi: \Ind^N_C F\to \widehat{S}_\lambda$, which is surjective since $\widehat{S}_\lambda$ is simple. Since $\Ind^N_C F$ is projective, it follows that $\psi$ can be lifted to $\theta: \Ind^N_C F\to E_\lambda$, and $\theta$ is surjective, because $E_\lambda$ is a projective cover of $\widehat{S}_\lambda$. Hence $\theta$ splits, and it follows that $E_\lambda$ is a direct summand of $\Ind^N_C F$. 

Since the modules $\widehat{S}_\lambda$ are all distinct, the modules $E_\lambda$ are all distinct as well, and it follows that $\oplus_\lambda E_\lambda$ is a direct summand of $\Ind^N_C F$.  Since $\langle u\rangle$ is a $p'$-group, this implies in particular that
\begin{align*}
|u| \dim_k E = \dim_k (\oplus_\lambda E_\lambda) \le |N:C| \dim_k F = |u| \dim_k F\,,
\end{align*}
so $\dim_k E \le \dim_k F$. 

But on the other hand, the surjection $E\to \widehat{S}$ restrict to a surjection of $kC$-modules $\Res^N_C E\to\nolinebreak S$. Since $\Res^N_C E$ is projective, we get as above that the projective cover $F$ of $S$ is a direct summand of $\Res^N_C E$. In particular, $\dim_k F\le \dim_k E$. It follows that $\dim_k E =\dim_k F$ and this proves {\rm (i)}.

\smallskip
{\rm (ii)} Let $M$ be a projective indecomposable $kN$-module such that $\Res^N_C M$ admits an $N$-invariant indecomposable summand $F$. Then as above, the simple $kC$-module $S=F/J(F)$ can be extended to a simple $kN$-module $\widehat{S}$ with projective cover $E$, and the simple $kN$-modules $\widehat{S}_\lambda$ are all distinct. Now
\begin{align*}
\Hom_{kN}(\Ind^N_C S, \widehat{S}_\lambda)\cong \Hom_{kC}(S,S)
\end{align*}
is one dimensional, so there is a non-zero morphism $\Ind^N_C S\to\widehat{S}_\lambda$ which is surjective since $\widehat{S}_\lambda$ is simple. It follows that we have a surjective morphism of $kN$-modules
\begin{align*}
\sigma: \Ind^N_C S\to\bigoplus_\lambda \widehat{S}_\lambda\,.
\end{align*}
But these two modules have the same dimension $|u|\dim_k S$, so $\sigma$ is an isomorphism. In particular $\Ind^N_C S$ is semisimple. 

Since $F$ is a direct summand of $\Res^N_C M$, we get a non-zero morphism $\Res^N_C M\to S$, hence a non-zero morphism $M\to \Ind^N_C S$. The image $L$ of this morphism is a semisimple quotient of~$M$, which is projective and indecomposable. Hence $L$ is simple, and isomorphic to one of the modules~$\widehat{S}_\lambda$.  Then $M\cong E_\lambda$ and $F=\Res^N_C M$. This proves {\rm (ii)}. 
\end{proof}

\section{An equivalence of categories}\label{sec equivalenceofcats}
Let $\calA$ be an abelian category with arbitrary direct sums. Recall that an object $P$ of $\calA$ is called {\it compact} if for any family $(X_i)_{i\in I}$ of objects of $\calA$, the natural morphism
\begin{align*}
\bigoplus_{i\in I} \Hom_\calA (P, X_i) \to \Hom_\calA(P, \bigoplus_{i\in I} X_i)
\end{align*}
is an isomorphism. The following is well-known to specialists. We include the proof for the convenience of the reader. 

\begin{theorem}\label{thm equivalenceofcats}
Let $R$ be a commutative ring and $\calA$ an $R$-linear abelian category with arbitrary direct sums. Let moreover $\calP$ be a set of objects of $\calA$ with the following properties:

\smallskip
{\rm (i)} If $P\in \calP$, then $P$ is projective and compact in $\calA$.

\smallskip
{\rm (ii)} The set $\calP$ generates $\calA$, i.e., for any object $X$ of $\calA$, there exists a family $(P_j)_{j\in J}$ of elements of $\calP$, and an epimorphism $\oplus_{j\in J} P_j \twoheadrightarrow X$ in $\calA$. 

Let $\Fun_R(\calP^{\mathrm{op}}, \lMod{R})$ be the category of $R$-linear contravariant functors from the full subcategory $\calP$ of $\calA$ to the category of $R$-modules.  Then the functor
\begin{align*}
\calH: X\in \calA \mapsto H_X = \Hom_\calA(-, X) \in \Fun_R(\calP^{\mathrm{op}}, \lMod{R})
\end{align*}
is an equivalence of $R$-linear abelian categories. 
\end{theorem}
\begin{proof}
We first show that $\calH$ is fully faithful. Let $X$ be an object of $\calA$. By Condition {\rm (ii)}, there exists an exact sequence in $\calA$ of the form
\begin{align}\label{eqn firstsequence}
\oplus_{j\in J} Q_j \longrightarrow \oplus_{i\in I} P_i \longrightarrow X \longrightarrow 0 \,,
\end{align}
where $(Q_j)_{j\in J}$ and $(P_i)_{i\in I}$ are families of elements of $\calP$. Now if $P\in\calP$, then applying the functor $\Hom_\calA(P, -)$ gives an exact sequence of $R$-modules
\begin{align*}
\Hom_\calA(P, \oplus_{j\in J} Q_j) \longrightarrow \Hom_\calA(P,\oplus_{i\in I} P_i) \longrightarrow \Hom_\calA(P,X) \longrightarrow 0\,,
\end{align*}
since $P$ is projective in $\calA$. Since moreover $P$ is compact in $\calA$, this exact sequence is naturally isomorphic to the exact sequence 
\begin{align*}
\oplus_{j\in J} \Hom_\calA(P, Q_j) \longrightarrow \oplus_{i\in I} \Hom_\calA(P,P_i) \longrightarrow \Hom_\calA(P,X) \longrightarrow 0\,.
\end{align*}
In other words, we get an exact sequence in the category $\calF=\Fun_R(\calP^{\mathrm{op}}, \lMod{R})$
\begin{align*}
\oplus_{j\in J} H_{Q_j} \longrightarrow \oplus_{i\in I} H_{P_i} \longrightarrow H_X \longrightarrow 0\,.
\end{align*}
Now for any object $Y$ of $\calA$, applying the functor $\Hom_\calF(-, H_Y)$ to this sequence gives the exact sequence
\begin{align*}
0\longrightarrow \Hom_\calF(H_X,H_Y) \longrightarrow \Hom_\calF(\oplus_{i\in I} H_{P_i}, H_Y)\longrightarrow \Hom_\calF(\oplus_{j\in J} H_{Q_j}, H_Y)
\end{align*}
of $R$-modules.  This sequence is naturally isomorphic to 
\begin{align*}
0\longrightarrow \Hom_\calF(H_X,H_Y) \longrightarrow \prod_{i\in I} \Hom_\calF( H_{P_i}, H_Y)\longrightarrow \prod_{j\in J}  \Hom_\calF(H_{Q_j}, H_Y)\,.
\end{align*}
Now by the Yoneda lemma, for each $P\in \calP$, we get a natural isomorphism $\Hom_\calF(H_P, H_Y)\cong \Hom_\calA(P,Y)$, so the previous sequence is isomorphic to 
\begin{align*}
0\longrightarrow \Hom_\calF(H_X,H_Y) \longrightarrow \prod_{i\in I} \Hom_\calA( P_i, H_Y)\longrightarrow \prod_{j\in J}  \Hom_\calA(Q_j, H_Y)\,,
\end{align*}
or in other words to the sequence
\begin{align}\label{eqn secondsequence}
0\longrightarrow \Hom_\calF(H_X,H_Y) \longrightarrow  \Hom_\calA(\oplus_{i\in I} P_i, H_Y)\longrightarrow  \Hom_\calA(\oplus_{j\in J} Q_j, H_Y)\,.
\end{align}
Now applying the functor $\Hom_\calA(-,Y)$ to the exact sequence (\ref{eqn firstsequence}) gives the exact sequence
\begin{align}\label{eqn thirdsequence}
0\longrightarrow \Hom_\calA(X,Y) \longrightarrow  \Hom_\calA(\oplus_{i\in I} P_i, H_Y)\longrightarrow  \Hom_\calA(\oplus_{j\in J} Q_j, H_Y)\,.
\end{align}
Moreover, the exact sequences (\ref{eqn secondsequence}) and (\ref{eqn thirdsequence}) fit into a commutative diagram
\begin{center}
\begin{tikzcd}
0 \arrow[rightarrow]{r} 
  & \Hom_\calA(X,Y) \arrow[rightarrow]{r}\arrow[rightarrow]{d}
  & \Hom_\calA(\oplus_{i\in I} P_i, H_Y)\arrow{r} \arrow{d}{=} 
  & \Hom_\calA(\oplus_{j\in J} Q_j, H_Y) \arrow{d}{=} \\
  
0 \arrow[rightarrow]{r} 
  & \Hom_\calF(H_X,H_Y) \arrow[rightarrow]{r}
  & \Hom_\calA(\oplus_{i\in I} P_i, H_Y) \arrow{r}
  & \Hom_\calA(\oplus_{j\in J} Q_j, H_Y) 

\end{tikzcd}
\end{center}
of $R$-modules.  It follows that the vertical arrow $\Hom_\calA(X,Y) \longrightarrow \Hom_\calF(H_X,H_Y)$ is an isomorphism for any objects $X$ and $Y$ of $\calA$. In other words, the functor $\calH$ is fully faithful as was to be shown. 

We now prove that $\calH$ is essentially surjective. Let $F$ be any object of $\calF=\Fun_R(\calP^{\mathrm{op}}, \lMod{R})$. For every $P\in \calP$ we choose a generating set $s_P$ of $F(P)$ as an $R$-module. By the Yoneda lemma again, we have an isomorphism $\Hom_\calF(H_P,F)\cong F(P)$ from which we get a morphism
\begin{align}\label{eqn fourthsequence}
\oplus_{P\in\calP}\oplus_{s\in s_P} H_P \longrightarrow F\,,
\end{align}
which is an epimorphism as $s_P$ generates $F(P)$ for any $P\in\calP$. Now since $\calP$ consists of compact objects of $\calA$, we have an isomorphism
\begin{align*}
\oplus_{P\in\calP}\oplus_{s\in s_P} \Hom_\calA(-, P) \cong \Hom_\calA(-, \oplus_{P\in\calP}\oplus_{s\in s_P} P)
\end{align*}
in $\calF$, in other words an isomorphism $\oplus_{P\in\calP}\oplus_{s\in s_P} H_P\cong H_X$ where $X=\oplus_{P\in\calP}\oplus_{s\in s_P} P$.

So for any $F\in\calF$, there exists $X\in\calA$ and an epimorphism $H_X \twoheadrightarrow F$ in $\calF$ as in (\ref{eqn fourthsequence}). Then there exists $Y\in\calA$ and an exact sequence
\begin{align*}
H_Y \longrightarrow H_X \longrightarrow F\longrightarrow 0
\end{align*}
in $\calF$. Since the functor $\calH$ is fully faithful by the first part of the proof, the left arrow $H_Y \to H_X$ in this sequence is induced by a morphism $f: Y\to X$ in $\calA$. Let $Z$ be the cokernel of $f$ in $\calA$. The exact sequence 

\begin{center}
\begin{tikzcd}
Y \arrow{r}{f} 
& X \arrow{r} & Z\arrow{r} &0
\end{tikzcd}
\end{center}
in $\calA$ gives an exact sequence $H_Y \longrightarrow H_X \longrightarrow H_Z\longrightarrow 0$ in $\calF$, because $\calP$ consists of projective objects of $\calA$. It follows that $H_Z$ is isomorphic to the cokernel of $H_Y \longrightarrow H_X$ in $\calF$, that is $H_Z\cong F$. It follows that the functor $\calH$ is essentially surjective, which completes the proof of the theorem.  
\end{proof}

\section{Semisimplicity of the functor category}\label{sec semisimplicity}
Recall that $k$ denotes an algebraically closed field of positive characteristic $p$ and $\FF$ denotes an algebraically closed field of characteristic zero. We recall from the introduction the definition of diagonal $p$-permutation functors over $\FF$. 
\begin{definition}
Let $\FF pp_k^\Delta$ be the category with
\begin{itemize}
\item objects: finite groups
\item $\mathrm{Mor}_{\FF pp_k^\Delta}(G,H) = \FF\otimes_{\ZZ}T^{\Delta}(H,G)=\FFT^{\Delta}(H,G)$.
\end{itemize} 
\end{definition}
An $\FF$-linear functor from $\FF pp_k^\Delta$ to $\lMod{\FF}$ is called a \textit{diagonal $p$-permutation functor} over $\FF$. Diagonal $p$-permutation functors form an abelian category $\Fppk{\FF}$.

Let $G$ be a finite group and let $(P,r)$ be a pair of $G$. Then the element $\widetilde{F^G_{P,r}}$ is an idempotent of $\FFTD(G,G)$ which is the endomorphism algebra of the representable functor $\FFTD_G$. Let $\FFTD_G\widetilde{F^G_{P,r}}$ denote the corresponding direct summand of the functor $\FFTD_G$. For any finite group $H$, the evaluation $\FFTD_G\widetilde{F^G_{P,r}}(H)$ is given by
\begin{align*}
\FFTD_G\widetilde{F^G_{P,r}}(H)=\FFTD(H,G)\widetilde{F^G_{P,r}}:=\{X\otimes_{kG} \widetilde{F^G_{P,r}}: X\in \FFTD(H,G)\}\,.
\end{align*}

\begin{lemma}\label{lem naturaltransformations}
Let $G$ and $K$ be finite groups. Let $(P,r)$ be a pair of $G$ and $(Q,u)$ a pair of~$K$. Let also $F\in \Fppk{\FF}$ be a diagonal $p$-permutation functor over $\FF$.

\smallskip
{\rm (i)} We have $\Hom_{\Fppk{\FF}}(\FFTD_G\widetilde{F^G_{P,r}}, F)\equiv F(G)^{\widetilde{F^G_{P,r}}}:=\{m\in F(G): F(\widetilde{F^G_{P,r}})(m)=m\}$.

\smallskip 
{\rm (ii)} We have $\Hom_{\Fppk{\FF}}(\FFTD_G\widetilde{F^G_{P,r}}, \FFTD_K)\equiv \widetilde{F^G_{P,r}}\FFTD(G,K)$.

\smallskip
{\rm (iii)} We have $\Hom_{\Fppk{\FF}}(\FFTD_G\widetilde{F^G_{P,r}}, \FFTD_K\widetilde{F^K_{Q,u}})\equiv \widetilde{F^G_{P,r}}\FFTD(G,K)\widetilde{F^K_{Q,u}}$.
\end{lemma}
\begin{proof}
Let $\Phi\in \Hom_{\Fppk{\FF}}(\FFTD_G\widetilde{F^G_{P,r}}, F)$ be a natural transformation. One shows that $\Phi$ is completely determined by $\Phi_G(\widetilde{F^G_{P,r}})=:m\in F(G)$. Moreover, we have $F(\widetilde{F^G_{P,r}})(m)=m$ and hence $m\in F(G)^{\widetilde{F^G_{P,r}}}$. One can also show that every $m\in F(G)^{\widetilde{F^G_{P,r}}}$ determines a natural transformation. This proves the statement (i).

Statements (ii) and (iii) are proved similarly.
\end{proof}

\begin{lemma}\label{lem projectiveandcompact}
Let $(P,r)$ be a pair of $G$.  The functor $\FFTD_G\widetilde{F^G_{P,r}}$ is compact and projective in $\Fppk{\FF}$. 
\end{lemma}
\begin{proof}
The functor $\FFTD_G\widetilde{F^G_{P,r}}$ is a direct summand of the projective functor $\FFTD_G$. Hence it is projective as well. Let $\{F_i\}$ be a set of objects of $\Fppk{\FF}$. By Lemma \ref{lem naturaltransformations}(i) there is an isomorphism
\begin{align*}
\bigoplus_{i} \Hom(\FFTD_G\widetilde{F^G_{P,r}}, F_i)\to \Hom\big(\FFTD_G\widetilde{F^G_{P,r}}, \bigoplus_i F_i\big)
\end{align*}
in $\Fppk{\FF}$. Hence $\FFTD_G\widetilde{F^G_{P,r}}$ is compact as well. 
\end{proof}

\begin{proposition}\label{prop generators}
Every functor in $\Fppk{\FF}$ is a quotient of a direct sum of representable functors of the form $\FFTD_{Q\langle v\rangle}\widetilde{F^{Q\langle v\rangle}_{Q,v}}$ where $(Q,v)$ is a $\DD$-pair.  
\end{proposition}
\begin{proof}
Since the representable functors generate the functor category $\Fppk{\FF}$, it suffices to prove the statement for the representable functors. Let $G$ be a finite group and consider the representable functor $\FFTD_G$. 

If the essential algebra $\FFTD(G,G)/ \left(\sum_{|H|<|G|} \FFTD(G,H)\circ \FFTD(H,G) \right)$ at $G$ vanishes, then the identity morphism of the functor $\FFTD_G$ can be written as a sum of morphisms $a_i \circ b_i$, where $b_i: G\to K_i$, $a_i: K_i\to G$ and $|K_i|< |G|$. Equivalently, the functor $\FFTD_G$ is a quotient of the direct sum $\oplus_i \FFTD_{K_i}$. By induction on the order of $G$, we can assume that the essential algebra at $G$ does not vanish, i.e., $G=L\langle u\rangle$ for some $\DD$-pair $(L,u)$ (see \cite[Theorem~3.3]{BoucYilmaz2020}).

Now the identity map of $\FFTD_G$ is equal to the sum of the idempotents $\widetilde{F^G_{P,r}}$ where $(P,r)$ runs in a set of conjugacy classes of pairs of $G$. If $P\langle r\rangle \neq G$, then the idempotent $F^G_{P,r}$ is induced from a proper subgroup and hence $\widetilde{F^G_{P,r}}$ factors through this subgroup. If $P\langle r\rangle=G$, then $P=L$ and $\langle r\rangle$ is conjugate to $\langle u\rangle$. In particular, $(P,r)$ is a $\DD$-pair. The result follows by induction on $|G|$. 
\end{proof}

\begin{notation}
Let $\calP$ denote the full subcategory of $\Fppk{\FF}$ consisting of the functors $\FFTD_{L\langle u\rangle}\widetilde{F^{L\langle u\rangle}_{L, u}}$, where $(L, u)$ runs through a set of isomorphism classes of $\DD$-pairs. 
\end{notation}

\begin{corollary}\label{cor equivalenceofcats}
The functor from $\Fppk{\FF}$ to $\Fun_{\FF}(\calP^{op}, \lMod{\FF})$ sending a functor $X$ to the representable functor $\Hom_{\Fppk{\FF}}(-,X)$ is an equivalence of categories. 
\end{corollary}
\begin{proof}
This follows from Lemma \ref{lem projectiveandcompact}, Proposition \ref{prop generators} and Theorem \ref{thm equivalenceofcats}.  
\end{proof}

Let $(L,u)$ and $(M,v)$ be $\DD$-pairs. By Lemma \ref{lem naturaltransformations} we have 
\begin{align*}
\Hom_{\calP}\big(\FFTD_{L\langle u\rangle}\widetilde{F^{L\langle u\rangle}_{L, u}}, \FFTD_{M\langle v\rangle}\widetilde{F^{M\langle v\rangle}_{M, v}}\big)\equiv \widetilde{F^{L\langle u\rangle}_{L, u}}\FFTD(L\langle u\rangle, M\langle v\rangle)\widetilde{F^{M\langle v\rangle}_{M, v}}\,.
\end{align*}
Therefore the category $\calP^{op}$ is isomorphic to the following category.
\begin{definition}
Let $\calD^{\Delta}$ denote the following category:
\begin{itemize}
\item objects are the isomorphism classes of $\DD$-pairs.
\item $\Hom_{\calD^{\Delta}}\big((L,u), (M,v)\big)=\widetilde{F^{M\langle v\rangle}_{M, v}}\FFTD(L\langle u\rangle, M\langle v\rangle)\widetilde{F^{L\langle u\rangle}_{L, u}}$.
\end{itemize}
\end{definition}
It follows from Theorem \ref{thm keypoint} that $\widetilde{F^{L\langle u\rangle}_{L, u}}\FFTD(L\langle u\rangle, M\langle v\rangle)\widetilde{F^{M\langle v\rangle}_{M, v}}$ is non-zero if and only if $(L,u)$ and $(M,v)$ are isomorphic. Hence our next aim is to understand the structure of $\widetilde{F^{L\langle u\rangle}_{L, u}}\FFTD(L\langle u\rangle, L\langle u\rangle)\widetilde{F^{L\langle u\rangle}_{L, u}}$. We start with some preliminary results.

\begin{notation}\label{OutLu}
{\rm (i)} Let $\Aut_u(L)$ denote the set of automorphisms $f$ of $L$ with the property that $f(\lexp{u}l)=\lexp{u}f(l)$.

\smallskip
{\rm (ii)} Let $\Inn(C_{L\langle u\rangle}(u))$ denote the normal subgroup of $\Aut_u(L)$ consisting of conjugations induced by the elements of $C_{L\langle u\rangle}(u)$.

\smallskip
{\rm (iii)} Let $\Aut(L,u)$ denote the set of automorphisms of $L\langle u\rangle$ that sends $u$ to a conjugate of $u$. 

\smallskip
{\rm (iv)} Let $\Inn(L\langle u\rangle)$ denote the normal subgroup of $\Aut(L,u)$ consisting of conjugations induced by elements of $L\langle u\rangle$. 

\smallskip
{\rm (v)} Set $\Out(L,u):=\Aut(L,u)/\Inn(L\langle u\rangle)$.
\end{notation}

\begin{Remark}\label{rem remarkonpairs}
For any $f\in\Aut_u(L)$ the pair $\big(\Delta(L,f,L),(u,u)\big)$ is a pair of $L\langle u\rangle$. Moreover, for any $f_1,f_2\in \Aut_u(L)$ the pairs $\big(\Delta(L,f_1,L),(u,u)\big)$ and $\big(\Delta(L,f_2,L),(u,u)\big)$ are conjugate if and only if $f_1f_2^{-1}\in\Inn(C_{L\langle u\rangle}(u))$.
\end{Remark}

We extend any $f\in\Aut_u(L)$ to $f\in\Aut(L\langle u\rangle)$ by defining $f(lu^i):=f(l)u^i$. This induces an embedding $\Aut_u(L)\hookrightarrow \Aut(L,u)$. We identify $\Aut_u(L)$ by its image in $\Aut(L,u)$. 

\begin{lemma}
We have $\Out(L,u)\cong \Aut_u(L)/\Inn(C_{L\langle u\rangle}(u))$.
\end{lemma} 
\begin{proof}
Let $\phi\in \Aut(L,u)$ be an automorphism and let $g\in L\langle u\rangle$ be an element with the property that $\phi(u)=\lexp{g}u$. Set $f:=i_{g^{-1}}\circ \phi$. Then $f\in\Aut_u(L)$ and $\phi=i_g\circ f \in \Aut(L,u)$. This shows that $\Aut(L,u)=\Inn(L\langle u\rangle)\Aut_u(L)$. We also have $$\Inn(L\langle u\rangle)\cap \Aut_u(L)=\Inn(C_{L\langle u\rangle}(u))\,.$$
Hence the result follows from the second isomorphism theorem. 
\end{proof}

\begin{lemma}\label{lem conjugatesofcomposition}
Let $\left(\Delta(P,\pi,L),(s,u)\right)$ be a pair of $G\times L\langle u\rangle$, let $\gamma_1, \gamma_2 \in \Aut_u(L)$, let $x,y\in C_{L\langle u\rangle}(u)$ and let $z\in L\langle u\rangle$.

\smallskip
{\rm(i)}
The subgroups $\Delta(P,\pi\gamma_1,L)$ and $\Delta(P,\pi i_z\gamma_1,L)$ of $G\times L\langle u\rangle$ are conjugate.

\smallskip
{\rm (ii)} The pairs $\left(\Delta(P,\pi\gamma_1,L),(s,u)\right)$ and $\left(\Delta(P,\pi i_x\gamma_1,L),(s,u)\right)$ of $G\times L\langle u\rangle$ are conjugate.

\smallskip
{\rm (iii)} The pairs $\big(\Delta(L, \gamma_2 i_x \gamma_1, L), (u,u)\big)$ and $\big(\Delta(L, \gamma_2 i_y \gamma_1, L), (u,u)\big)$ of $L\langle u\rangle \times L\langle u\rangle$ are conjugate. In particular, the pairs $\big(\Delta(L, \gamma_2 i_x \gamma_1, L), (u,u)\big)$ and $\big(\Delta(L, \gamma_2 \gamma_1, L), (u,u)\big)$ are conjugate. 
\end{lemma}
\begin{proof}
Let $a:=\gamma_1^{-1}(z^{-1})$. Then we have
\begin{align*}
\lexp{(1,a)}{\Delta(P,\pi\gamma_1,L)}=\Delta(P,\pi i_z\gamma_1,L)
\end{align*}
which proves {\rm (i)}. The other parts are proved similarly.
\end{proof}

\begin{lemma}\label{lem tensorprodofidempotents}
Let $G$ be a finite group and $(L,u)$ a $\DD$-pair. Let $\big(\Delta(P, \phi, L), (s,u)\big)$ be a pair of $G\times L\langle u\rangle$ and let $\big(\Delta(L, \gamma, L), (u,u)\big)$ be a pair of $L\langle u\rangle\times L\langle u\rangle$. Then
\begin{align*}
F^{G\times L\langle u\rangle}_{\Delta(P,\phi, L),(s,u)}\otimes_{kL\langle u\rangle} F^{L\langle u\rangle\times L\langle u\rangle}_{\Delta(L,\gamma, L),(u,u)}=\frac{1}{|Z(L\langle u\rangle)|}F^{G\times L\langle u\rangle}_{\Delta(P,\phi\gamma, L),(s,u)} \quad \text{in} \quad \FFTD(G,L\langle u\rangle)\,.
\end{align*}
\end{lemma}
\begin{proof}
Using Corollary \ref{cor brauercharacteroftensor} one shows that if $\big(\Delta(R,\sigma,L), (t,u^i)\big)$ is a pair of $G\times L\langle u\rangle$ with the property that
\begin{align*}
\tau^{G\times L\langle u\rangle}_{\Delta(R, \sigma,L),(t,u^i)}\big(F^{G\times L\langle u\rangle}_{\Delta(P,\phi, L),(s,u)}\otimes_{kL\langle u\rangle} F^{L\langle u\rangle\times L\langle u\rangle}_{\Delta(L,\gamma, L),(u,u)}\big)\neq 0
\end{align*}
then $\big(\Delta(R,\sigma,L), (t,u^i)\big)$ is conjugate to a pair of the form $\big(\Delta(P,\phi i_x\gamma,L), (s,u)\big)$ where $x\in C_{L\langle u\rangle}(u)$. Lemma \ref{lem conjugatesofcomposition} implies that in this case that the pair  $\big(\Delta(R,\sigma,L), (t,u^i)\big)$ is conjugate to $\big(\Delta(P,\phi \gamma,L), (s,u)\big)$.

Now by Corollary \ref{cor brauercharacteroftensor} again, we have
\begin{align*}
&\tau^{G \times L\langle u\rangle}_{\Delta(P,\phi\gamma, L),(s,u)}\big(F^{G\times L\langle u\rangle}_{\Delta(P,\phi, L),(s,u)}\otimes_{kL\langle u\rangle} F^{L\langle u\rangle\times L\langle u\rangle}_{\Delta(L,\gamma, L),(u,u)}\big)\\&
=\!\frac{1}{|L\langle u\rangle|}\!\!\!\!\!\sum_{\substack{\rule{0ex}{2ex}(\alpha,L,\beta)\in \Gamma_{L\langle u\rangle}(P,\phi\gamma,L)\\\rule{0ex}{1.6ex} c\in \langle u\rangle\\\rule{0ex}{1.6ex}(s,c)\in N_{G\times L\langle u\rangle}(\Delta(P,\alpha,L))\\\rule{0ex}{1.6ex}(c,u)\in N_{L\langle u\rangle\times L\langle u\rangle}(\Delta(L,\beta,L))}}\!\!\!\!\!\tau_{\Delta(P,\alpha,L),(s,c)}^{G\times L\langle u\rangle}\big(F^{G\times L\langle u\rangle}_{\Delta(P,\phi, L),(s,u)}\big)\,\tau_{\Delta(L,\beta,L),(c,u)}^{L\langle u\rangle\times L\langle u\rangle}\big(F^{L\langle u\rangle\times L\langle u\rangle}_{\Delta(L,\gamma, L),(u,u)}\big)\,.
\end{align*}
The product
\begin{align*}
\tau_{\Delta(P,\alpha,L),(s,c)}^{G\times L\langle u\rangle}\big(F^{G\times L\langle u\rangle}_{\Delta(P,\phi, L),(s,u)}\big)\,\tau_{\Delta(L,\beta,L),(c,u)}^{L\langle u\rangle\times L\langle u\rangle}\big(F^{L\langle u\rangle\times L\langle u\rangle}_{\Delta(L,\gamma, L),(u,u)}\big)
\end{align*}
is non-zero if and only if there exist $g\in G$ and $l_1,l_2,l_3\in L\langle u\rangle$ such that
\begin{align*}
\big(\Delta(P,\alpha,L),(s,c)\big)=\lexp{(g,l_1)}{\big(\Delta(P,\phi, L),(s,u)\big)}=\big(\Delta(\lexp{g}P, i_{g}\phi i_{l_1}^{-1}, L), (\lexp{g}s,\lexp{l_1}u)\big)
\end{align*}
and
\begin{align*}
\big(\Delta(L,\beta,L),(c,u)\big)=\lexp{(l_2,l_3)}{\big(\Delta(L,\gamma, L),(u,u)\big)}=\big(\Delta(L, i_{l_2}\gamma i_{l_3}^{-1}, L), (\lexp{l_2}u,\lexp{l_3}u)\big)
\end{align*}
hold. These conditions imply that 
\begin{align*}
g\in N_G(P,s)\,, \quad l_3, l_2^{-1}l_1 \in C_{L\langle u\rangle}(u)\,, \quad \alpha=i_{g}\phi i_{l_1}^{-1} \quad \text{and}\quad \beta=i_{l_2}\gamma i_{l_3}^{-1}\,.
\end{align*}
Moreover, the condition that $\phi\gamma=\alpha\beta$ implies that
\begin{align*}
\phi\gamma=i_g\phi\gamma i_x
\end{align*}
as maps from $L$ to $P$, where $x=\gamma^{-1}(l_1^{-1}l_2)l_3^{-1}$. 

The number of quadruples $(g, l_1, l_2, l_3)$ that satisfy these conditions is
\begin{align*}
|C_G(P,s)||C_{L\langle u\rangle}(u)|^2 |L\langle u\rangle|\,,
\end{align*}
where $C_G(P,s):= N_G(P,s) \cap C_G(P)$.
However, when $(\alpha, L,\beta)\in \Gamma_{L\langle u\rangle}$ and $c\in\langle u\rangle$ are fixed there are $|C_{L\langle u\rangle}(u)|^2 |Z(L\langle u\rangle)| |C_G(P,s)|$ quadruples with these properties. Therefore, we have
\begin{align*}
\tau^{G \times L\langle u\rangle}_{\Delta(P,\phi\gamma, L),(s,u)}\big(F^{G\times L\langle u\rangle}_{\Delta(P,\phi, L),(s,u)}\otimes_{kL\langle u\rangle} F^{L\langle u\rangle\times L\langle u\rangle}_{\Delta(L,\gamma, L),(u,u)}\big)&=\frac{|C_G(P,s)||C_{L\langle u\rangle}(u)|^2 |L\langle u\rangle|}{|L\langle u\rangle||C_{L\langle u\rangle}(u)|^2 |Z(L\langle u\rangle)| |C_G(P,s)|}\\&
=\frac{1}{|Z(L\langle u\rangle)|} \,.
\end{align*}
This completes the proof.
\end{proof}

\begin{proposition}\label{prop algebragenerated}
Let $(L,u)$ be a $\DD$-pair. Then $\widetilde{F^{L\langle u\rangle}_{L, u}}\FFTD(L\langle u\rangle, L\langle u\rangle)\widetilde{F^{L\langle u\rangle}_{L, u}}$ is equal to the $\FF$-algebra generated by the elements of the form $F^{L\langle u\rangle\times L\langle u\rangle}_{\Delta(L,\gamma, L),(u,u)}$. In particular, the $\FF$-dimension of $\widetilde{F^{L\langle u\rangle}_{L, u}}\FFTD(L\langle u\rangle, L\langle u\rangle)\widetilde{F^{L\langle u\rangle}_{L, u}}$ is equal to the cardinality of $\Out(L,u)$.
\end{proposition}
\begin{proof}
Let $\big(\Delta(L,\gamma, L),(u^i,u^j)\big)$ be a pair of $L\langle u\rangle\times L\langle u\rangle$ such that
\begin{align*}
\widetilde{F^{L\langle u\rangle}_{L, u}} F^{L\langle u\rangle\times L\langle u\rangle}_{\Delta(L,\gamma, L),(u^i,u^j)}\widetilde{F^{L\langle u\rangle}_{L, u}} \neq 0\,.
\end{align*}
Then there exists a pair $\big(\Delta(L, \varphi,L),(s,t)\big)$ of $L\langle u\rangle\times L\langle u\rangle$ such that
\begin{align*}
\tau^{L\langle u\rangle\times L\langle u\rangle}_{\Delta(L, \varphi,L),(s,t)}\big(\widetilde{F^{L\langle u\rangle}_{L, u}} F^{L\langle u\rangle\times L\langle u\rangle}_{\Delta(L,\gamma, L),(u^i,u^j)}\widetilde{F^{L\langle u\rangle}_{L, u}}\big)\neq 0\,.
\end{align*}
By Corollary \ref{cor brauercharacteroftensor}, there exists $(\alpha, L,\beta)\in\Gamma_{L\langle u\rangle}(L,\varphi, L)$ and $a\in \langle u\rangle$ such that $(s,a)\in N_{L\langle u\rangle\times L\langle u\rangle}(\Delta(L,\alpha,L)$, $(a,t)\in N_{L\langle u\rangle\times L\langle u\rangle}(\Delta(L,\beta,L)$ and
\begin{align*}
\tau^{L\langle u\rangle\times L\langle u\rangle}_{\Delta(L, \alpha,L),(s,a)}\big(\widetilde{F^{L\langle u\rangle}_{L, u}}\big)\cdot \tau^{L\langle u\rangle\times L\langle u\rangle}_{\Delta(L, \beta,L),(a,t)}\big(F^{L\langle u\rangle\times L\langle u\rangle}_{\Delta(L,\gamma, L),(u^i,u^j)}\widetilde{F^{L\langle u\rangle}_{L, u}}\big)\neq 0\,.
\end{align*}
This implies, in particular, that 
\begin{align}\label{eqn first}
\big(\Delta(L, \alpha,L),(s,a)\big)=_{L\langle u\rangle\times L\langle u\rangle} \big(\Delta L, (u,u)\big)\,.
\end{align}
Moreover, applying Corollary \ref{cor brauercharacteroftensor} again, there exists $(\phi, L,\psi)\in\Gamma_{L\langle u\rangle}(L,\beta, L)$ and $b\in \langle u\rangle$ such that $(a,b)\in N_{L\langle u\rangle\times L\langle u\rangle}(\Delta(L,\phi,L)$, $(b,t)\in N_{L\langle u\rangle\times L\langle u\rangle}(\Delta(L,\psi,L)$ and 
\begin{align*}
\tau^{L\langle u\rangle\times L\langle u\rangle}_{\Delta(L, \phi,L),(a,b)}\big(F^{L\langle u\rangle\times L\langle u\rangle}_{\Delta(L,\gamma, L),(u^i,u^j)}\big)\cdot \tau^{L\langle u\rangle\times L\langle u\rangle}_{\Delta(L, \psi,L),(b,t)}\big(\widetilde{F^{L\langle u\rangle}_{L, u}}\big)\neq 0\,.
\end{align*}
Therefore
\begin{align}\label{eqn second}
\big(\Delta(L, \phi,L),(a,b)\big)=_{L\langle u\rangle\times L\langle u\rangle}\big(\Delta(L,\gamma, L),(u^i,u^j)\big)
\end{align}
and
\begin{align}\label{eqn third}
\big(\Delta(L, \psi,L),(b,t)\big)=_{L\langle u\rangle\times L\langle u\rangle} \big(\Delta L,(u,u)\big)\,. 
\end{align}
Now using the conditions (\ref{eqn first}), (\ref{eqn second}) and (\ref{eqn third}) one shows that
\begin{align*}
\big(\Delta(L,\varphi,L),(s,t)\big)=_{L\langle u\rangle\times L\langle u\rangle} \big(\Delta(L, \gamma, L), (u^i,u^j)\big)\,.
\end{align*}
This means that the algebra maps $\tau^{L\langle u\rangle\times L\langle u\rangle}_{\Delta(L, \varphi,L),(s,t)}$ and $\tau^{L\langle u\rangle\times L\langle u\rangle}_{\Delta(L, \gamma, L),(u^i,u^j)}$ are equal and hence we can replace the pair $\big(\Delta(L, \varphi,L),(s,t)\big)$ by the pair $\big(\Delta(L, \gamma, L),(u^i,u^j)\big)$. But then again the conditions above implies that there exists $l_1, l_2\in L\langle u\rangle$ such that $u^i=\lexp{l_1}u$ and $u^j=\lexp{l_2}u$. Therefore we have
\begin{align*}
\lexp{(l_1^{-1}, l_2^{-1})}{\big(\Delta(L, \gamma, L),(u^i,u^j)\big)}=\big(\Delta(L, i_{l_1}^{-1}\gamma i_{l_2}, L),(u,u)\big)\,.
\end{align*}
This shows that $\widetilde{F^{L\langle u\rangle}_{L, u}} F^{L\langle u\rangle\times L\langle u\rangle}_{\Delta(L,\gamma, L),(u^i,u^j)}\widetilde{F^{L\langle u\rangle}_{L, u}}$ is non-zero if and only if the pair $\big(\Delta(L,\gamma, L),(u^i,u^j)\big)$ is conjugate to a pair of the form $\big(\Delta(L,\gamma', L),(u,u)\big)$, and in that case $\widetilde{F^{L\langle u\rangle}_{L, u}} F^{L\langle u\rangle\times L\langle u\rangle}_{\Delta(L,\gamma, L),(u^i,u^j)}\widetilde{F^{L\langle u\rangle}_{L, u}}$ is a scalar multiple of $F^{L\langle u\rangle\times L\langle u\rangle}_{\Delta(L,\gamma, L),(u^i,u^j)}$. This proves the first claim. The second claim follows now from Remark \ref{rem remarkonpairs}.
\end{proof}

\begin{corollary}\label{cor algebraisom}
The map 
\begin{align*}
\widetilde{F^{L\langle u\rangle}_{L, u}}\FFTD(L\langle u\rangle, L\langle u\rangle)\widetilde{F^{L\langle u\rangle}_{L, u}} &\to \FF \Out(L,u) \\
F^{L\langle u\rangle\times L\langle u\rangle}_{\Delta(L,\gamma, L),(u,u)} &\mapsto \frac{1}{|Z(L\langle u\rangle)|} \gamma
\end{align*}
is an algebra isomorphism.
\end{corollary}
\begin{proof}
This follows from Lemma \ref{lem tensorprodofidempotents} and Proposition \ref{prop algebragenerated}. 
\end{proof}

\begin{corollary}\label{Fppk semisimple}
The category $\Fppk{\FF}$ is semisimple. Moreover, the simple diagonal $p$-permutation functors, up to isomorphism, are parametrized by the isomorphism classes of triples $(L,u,V)$ where $(L,u)$ is a $\DD$-pair, and $V$ is a simple $\FF\Out(L,u)$-module. 
\end{corollary}
\begin{proof}
By Corollary \ref{cor equivalenceofcats}, the category $\Fppk{\FF}$ is equivalent to $\Fun_{\FF}(\calD^{\Delta}, \lMod{\FF})$. By Theorem \ref{thm keypoint} the category $\calD^{\Delta}$ is a product of the categories $\calD^{\Delta}(L,u)$ where $\calD^{\Delta}(L,u)$ is a category with one object, a $\DD$-pair $(L,u)$ up to isomorphism, and hom set $\widetilde{F^{L\langle u\rangle}_{L, u}}\FFTD(L\langle u\rangle, L\langle u\rangle)\widetilde{F^{L\langle u\rangle}_{L, u}}$. The result now follows from Corollary \ref{cor algebraisom}.
\end{proof}

\section{More on simple functors}\label{sec simplefunctors}

Let $(L,u)$ be a $\DD$-pair and let $V$ be a simple $\FF \Out(L,u)$-module. Let $E_{L,u}:=\widetilde{F^{L\langle u\rangle}_{L, u}}\FFTD(L\langle u\rangle, L\langle u\rangle)\widetilde{F^{L\langle u\rangle}_{L, u}}$. We can consider $V$ as an $E_{L,u}$-module via the isomorphism in Corollary \ref{cor algebraisom}. Let $e_V$ denote a primitive idempotent of $\FF \Out(L,u)$ such that $V$ is isomorphic to $\FF \Out(L,u)e_V$. Then the simple diagonal $p$-permutation functor $S_{L,u,V}$ that correspond to the triple $(L,u,V)$ is $S_{L,u,V}=\FFTD_{L,u}e_V$. More precisely, for any finite group $G$, we have
\begin{align*}
S_{L,u,V}(G)=\FFTD(G, L\langle u\rangle)e_V=\{X\otimes_{kL\langle u\rangle} e_V | X\in \FFTD(G, L\langle u\rangle)\}\,.
\end{align*}

Our aim is to give a more precise description for the evaluation $S_{L,u,V}(G)$. 

\begin{nothing}
Let $M:=\FFTD(G,L\langle u\rangle)\widetilde{F^{L\langle u\rangle}_{L, u}}$. Note that $M$ is an $(\FFTD(G,G), E_{L,u})$-bimodule. Hence via the isomorphism in Corollary \ref{cor algebraisom}, $M$ can be viewed as an $(\FFTD(G,G), \FF\Out(L,u))$-bimodule. Note that
\begin{align*}
\sum_{(P,s)\in [\mathcal{Q}^{\Delta}_{G,p}]} F^G_{P,s}=[k] \in \FFT(G)
\end{align*}
implies that the sum
\begin{align*}
\sum_{(P,s)\in [\mathcal{Q}^{\Delta}_{G,p}]} \widetilde{F^G_{P,s}}=[kG] \in \FFTD(G,G)
\end{align*}
is equal to the identity element of $\FFTD(G,G)$. This means that we have the decomposition 
\begin{align*}
M=\bigoplus_{(P,s)\in [\mathcal{Q}^{\Delta}_{G,p}]} \widetilde{F^G_{P,s}} M =\bigoplus_{(P,s)\in [\mathcal{Q}^{\Delta}_{G,p}]} \widetilde{F^G_{P,s}} \FFTD(G,L\langle u\rangle)\widetilde{F^{L\langle u\rangle}_{L, u}}
\end{align*}
as $\FF$-vector spaces. Note that Theorem \ref{thm keypoint} implies that $\widetilde{F^G_{P,s}} M=0$ unless $(\tilde{P},\tilde{s})\cong (L,u)$. Hence as $\FF$-vector spaces, we have
\begin{align*}
M=\sum_{\substack{(P,s)\in [\mathcal{Q}^{\Delta}_{G,p}]\\ (\tilde{P},\tilde{s})\cong (L,u)}}  \widetilde{F^G_{P,s}} \FFTD(G,L\langle u\rangle)\widetilde{F^{L\langle u\rangle}_{L, u}}\,.
\end{align*} 
\end{nothing} 

\begin{nothing}
{\rm (a)} Let $(P,s)$ be a pair of $G$ with the property that $(\tilde{P},\tilde{s})\cong (L,u)$. Using Corollary \ref{cor brauercharacteroftensor} one shows that if $\big(\Delta(R,\gamma,L),(t,u^{i})\big)$ is a pair of $G\times L\langle u\rangle$ with the property that
\begin{align*}
\widetilde{F^G_{P,s}} F^{G\times L\langle u\rangle}_{\Delta(R,\gamma,L),(t,u^{i})}\widetilde{F^{L\langle u\rangle}_{L, u}} \neq 0
\end{align*}
then the pair $\big(\Delta(R,\gamma,L),(t,u^{i})\big)$ is conjugate to a pair of the form $\big(\Delta(P,\phi, L),(s,u)\big)$. But then Lemma \ref{lem scalarmultiple} further implies that $\widetilde{F^G_{P,s}} \FFTD(G,L\langle u\rangle)\widetilde{F^{L\langle u\rangle}_{L, u}}$ is generated by the elements of the form $F^{G\times L\langle u\rangle}_{\Delta(P,\phi,L),(s,u)}$.

\smallskip
{\rm (b)} Fix an isomorphism $\phi_{P,s}:L\to P$ satisfying $\phi_{P,s}(\lexp{u}l)=\lexp{s}\phi_{P,s}(l)$ for all $l\in L$.  Note that the existence of such an isomorphism follows from Lemma \ref{lem pairsofdirectproduct}. For any $g\in N_G(P,s)$, the map $\phi_{P,s}^{-1}\circ i_g\circ \phi_{P,s}$ belongs to $\Aut_{u}(L)$. Therefore we have a group homomorphism
\begin{align}\label{eqn homfromnormaliser}
N_G(P,s)\to \Out(L,u)
\end{align}
that sends $g\in N_G(P,s)$ to the image of $\phi_{P,s}^{-1}\circ i_g\circ \phi_{P,s}$ in $ \Out(L,u)$. This allows us to define an $\FF N_G(P,s)$-module structure on any $\FF \Out(L,u)$-module. Let $N$ denote the image of $N_G(P,s)$ under this homomorphism and set
\begin{align*}
e_{P,s}:=\frac{1}{|N|}\sum_{n\in N}n\,.
\end{align*}
\end{nothing}

\begin{proposition}\label{prop isomofvectorspaces}
Let $G$ be a finite group and let $(L,u)$ be a $\DD$-pair. The map
\begin{align*}
\Psi: \widetilde{F^G_{P,s}} \FFTD(G,L\langle u\rangle)\widetilde{F^{L\langle u\rangle}_{L, u}}&\to e_{P,s} \FF \Out(L,u) \\
F^{G\times L\langle u\rangle}_{\Delta(P,\phi,L),(s,u)} &\mapsto e_{P,s} \overline{\phi_{P,s}^{-1}\phi}
\end{align*}
is an isomorphism of right $\FF\Out(L,u)$-modules where $\overline{\cdot}$ denotes the image in $\Out(L,u)$. 
\end{proposition}
\begin{proof}
First we show that the map $\Psi$ is well-defined. Let $\left(\Delta(P,\phi,L),(s,u)\right)$ and $\left(\Delta(P,\sigma,L),(s,u)\right)$ be two conjugate pairs of $G\times L\langle u\rangle$. Then there exists $(g,l)\in G\times L\langle u\rangle$ such that
\begin{align*}
\left(\Delta(P,\phi,L),(s,u)\right)=\lexp{(g,l)}{\big(\Delta(P,\sigma,L),(s,u)\big)} \,.
\end{align*}
This means that $g\in N_G(P,s)$, $l\in C_{L\langle u\rangle}(u)$ and $\phi=i_g\sigma i_l^{-1}$. Therefore we have
\begin{align*}
\Psi(F^{G\times L\langle u\rangle}_{\Delta(P,\phi,L),(s,u)})&=e_{P,s} \overline{\phi_{P,s}^{-1}\phi}=e_{P,s} \overline{\phi_{P,s}^{-1}i_g\sigma i_l^{-1}}\\&
=\frac{1}{|N|}\sum_{n\in N} n \overline{\phi_{P,s}^{-1}i_g\sigma i_l^{-1}}\\&
=\frac{1}{|N|}\sum_{n\in N} n \overline{\phi_{P,s}^{-1}i_g\phi_{P,s}}\cdot \overline{\phi_{P,s}^{-1}\sigma i_l^{-1}}\\&
=\frac{1}{|N|}\sum_{n\in N} n \overline{\phi_{P,s}^{-1}\sigma}\\&
=\Psi(F^{G\times L\langle u\rangle}_{\Delta(P,\sigma,L),(s,u)})\,.
\end{align*}
This proves the well-definedness. The map $\Psi$ is clearly surjective. For the injectivity, assume that
\begin{align}\label{eqn imagezero}
\Psi\left(\sum_\phi \lambda_\phi F^{G\times L\langle u\rangle}_{\Delta(P,\phi,L),(s,u)} \right) = 0
\end{align}
where $\lambda_\phi \in \FF$ and where the sum runs over a set of isomorphisms $\phi: L\to P$ such that the idempotents $F^{G\times L\langle u\rangle}_{\Delta(P,\phi,L),(s,u)}$ are all distinct. This implies that
\begin{align}\label{eqn imagezero2}
\sum_\phi \lambda_\phi \left(\sum_{n\in N} n \overline{\phi_{P,s}^{-1} \phi} \right) = 0 \,.
\end{align}
Now if $\phi$ and $\sigma$ are isomorphisms $L\to P$ that appear in (\ref{eqn imagezero2}), and if $m,n\in N$ such that
\begin{align*}
n\overline{\phi_{P,s}^{-1}\phi}=m\overline{\phi_{P,s}^{-1}\sigma}\,,
\end{align*}
then there exists $g\in N_G(P,s)$ and $l\in C_{L\langle u\rangle}(u)$ such that
\begin{align*}
\phi_{P,s}^{-1}\circ i_g \circ \phi_{P,s}\circ \phi_{P,s}^{-1} \circ \phi=\phi_{P,s}^{-1}\circ \sigma\circ i_l
\end{align*}
which implies that $i_g\circ\phi\circ i_l^{-1}=\sigma$. Therefore the pairs $\left(\Delta(P,\phi,L),(s,u)\right)$ and $\left(\Delta(P,\sigma,L),(s,u)\right)$ are conjugate and hence $F^{G\times L\langle u\rangle}_{\Delta(P,\phi,L),(s,u)}=F^{G\times L\langle u\rangle}_{\Delta(P,\sigma,L),(s,u)}$. Since the idempotents in (\ref{eqn imagezero}) are chosen to be distinct, this implies that $\phi=\sigma$ and so $n=m$. This shows that the elements $n \overline{\phi_{P,s}^{-1} \phi} \in \Out(L,u)$ in (\ref{eqn imagezero2}) are distinct.  Therefore we have $\lambda_\phi =0$ for any $\phi$ in the sum.  This shows that the map $\Psi$ is injective and hence an isomorphism of $\FF$-vector spaces. The right $\Out(L,u)$-module structure on $\widetilde{F^G_{P,s}} \FFTD(G,L\langle u\rangle)\widetilde{F^{L\langle u\rangle}_{L, u}}$ is given via the algebra isomorphism in Corollary \ref{cor algebraisom} and hence Lemma \ref{lem tensorprodofidempotents} implies that the map $\Phi$ is also an $\FF\Out(L,u)$-module homomorphism. 
\end{proof}

We define an $\FF N_G(P,s)$-module structure on the simple $\FF\Out(L,u)$-module $V$ via the homomorphism in (\ref{eqn homfromnormaliser}). Proposition \ref{prop isomofvectorspaces} implies that the vector space $ \widetilde{F^G_{P,s}} \FFTD(G,L\langle u\rangle)\widetilde{F^{L\langle u\rangle}_{L, u}}e_V$ is isomorphic to the space of $N_G(P,s)$-fixed points $V^{N_G(P,s)}$ of $V$. We proved the following.

\begin{corollary}\label{evaluation simple}
For any finite group $G$, we have
\begin{align*}
S_{L,u,V}(G)\cong \bigoplus_{\substack{(P,s)\in[\mathcal{Q}^{\Delta}_{G,p}]\\ (\tilde{P},\tilde{s})\cong (L,u)}} V^{N_G(P,s)}\,.
\end{align*}
\end{corollary}
\begin{Remark} For $p$-groups $G$ and $H$, the $\FF$-vector space $\FFTD(G,H)$ is canonically isomorphic to $\FF B^\Delta(G,H)$, the $\FF$-linear extension of the Burnside group of bifree $(G,H)$-bisets. Moreover, through this isomorphism, the tensor product of bimodules becomes the composition of bisets. Hence diagonal $p$-permutation functors restricted to $p$-groups are precisely global Mackey functors restricted to $p$-groups (see \cite{Webb1993} for more information on global Mackey functors).\par
For a simple diagonal $p$-permutation functor $S_{L,1,V}$ and a $p$-group $G$, the isomorphism in Corollary~\ref{evaluation simple} becomes
$$S_{L,1,V}(G)\cong \bigoplus_{P\cong L}V^{N_G(P)}$$
where $P$ runs through the subgroups of $G$ up to conjugation, and we recover the formula for the evaluations of simple global Mackey functors (see \cite{Webb1993}, Theorem 2.6(ii)]).
\end{Remark}
\section{Blocks as functors}\label{sec Blocksasfunctors}

In this section, $G$ denotes a finite group and $b$ a block idempotent of $kG$. For an arbitrary commutative ring of coefficients $R$, we define the {\em the block diagonal $p$-permutation functor} $\RTD_{G,b}$ as
\begin{align*}
\RTD_{G,b}: Rpp_k^{\Delta}&\to \lMod{R}\\
H&\mapsto RT^\Delta(H,G)\otimes_{kG} kGb\,.
\end{align*}
Our aim in this section is to describe the functor $\FFTD_{G,b}$ in terms of the simple functors $S_{L,u,V}$. We first make a remark for the case of an arbitrary ring $R$.
\begin{Remark} 
Let $D$ be a defect group of $b$ and let $i\in (kGb)^D$ be a source idempotent of $b$. The source algebra $ikGi$ of $b$ is an interior $D$-algebra and for any finite group $H$ we denote by $\RTD(H, ikGi)$ the Grothendieck group of $(kH, ikGi)$-bimodules whose restriction to $H\times D$ lies in $\RTD(H,D)$. 

By \cite{Puig1981} the map sending a $kGb$-module $M$ to the $ikGi$-module $iM$ induces a Morita equivalence between $\lmod{kGb}$ and $\lmod{ikGi}$.  Hence we have a Morita equivalence between $\lmod{kH}_{kGb}$ and $\lmod{kH}_{ikGi}$ given by a $p$-permutation bimodule in $\RTD(kH, kG)$.  It follows that the functor $\RTD_{G,b}$ is isomorphic to the diagonal $p$-permutation functor $\RTD(-, ikGi)$.  This means in particular that the functor $\RTD_{G,b}$ depends only on the source algebra of $b$. By \cite{Puig1986}, the source algebra of $b$ determines the local points on $kGb$ and one of our aims in the rest of this section is to give a description of the multiplicities of the simple functors in $\FFTD_{G,b}$ in terms of the local points on $kGb$ (see Theorem~\ref{thm multiplicityformulas}).
\end{Remark}
\begin{nothing}\label{noth multiplicityofsimple}
Let $(L,u)$ be a $\DD$-pair and $V$ a simple $\FF\Out(L,u)$-module. We set
\begin{align*}
\mathrm{Mult}(G,b,L,u,V):=kGb\otimes_{kG}\FFTD(G, L\langle u\rangle)\otimes_{kL\langle u\rangle} \widetilde{F^{L\langle u\rangle}_{L,u}}e_V\,,
\end{align*}
where $e_V$ is an idempotent of $\FF\Out(L,u)$ such that $V\cong \FF\Out(L,u)e_V$. 

By the Yoneda lemma we have
\begin{align*}
\Hom_{\mathcal{F}^{\Delta}_{pp_k}}(\FFTD_G, S_{L,u,V})\cong S_{L,u,V}(G)\,.
\end{align*}
Therefore Schur's lemma implies that the multiplicity of the simple functor $S_{L,u,V}$ in the representable functor $\FFTD_G$ is equal to
\begin{align*}
\dim_{\FF}\left(S_{L,u,V}(G)\right)=\dim_{\FF}\left(\FFTD(G, L\langle u\rangle)\otimes_{kL\langle u\rangle} \widetilde{F^{L\langle u\rangle}_{L,u}}e_V\right)\,.
\end{align*} 
This implies that the multiplicity of  $S_{L,u,V}$ in the functor $\FFTD_{G,b}$ is equal to
 \begin{align*}
 \dim_{\FF}\left(\mathrm{Mult}(G,b,L,u,V)\right)\,.
 \end{align*}
 Our aim in this section is to give a description of $\mathrm{Mult}(G,b,L,u,V)$. First, we give a description of $kGb\otimes_{kG}\FFTD(G, L\langle u\rangle)\otimes_{kL\langle u\rangle} \widetilde{F^{L\langle u\rangle}_{L,u}}$. 
\end{nothing}

\begin{nothing} 
Let $M$ be a $(kG, kL\langle u\rangle)$-bimodule and let $\Delta(P, \pi,L)$ be a twisted diagonal subgroup of $G\times L\langle u\rangle$. The Brauer quotient of $kGb\otimes_{kG}M\cong bM$ at $\Delta(P,\pi,L)$ is isomorphic to $\Br_P(b)M\left[\Delta(P,\pi,L)\right]$. By \cite{Broue1985}, this implies, in particular, that any element of $b\FFTD(G, L\langle u\rangle) \widetilde{F^{L\langle u\rangle}_{L,u}}$ is a linear combination of the elements of the form 
 \begin{align*}
m_{P,\pi,E}:= M\left(\Delta(P,\pi,L), E\right)\widetilde{F^{L\langle u\rangle}_{L,u}}\,,
 \end{align*}
 where $\pi: L\to P\le G$ is a group isomorphism, $E$ is a projective indecomposable $kN_{G\times L\langle u\rangle}\left(\Delta(P,\pi,L)\right)/\Delta(P,\pi,L)$-module, and $ M\left(\Delta(P,\pi,L), E\right)$ is the unique, up to isomorphism, indecomposable $p$-permutation $(kG, kL\langle u\rangle)$-bimodule whose Brauer quotient at $\Delta(P,\pi,L)$ is isomorphic to $E$. \\
\end{nothing}

\begin{nothing}\label{noth introductionofcalP}
{\rm (a)} Let $\calP(G,L,u)$ denote the set of pairs $(P, \pi)$ where $P\le G$ is a $p$-subgroup and $\pi:L\to P$ is a group isomorphism for which there exists $s\in G$ such that $\left(\Delta(P,\pi,L), (s,u)\right)$ is a pair of $G\times L\langle u\rangle$. The group $G\times L\langle u\rangle$ acts on $\calP(G,L,u)$ via
\begin{align*}
(g,t)\cdot (P,\pi)=(\lexp{g}P, i_g\circ \pi\circ i_t^{-1})
\end{align*}
for $g\in G$, $t\in L\langle u\rangle$ and $(P,\pi)\in\calP(G,L,u)$. Two elements $(P,\pi)$ and $(Q,\rho)$ of $\calP(G,L,u)$ lie in the same $G\times L\langle u\rangle$-orbit if and only if the subgroups $\Delta(P,\pi,L)$ and $\Delta(Q,\rho,L)$ of $G\times L\langle u\rangle$ are conjugate.  

\smallskip
{\rm (b)} The group $\Aut(L,u)$ also acts on $\calP(G,L,u)$ via
\begin{align*}
(P,\pi)\cdot \gamma = (P,\pi\gamma)
\end{align*}
for $(P,\pi)\in\calP(G,L,u)$ and $\gamma\in\Aut(L,u)$. If also $g\in G$ and $t\in L\langle u\rangle$, we have
\begin{align*}
\Big(\left(g,t\right)\cdot (P,\pi)\Big)\cdot \gamma = (g,\gamma^{-1}(t))\cdot\Big((P,\pi)\cdot \gamma\Big)
\end{align*}

\smallskip
{\rm (c)} Let $[(G\times L\langle u\rangle)\backslash \calP(G,L,u)]$ denote a set of representatives of $G\times L\langle u\rangle$-orbits of $\calP(G,L,u)$. The group $\Out(L,u)$ acts on $[(G\times L\langle u\rangle)\backslash \calP(G,L,u)]$ via
\begin{align*}
[(P,\pi)]\cdot \overline{\gamma} = [(P, \pi\gamma)]
\end{align*}
for $[(P,\pi)] \in [(G\times L\langle u\rangle)\backslash \calP(G,L,u)]$ and $\overline{\gamma}\in \Out(L,u)$, where $\gamma\in \Aut(L,u)$ is an automorphism with image $\overline{\gamma}$ in $\Out(L,u)$. Note that the class $[(P, \pi\gamma)]$ does not depend on the choice of $\gamma$. 

\smallskip
{\rm (d)} Let $\pi:L\to P$ be a group isomorphism and let $s\in G$ be an element with $\pi(\lexp{u}l)=\lexp{s}\pi(l)$ for all $l\in L$.  The $p$-part of the order of $s$ is coprime to the order of $u$. Hence there are integers $a$ and $b$ with
\begin{align*}
a\cdot |u| + b\cdot |s|_p=1\,.
\end{align*}
Now setting $s':= s^{b|s|_p}$ and $s'':=s^{a|u|}$ we get that
\begin{align*}
i_s\circ \pi= i_{s'}\circ i_{s''}\circ \pi =i_{s'}\circ \pi \circ i_{u^{a|u|}}=i_{s'}\circ\pi\,,
\end{align*}
since $u^{a|u|}=1$.  Hence the $p'$-element $s'$ satisfies $i_{s'}\circ \pi=\pi\circ i_u$ which implies that $\pi\in\calP(G,L,u)$. 
\end{nothing}

\begin{nothing}
Let $(P,\pi)\in\calP(G,L,u)$ be a pair and let $\gamma\in \Aut(L,u)$. 

\smallskip
(a) Set $\overline{N}_{P,\pi}:=N_{G\times L\langle u\rangle}\left(\Delta(P,\pi,L)\right)/\Delta(P,\pi,L)$. Since $(s,u)\in N_{G\times L\langle u\rangle}\left(\Delta\left(P,\pi,L\right)\right)$ for some $s\in G$, the group homomorphism
\begin{align*}
\Phi: \overline{N}_{P,\pi}&\to \langle u\rangle \\
\overline{(a,lu^i)}&\mapsto u^i
\end{align*}
is surjective. Also the map
\begin{align*}
\iota: C_G(P)&\to \overline{N}_{P,\pi} \\
x&\mapsto \overline{(x,1)}
\end{align*}
is an injective group homomorphism. One also shows that the kernel of $\Phi$ is equal to the image of~$\iota$. Therefore we have a short exact sequence
\begin{align*}
1\to C_G(P)\to \overline{N}_{P,\pi}\to \langle u\rangle \to 1
\end{align*}
of groups.

\smallskip
(b) Similarly, let $\overline{N}_\gamma=N_{L\langle u\rangle\times L\langle u\rangle}\left(\Delta(L,\gamma,L)\right)/(\Delta(L,\gamma,L)$. One shows that the map
\begin{align*}
\Phi: \overline{N}_\gamma &\to \langle u\rangle \\
\overline{(a,lu^i)}&\mapsto u^i
\end{align*}
is well-defined and surjective. Also the map
\begin{align*}
\iota: Z(L)&\to \overline{N}_\gamma \\
x&\mapsto \overline{(x,1)}
\end{align*}
is an injective group homomorphism. One also shows that the kernel of $\Phi$ is equal to the image of~$\iota$. Therefore we have a short exact sequence
\begin{align*}
1\to Z(L)\to \overline{N}_\gamma \to \langle u\rangle \to 1
\end{align*}
of groups which is split since $\langle u\rangle$ is a $p'$-group and $Z(L)$ is a $p$-group. Therefore we have $\overline{N}_\gamma\cong Z(L)\langle u\rangle$. 

\smallskip
{\rm(c)} We have a group isomorphism 
\begin{align*}
N_{G\times L\langle u\rangle}\left(\Delta(P,\pi\gamma,L)\right) &\to N_{G\times L\langle u\rangle}\left(\Delta(P,\pi,L)\right)\\
(s,t) &\mapsto (s,\gamma(t))
\end{align*}
which maps $\Delta(P,\pi\gamma,L)$ to $\Delta(P,\pi,L)$. Hence we have a group isomorphism
\begin{align} \label{eqn isomofnormalizers}
\overline{N}_{\pi\gamma} \to \overline{N}_\pi\,\,,  \quad \overline{(s,t)} &\mapsto \overline{(s,\gamma(t))} \,.
\end{align}
\end{nothing}

\begin{lemma}\label{lem isomofmodules}
Let $(P,\pi)\in\calP(G,L,u)$ and let $\gamma\in\Aut(L,u)$. Let also $V$ be a  $k\overline{N}_{P,\pi}$-module and let $W$ be a $k\langle u\rangle$-module. Consider the $k\overline{N}_{P,\pi\gamma}$-module $V\otimes_{kZ(L)}\left(\Ind^{\overline{N}_\gamma}_{\langle u\rangle} W \right)$ with the action 
\begin{align*}
\overline{(s,t)}\cdot (v\otimes w):=\overline{(s,\gamma(t))}v\otimes \overline{(\gamma(t),t)} w
\end{align*}
for $\overline{(s,t)}\in \overline{N}_{P,\pi\gamma}$, $v\in V$ and $w\in \Ind^{\overline{N}\gamma}_{\langle u\rangle} W$. Consider also  the $k\overline{N}_{P,\pi\gamma}$-module $V\otimes_{k}W$ with the action
\begin{align*}
\overline{(s,t)}\cdot (v\otimes w):=\overline{(s,\gamma(t))}v\otimes u^i w
\end{align*}
for $\overline{(s,t)}\in \overline{N}_{P,\pi\gamma}$, $v\otimes w\in V\otimes_k W$ where $t=lu^i$. Then the map
\begin{align*}
\Phi: V\otimes_{kZ(L)} \left(k\overline{N}_\gamma\otimes_{k\langle u\rangle}W\right)&\to V\otimes_k W\\
v\otimes (lu^i\otimes w)&\mapsto v\gamma(l)\otimes u^iw
\end{align*}
is an isomorphism of $k\overline{N}_{P,\pi\gamma}$-modules, where we use the isomorphism $\overline{N}_\gamma \cong Z(L)\langle u\rangle$. 
\end{lemma}
\begin{proof}
Clearly, the map $\Phi$ is well-defined and surjective. Let $v\otimes (lu^i\otimes w)\in V\otimes_{kZ(L)} \left(k\overline{N}\otimes_{k\langle u\rangle}W\right)$ and $\overline{(s,t)}\in \overline{N}_{P,\pi\gamma}$. Write $t=l'u^j$ and note that the image of $(\gamma(t),t)$ in $\overline{N}_\gamma$ is $\overline{(\gamma(u^j),u^j)}$. Let also $l_0\in Z(L)$ be the element with the property that $u^jl=l_0 u^j$.  We have
\begin{align*}
\Phi\left(\overline{(s,t)}\cdot\left(v\otimes (lu^i\otimes w)\right)\right)&=\Phi\left(\overline{(s,\gamma(t))}v\otimes\left(\overline{(\gamma(t),t)}\cdot (lu^i\otimes w)\right) \right)\\&
=\Phi\left(\overline{(s,\gamma(t))}v\otimes\left(l_0u^{i+j}\otimes w\right) \right)\\&
=(\overline{(s,\gamma(t))}v)\gamma(l_0)\otimes u^{i+j} w\\&
=\overline{(1,\gamma(l_0^{-1})})\overline{(s,\gamma(t))}v\otimes u^{i+j}w\\&
=\overline{(s,\gamma(l_0^{-1}t))}v\otimes u^{i+j}w\\&
=\overline{(s,\gamma(tl^{-1}))}v\otimes u^{i+j}w\\&
=\overline{(s,\gamma(t))} \overline{(1,\gamma(l^{-1}))} v \otimes u^{i}w\\&
=\overline{(s,\gamma(t))} v\gamma(l)\otimes u^{i}w\\&
=\overline{(s,t)}\cdot \Phi\left(v\otimes (lu^i\otimes w)\right)\,.
\end{align*} 
This shows that $\Phi$ is a $k\overline{N}_{P,\pi\gamma}$-module homomorphism. Since both sides have the same $k$-dimension, it follows that $\Phi$ is an isomorphism.
\end{proof}

\begin{lemma}\label{lem brauerconstructionofprimitiveidemp}
Let $(L,u)$ be a $\DD$-pair and let $\gamma\in \Aut(L,u)$. We have
\begin{align*}
F^{L\langle u\rangle \times L\langle u\rangle}_{\Delta(L,\gamma,L), (u,u)} [\Delta(L,\gamma,L)] = \frac{1}{| Z(L\langle u\rangle)|} \Ind_{\langle u\rangle}^{\overline{N}_\gamma} \left( F^{\langle u\rangle}_{1,u}\right)
\end{align*}
in $\FFT(\overline{N}_\gamma)$. 
\end{lemma}
\begin{proof}
Let $X\le \Delta(L,\gamma,L)$ be a subgroup with the property that $X^{(u,u)}=X$. Let also $\lambda:\langle u\rangle\to k^\times$ be a group homomorphism.	The indecomposable direct summands of 
\begin{align*}
\Ind^{L\langle u\rangle \times L\langle u\rangle}_{\langle X(u,u)\rangle}\left( k^{\langle \Delta(L,\gamma,L)(u,u)\rangle}_{\langle X(u,u)\rangle, \lambda}\right)
\end{align*}
have vertices contained in $\langle X(u,u)\rangle \cap \lexp{s}\Delta(L,\gamma,L)$ for some $s$. Therefore the Brauer construction of such an induced module at $\Delta(L,\gamma,L)$ is zero if $X$ is strictly contained in $\Delta(L,\gamma,L)$. It follows by the primitive idempotent formula (see \cite[Proposition~2.7.8]{Ducellier2015}) that
\begin{align*}
&F^{L\langle u\rangle \times L\langle u\rangle}_{\Delta(L,\gamma,L), (u,u)} [\Delta(L,\gamma,L)]=\\&
\frac{|C_{\Delta(L,\gamma,L)}(u,u)|}{|C_{N_\gamma}(u,u)|}\sum_{\lambda:\langle u\rangle\to k^\times} \tilde{\lambda}(u^{-1})\left(\Ind^{L\langle u\rangle \times L\langle u\rangle}_{\langle \Delta(L,\gamma,L)(u,u)\rangle}\left(\Inf^{\langle \Delta(L,\gamma,L)(u,u)\rangle}_{\langle u\rangle} k_\lambda \right)\right)[\Delta(L,\gamma,L)]\,,
\end{align*}
where $\tilde{\lambda}$ is the Brauer character of $k_\lambda$. The classical formula for the Brauer construction of an induced module implies that
\begin{align*}
\left(\Ind^{L\langle u\rangle \times L\langle u\rangle}_{\langle \Delta(L,\gamma,L)(u,u)\rangle}\left(\Inf^{\langle \Delta(L,\gamma,L)(u,u)\rangle}_{\langle u\rangle} k_\lambda \right)\right)[\Delta(L,\gamma,L)]= \Ind^{N_\gamma}_{\langle \Delta(L,\gamma,L)(u,u)\rangle}\left(\Inf^{\langle \Delta(L,\gamma,L)(u,u)\rangle}_{\langle u\rangle} k_\lambda\right) 
\end{align*}
as $kN_\gamma$-modules. Therefore,
\begin{align*}
F^{L\langle u\rangle \times L\langle u\rangle}_{\Delta(L,\gamma,L), (u,u)} [\Delta(L,\gamma,L)]&=\frac{1}{|\langle u\rangle||Z(L\langle u\rangle)|}\sum_{\lambda:\langle u\rangle\to k^\times} \tilde{\lambda}(u^{-1}) \Ind^{\overline{N}_\gamma}_{\langle u\rangle}\left(k_\lambda \right)\\&
=\frac{1}{|Z(L\langle u\rangle)|} \Ind^{\overline{N}_\gamma}_{\langle u\rangle}\left(F^{\langle u\rangle}_{1,u} \right)\,,
\end{align*}
as desired. 
\end{proof}
 
\begin{lemma}\label{lem Brauerconstructionisom}
Let $\Delta(P,\pi,L)$ be a twisted diagonal subgroup of $G\times L\langle u\rangle$ and let $m_{P,\pi,E}\in b\FFTD(G, L\langle u\rangle) \widetilde{F^{L\langle u\rangle}_{L,u}}$.
\smallskip

{\rm (i)} Let $\Delta(Q,\rho,L)$ be a twisted diagonal subgroup of $G\times L\langle u\rangle$. We have
\begin{align*}
m_{P,\pi,E}[\Delta(Q,\rho,L)]=0
\end{align*}
unless $\Delta(Q,\rho,L)$ is conjugate to $\Delta(P,\pi,L)$.

\smallskip
{\rm (ii)} We have
\begin{align*}
m_{P,\pi,E}\left[\Delta(P,\pi,L)\right] = E\otimes_{k}F^{\langle u\rangle}_{1,u}=\frac{1}{|\langle u\rangle|}\sum_{\lambda:\langle u\rangle\to k^\times}\tilde{\lambda}(u^{-1}) E\otimes_k k_\lambda
\end{align*}
in $\FF\Proj(k\overline{N}_{P,\pi})$. In particular, if $m_{P,\pi,E}\left[\Delta(P,\pi,L)\right]$ is non-zero, then $(P,\pi)\in \calP(G,L,u)$.
\end{lemma}
\begin{proof}
{\rm (i)} By Proposition \ref{Brauer quotient}, the Brauer quotient of the element
\begin{align*}
m_{P,\pi,E}\in b\FFTD(G, L\langle u\rangle) \widetilde{F^{L\langle u\rangle}_{L,u}}
\end{align*}  
at $\Delta(Q,\rho,L)$ is equal to
\begin{align*}
\bigoplus_{\theta:=(\alpha, V,\beta)}\Ind_{X(\theta) \ast Y(\theta)} ^{N_{G\times L\langle u\rangle}(\Delta(Q,\rho, L))} \left( M\left(\Delta(P,\pi,L), E\right)[\Delta(Q,\alpha, V)]\otimes_{kZ(L)} \widetilde{F^{L\langle u\rangle}_{L,u}}[\Delta(V,\beta, L)]\right)\,,
\end{align*}
in $\FFT\left(N_{G\times L\langle u\rangle}(\Delta(Q,\rho, L))\right)$, where $(\alpha, V,\beta)\in \tilde{\Gamma}_{L\langle u\rangle}(Q,\rho,L)$, $X(\theta)=N_{G\times L\langle u\rangle}(\Delta(Q,\alpha,V))$ and $Y(\theta)=N_{L\langle u\rangle\times L\langle u\rangle}(\Delta(V,\beta,L))$.  By Remark \ref{rem conjugatepairs}, we have $ \widetilde{F^{L\langle u\rangle}_{L,u}}=|Z(L\langle u\rangle|\cdot F^{L\langle u\rangle\times L\langle u\rangle}_{\Delta L, (u,u)}$ which implies that the Brauer quotient $\widetilde{F^{L\langle u\rangle}_{L,u}}[\Delta(V,\beta, L)]$ is zero unless the group $\Delta(V,\beta, L)$ is $(L\langle u\rangle\times L\langle u\rangle)$-conjugate to $\Delta L$. This implies that $V=L$ and the map $\beta$ is an inner automorphism of~$L$. Therefore, up to the action of $N_{G\times L\langle u\rangle}\left(\Delta(Q,\rho,L)\right)\times L\langle u\rangle$, we can assume that $\beta$ is equal to identity and hence $\alpha=\rho$. This shows that
\begin{align*}
m_{P,\pi,E}[\Delta(Q,\rho,L)]=0
\end{align*}
unless $\Delta(Q,\rho,L)$ is conjugate to $\Delta(P,\pi,L)$.

\smallskip
{\rm (ii)} We use the calculations in the proof of part (i). One shows that for $\theta=(\pi, L,\id)$, we have
\begin{align*}
N_{G\times L\langle u\rangle}\left(\Delta(P,\pi,L)\right)=X(\theta)\ast Y(\theta)\,.
\end{align*}  
Therefore, the calculations above, together with Lemma \ref{lem isomofmodules} and Lemma \ref{lem brauerconstructionofprimitiveidemp}, imply that
\begin{align*}
m_{P,\pi,E}[\Delta(P,\pi,L)]&= M\left(\Delta(P,\pi,L), E\right)[\Delta(P,\pi,L)]\otimes_{kZ(L)} \widetilde{F^{L\langle u\rangle}_{L,u}}[\Delta(L)]\\&
= E\otimes_{kZ(L)} \Ind_{\langle u\rangle}^{\overline{N}} \left( F^{\langle u\rangle}_{1,u}\right)\\&
=E\otimes_{k}F^{\langle u\rangle}_{1,u} 
\end{align*}
in $\FFT(\overline{N}_\pi)$. For the second assertion, note that if the Brauer construction $m_{P,\pi,E}[\Delta(P,\pi,L)]$ is non-zero, then its Brauer character is non-zero at some $p'$-element $(s,t)\in N_{G\times L\langle u\rangle}(\Delta(P,\pi, L))$. This implies that the Brauer character of the module $F^{\langle u\rangle}_{1, u}$ is non-zero at $t\in \langle u\rangle$ which in turn implies that $t=u$. This shows that $\left(\Delta(P,\pi,L), (s,u)\right)$ is a pair of $G\times L\langle u\rangle$. 
\end{proof}

\begin{nothing}
(a) Let  $\FF \Proj\left(k\Br_P(b)\overline{N}_{P,\pi}, u\right)$ denote the subgroup consisting of elements $\omega$ of the group $\FF\Proj(k\overline{N}_{P,\pi})$ of projective $k\overline{N}_{P,\pi}$-modules such that

\begin{equation*}
 \left \{
  \begin{aligned}
    &\Br_P(b)w=w \\
    & \omega_\lambda = \tilde{\lambda}(u)\omega ~\text{for any} ~\lambda:\langle u\rangle\to k^\times\,.
  \end{aligned} \right.
\end{equation*}

\smallskip
(b) Let $\FF \Proj\left(k\Br_P(b)\overline{N}_{P,\pi}, u\right)^\sharp$ denote the subgroup of $\FF \Proj\left(k\Br_P(b)\overline{N}_{P,\pi}, u\right)$ consisting of linear combinations of projective indecomposable $k\overline{N}_{P,\pi}$-modules $E$ with isotypic restriction to $C_G(P)$, i.e., $\Res^{\overline{N}_{P,\pi}}_{C_G(P)} E$ has an indecomposable direct summand with the inertia group $\overline{N}_{P,\pi}$.  By Theorem \ref{thm cliffordthm} this is equivalent to requiring that $\Res^{\overline{N}_{P,\pi}}_{C_G(P)} E$ is indecomposable. 
\end{nothing}

\begin{lemma}\label{lem spanofsums}
Elements of $\FF \Proj\left(k\Br_P(b)\overline{N}_{P,\pi}, u\right)$ are linear combinations of sums of the type
\begin{align*}
S_E:= \sum_{\lambda:\langle u\rangle\to k^\times}\tilde{\lambda}(u^{-1}) E_\lambda
\end{align*}
where $E$ is a projective indecomposable $k\overline{N}_{P,\pi}$-module.
\end{lemma}
\begin{proof}
Let $v\in \FF \Proj\left(k\Br_P(b)\overline{N}_{P,\pi}, u\right)$ be an arbitrary element and let $E$ be a projective indecomposable $k\overline{N}_{P,\pi}$-module appearing in $v$ with a nonzero coefficient $x$. Since $v_\lambda=\tilde{\lambda}(u)v$ for any $\lambda:\langle u\rangle\to k^\times$, it follows that the coefficient of $E_\lambda$ in $v$ is $\tilde{\lambda}(u^{-1}) \cdot x$. Hence the sum $S_E$ appears in~$v$ with the coefficient $x$. 
\end{proof}

Let $(P,\pi)\in \calP(G,L,u)$ and $m_{P,\pi,E}\in b\FFTD(G, L\langle u\rangle) \widetilde{F^{L\langle u\rangle}_{L,u}}$. Since $bm_{P,\pi,E}=m_{P,\pi,E}$, it follows that $\Br_P(b)$ acts as the identity on $E$, where we identify $C_G(P)$ with its image in $\overline{N}_{P,\pi}$. Moreover, by Lemma \ref{lem Brauerconstructionisom}, we have
\begin{align*}
m_{P,\pi,E}[\Delta(P,\pi,L)] =  \frac{1}{|\langle u\rangle|}\sum_{\lambda:\langle u\rangle\to k^\times}\tilde{\lambda}(u^{-1}) E_\lambda \,.
\end{align*} 
This implies that $m_{P,\pi,E}[\Delta(P,\pi,L)]$ lies in $\FF \Proj\left(k\Br_P(b)\overline{N}_{P,\pi}, u\right)$.

Let $F$ be an indecomposable direct summand of the restriction of $E$ to $C_G(P)$, and let $T$ be its inertial subgroup in $\overline{N}_{P,\pi}$. By Clifford theory, we have $E\cong \Ind_{T}^{\overline{N}_{P,\pi}} W$ for some indecomposable direct summand $W$ of $\Ind_{C_G(P)}^T F$. The isomorphism
\begin{align*}
(\Ind_{T}^{\overline{N}_{P,\pi}} W)\otimes_k k_{\lambda} &\to \Ind_{T}^{\overline{N}_{P,\pi}} (W\otimes_k k_{\lambda}) \\
(n\otimes w)\otimes k_0 &\mapsto n\otimes (w\otimes n^{-1}\cdot k_0)
\end{align*}
implies that we have
\begin{align*}
\sum_{\lambda:\langle u\rangle\to k^\times}\tilde{\lambda}(u^{-1}) E\otimes_k k_\lambda=\Ind_{T}^{\overline{N}_{P,\pi}}\left(\sum_{\lambda:\langle u\rangle\to k^\times}\tilde{\lambda}(u^{-1}) W\otimes_k k_\lambda\right)\,.
\end{align*}
The module $W\otimes_k k_\lambda$ depends only on the restriction of $\lambda$ to the image $\langle u^i\rangle$ of $T$ in $\langle u\rangle$. Let $\lambda_0:\langle u\rangle\to k^\times$ be given. If $\langle u^i\rangle$ is a proper subgroup of $\langle u\rangle$, then the element $u^{-1} \langle u^i\rangle \in \langle u\rangle /\langle u^i\rangle$ is not identity and hence the sum
\begin{align*}
\sum_{\substack{\lambda:\langle u\rangle\to k^\times \\ \lambda |_{\langle u^i\rangle}=\lambda_0 |_{\langle u^i\rangle}}} \tilde{\lambda}(u^{-1}) &= \sum_{\substack{\lambda:\langle u\rangle\to k^\times \\ \langle u^i\rangle \le \ker(\lambda \lambda_0^{-1})}} \tilde{\lambda}(u^{-1}) = \sum_{\substack{\lambda:\langle u\rangle\to k^\times \\ \langle u^i\rangle \le \ker(\lambda \lambda_0^{-1})}} \tilde{\lambda_0}(u^{-1}) \widetilde{\lambda \lambda_0^{-1}}(u^{-1})\\&
=\sum_{\overline{\lambda}:\langle u\rangle /\langle u^i\rangle\to k^\times} \tilde{\lambda_0}(u^{-1}) \Inf_{\langle u\rangle /\langle u^i\rangle}^{\langle u\rangle} \overline{\lambda}(u^{-1})\\&
=\sum_{\overline{\lambda}:\langle u\rangle /\langle u^i\rangle\to k^\times} \tilde{\lambda_0}(u^{-1}) \overline{\lambda}(u^{-1}\langle u^i\rangle)
\end{align*}
is zero. Therefore, $m_{P,\pi,E}[\Delta(P,\pi,L)]=0$ unless $T=\overline{N}_{P,\pi}$. This proves the following.

\begin{lemma}\label{lem isotypic}
Let $(P,\pi)\in \calP(G,L,u)$ and $m_{P,\pi,E}\in b\FFTD(G, L\langle u\rangle) \widetilde{F^{L\langle u\rangle}_{L,u}}$. We have $m_{P,\pi,E}[\Delta(P,\pi,L)]=0$ unless the restriction of $E$ to $C_G(P)$ is isotypic. In particular, $m_{P,\pi,E}[\Delta(P,\pi,L)]$ lies in $\FF \Proj\left(k\Br_P(b)\overline{N}_{P,\pi}, u\right)^\sharp$.
\end{lemma} 

\begin{theorem}\label{thm isomprojective}
The map
\begin{align*}
\Phi: b\FFTD(G, L\langle u\rangle) \widetilde{F^{L\langle u\rangle}_{L,u}} \to \bigoplus_{(P,\pi)\in [(G\times L\langle u\rangle)\backslash \calP(G,L,u)]} \FF \Proj\left(k\Br_P(b)\overline{N}_{P,\pi}, u\right)^\sharp
\end{align*}
sending an element $v$ to the sequence of its Brauer quotients $v[\Delta(P,\pi,L)]$, for $(P,\pi)$ in a set of representatives of $G\times L\langle u\rangle$-orbits of $\calP(G,L,u)$, is an isomorphism of $\FF$-vector spaces.  
\end{theorem}
\begin{proof}
An element in the kernel of $\Phi$ has all its Brauer quotients equal to zero, so it is zero. Hence $\Phi$ is injective. By Lemma \ref{lem spanofsums}, an element of $\FF \Proj\left(k\Br_P(b)\overline{N}_{P,\pi}, u\right)^\sharp$ is a linear combination of the elements $S_E$, where $E$ is a projective indecomposable $k\overline{N}_{P,\pi}$-module with isotypic restriction to $C_G(P)$. The Brauer quotient $m_{P,\pi,E}[\Delta(Q,\rho,L)]$ is equal to zero if $\Delta(Q,\rho,L)$ is not conjugate to $\Delta(P,\pi,L)$, and $m_{P,\pi,E}[\Delta(P,\pi,L)]$ is equal to a nonzero scalar multiple of $S_E$. Hence $\Phi$ is also surjective. 
\end{proof}

\begin{nothing}\label{thesetZ}
We define $\calZ=\calZ(G,L,u)$ as the set of triples $(P,\pi,E)$ where
\begin{itemize}
\item $P$ is a $p$-subgroup of $G$.
\item $\pi:L\to P$ is a group isomorphism such that there exists a $p'$-element $s\in G$ with $\pi(\lexp{u}l)=\lexp{s}\pi(l)$ for all $l\in L$.  
\item $E$ is a projective indecomposable $k\Br_P(b)\overline{N}_{P,\pi}$-module such that $\Res^{\overline{N}_{P,\pi}}_{C_G(P)} E$ is indecomposable. 
\end{itemize}
With the notation above this means that $(P,\pi)\in\calP(G,L,u)$.

\smallskip
{\rm (b)} The group $G\times L\langle u\rangle$ acts on $\calZ$ by
\begin{align*}
(g,t)\cdot (P,\pi,E):=(\lexp{g}P, i_g\pi i_{t^{-1}}, \lexp{(g,t)}E)
\end{align*}
for $(g,t)\in G\times L\langle u\rangle$ and $(P,\pi,E)\in\calZ$. Here $\lexp{(g,t)}E$ is the $k\overline{N}_{\lexp{g}P,i_g\pi i_{t^{-1}}}$-module equal to $E$ as a $k$-vector space and on which $\overline{(a,b)}\in \overline{N}_{\lexp{g}P,i_g\pi i_{t^{-1}}}$ acts by
\begin{align*}
\overline{(a,b)}\cdot \lexp{(g,t)}e:= \overline{(a^g,b^t)}e \,.
\end{align*}
To show that the action is well-defined, we first show that there exists a $p'$-element $a\in G$ such that $(i_g\pi i_{t^{-1}})(\lexp{u}l)=\lexp{a}{\left((i_g\pi i_{t^{-1}})(l)\right)}$ for all $l\in L$.  By \ref{thesetZ}{(a)}, it suffices to show that there exists $a\in G$ with this property.  In other words, we need to show that $\overline{(a,u)}\in \overline{N}_{\lexp{g}P,i_g\pi i_{t^{-1}}}$ for some $a\in G$.  This is equivalent to the condition that $\overline{(a^g,u^t)}\in \overline{N}_{P,\pi}$.  Now one shows that $u^t=l_0\cdot u$ for some $l_0\in L$. Therefore for any $l\in L$ we have
\begin{align*}
\pi(\lexp{u^t}l)=\pi(l_0)\pi(\lexp{u}l)\pi(l_0^{-1})=\pi(l_0) s \pi(l) s^{-1} \pi(l_0^{-1})\,.
\end{align*}
Hence the element $a= g\pi(l_0)sg^{-1} \in G$ satisfies the desired condition.  

Also since the action on $\lexp{(g,t)}E$ is induced from the action on $E$, it follows that $\lexp{(g,t)}E$ is a projective indecomposable $k\Br_{\lexp{g}P}(b) \overline{N}_{\lexp{g}P,i_g\pi i_{t^{-1}}}$-module whose restriction to $kC_G(\lexp{g}P)$ is also indecomposable. 

\smallskip
{\rm (c)} The group $\Aut(L,u)$ acts on $\calZ$ by
\begin{align*}
\varphi\cdot(P,\pi,E):=(P,\pi\varphi^{-1}, \lexp{\varphi}E)
\end{align*}
for $\varphi\in\Aut(L,u)$ and $(P,\pi,E)\in\calZ$, where $\lexp{\varphi}E$ is the $k\overline{N}_{P,\pi\varphi^{-1}}$-module equal to $E$ as the $k$-vector space on which $\overline{(a,b)}\in \overline{N}_{P,\pi\varphi^{-1}}$ acts via
\begin{align*}
\overline{(a,b)}\cdot \lexp{\varphi}e:= \overline{(a,\varphi^{-1}(b))}e \,.
\end{align*}
To show that this action is well-defined, we first need to show that $\overline{(a,u)}\in \overline{N}_{P,\pi\varphi^{-1}}$ for some $a\in G$.  Since $\varphi$ maps $u$ to a conjugate of $u$, the existence of $a$ is similar to Part(b).  

Also, since $\Res^{\overline{N}_{P,\pi\varphi^{-1}}}_{C_G(P)} \left( \lexp{\varphi}E\right) = \Res^{\overline{N}_{P,\pi}}_{C_G(P)} E$, it follows that the action is well-defined. 

\smallskip
{\rm (d)} The group $\widehat{\langle u\rangle}=\Hom(\langle u\rangle, k^\times)$ acts $\calZ$ on the right by
\begin{align*}
(P,\pi,E)\cdot \lambda := (P,\pi,E_\lambda)
\end{align*}
where $E_\lambda$ is the $k\overline{N}_{P,\pi}$-module equal to $E$ as the $k$-vector space on which $\overline{(a,b)}\in \overline{N}_{P,\pi}$ acts by
\begin{align*}
\overline{(a,b)}\cdot e_\lambda := \tilde{\lambda}(b)\overline{(a,b)}e\,.
\end{align*}
Here $\tilde{\lambda}:L\langle u\rangle\to \langle u\rangle\to k^\times$, $l\cdot u \mapsto \lambda(u)$ is the composition map. 

Since $E_\lambda$ is projective indecomposable $k\Br_P(b)\overline{N}_{P,\pi}$-module and since $\Res^{\overline{N}_{P,\pi}}_{C_G(P)} E_\lambda =\Res^{\overline{N}_{P,\pi}}_{C_G(P)}E$ is indecomposable, it follows that this action is well-defined. 

\smallskip
{\rm (e)} The group $\Aut(L,u)$ acts on $G\times L\langle u\rangle$ by $\varphi\cdot (g,t):=(g,\varphi(t))$ for $\varphi\in\Aut(L,u)$ and $(g,t)\in G\times L\langle u\rangle$. We set $S:= (G\times L\langle u\rangle)\rtimes \Aut(L,u)$ using this action. Let $\left((g,t),\varphi\right)\in (G\times L\langle u\rangle)\rtimes \Aut(L,u)$ and $(P,\pi,E)\in\calZ$.  Then for any $l\in L$, we have
\begin{align*}
(i_g\pi i_{t^{-1}} \varphi^{-1})(l)&=i_g\pi \left(t^{-1}\varphi^{-1}(l)t\right)= i_g\pi\left(\varphi^{-1}(\varphi(t^{-1})) \varphi^{-1}(l)\varphi^{-1}(\varphi(t))\right)\\&
=i_g\pi\left(\varphi^{-1}(\varphi(t^{-1}) l \varphi(t))\right)=\left(i_g\pi\varphi^{-1}i_{\varphi(t^{-1})} \right)(l)\,.
\end{align*}
Moreover, for $\overline{(a,b)}\in\overline{N}_{\lexp{g}P,i_g\pi i_{t^{-1}} \varphi^{-1}} = \overline{N}_{\lexp{g}P,i_g\pi \varphi^{-1} i_{\varphi(t^{-1})} }$ and $e\in E$, we have
\begin{align*}
\overline{(a,b)}\cdot \lexp{\varphi}{\left(\lexp{(g,t)} e \right)} &= \overline{(a,\varphi^{-1}(b))}\cdot \lexp{(g,t)} e = \overline{(a^g,\varphi^{-1}(b)^t)}  e\\&
=\overline{(a^g,\varphi^{-1}(b^{\varphi(t)}))}  e = \overline{(a^g,b^{\varphi(t)})} \cdot \lexp{\varphi} e \\&
= \overline{(a,b)}\cdot \lexp{(g,\varphi(t))}{\left(\lexp{\varphi} e \right)} \,.
\end{align*}
These show that
\begin{align*}
\varphi\cdot\left((g,t)\cdot (P,\pi,E)\right) = (g,\varphi(t))\cdot \left(\varphi\cdot (P,\pi,E)\right)\,.
\end{align*}
Therefore the group $S$ acts on $\calZ$.

\smallskip
{\rm (f)} Let $(g,t)\in G\times L\langle u\rangle$, $\varphi\in \Aut(L,u)$, $\lambda\in\widehat{\langle u\rangle}$ and $(P,\pi,E)\in\calZ$.  

For any $b=l_1 u^j\in L\langle u\rangle$, we have $b^t=l_2 u^j$ for some $l_2\in L$, and hence $\tilde{\lambda}(b)=\tilde{\lambda}(b^t)$. Therefore for any $\overline{(a,b)}\in \overline{N}_{\lexp{g}P, i_g\pi i_{t^{-1}}}$ and $e\in E$, we have
\begin{align*}
\overline{(a,b)} \cdot (\lexp{(g,t)} e)_\lambda &= \tilde{\lambda}(b)\overline{(a,b)}\cdot \lexp{(g,t)}e = \tilde{\lambda}(b)\overline{(a^g,b^t)}e = \tilde{\lambda}(b^t)\overline{(a^g,b^t)}e\\&
= \overline{(a^g,b^t)}\cdot e_\lambda = \overline{(a,b)}\cdot \lexp{(g,t)}(e_\lambda)\,.
\end{align*}
This means that 
\begin{align*}
\left((g,t)\cdot (P,\pi,E)\right)\cdot \lambda = (g,t)\cdot \left((P,\pi,E)\cdot\lambda\right)\,,
\end{align*}
i.e., the actions of $G\times L\langle u\rangle$ and $\widehat{\langle u\rangle}$ on $\calZ$ commute. 

Since $\varphi$ maps $u$ to a conjugate of $u$, one has that $\tilde{\lambda}(b)=\tilde{\lambda}(\varphi^{-1}(b))$ for any $b\in L\langle u\rangle$. Therefore calculations similar to above show that
\begin{align*}
\left(\varphi\cdot (P,\pi,E)\right)\cdot \lambda = \varphi\cdot \left((P,\pi,E)\cdot\lambda\right)\,,
\end{align*}
i.e., the actions of $\Aut(L,u)$ and $\widehat{\langle u\rangle}$ on $\calZ$ commute.  These imply that $\calZ$ is an $(S,\widehat{\langle u\rangle})$-biset.  Note that since $E_\lambda = E$ if and only if $\lambda=1$, it follows that $\calZ$ is free on the right.

\smallskip
{\rm (g)}
We have a map from $\calZ$ to $b\FFTD(G,L\langle u\rangle)\widetilde{F^{L\langle u\rangle}_{L,u}}$ sending $(P,\pi,E)$ to $m_{P,\pi,E}$.  This extends to a linear map
\begin{align*}
\Theta: \FF \calZ\to b\FFTD(G,L\langle u\rangle)\widetilde{F^{L\langle u\rangle}_{L,u}} \,.
\end{align*}
\end{nothing}

\begin{lemma}\label{lem propofTheta}
Let $(g,t)\in G\times L\langle u\rangle$, $\varphi\in \Aut(L,u)$, $\lambda\in\widehat{\langle u\rangle}$ and $z\in\calZ$.  Then

\smallskip
{\rm (i)} $\Theta((g,t)z)=\Theta(z)$.

\smallskip
{\rm (ii)} $\Theta(z\lambda)=\lambda(u) \Theta(z)$.

\smallskip
{\rm (iii)} $\Theta(\varphi z)=\Theta(z)\varphi^{-1}$.
\end{lemma}
\begin{proof}
{\rm (i)} Let $z=(P,\pi,E)\in\calZ$. We have
\begin{align*}
\Theta\left((g,t)(P,\pi,E)\right)=\Theta((\lexp{g}P,i_g\pi i_{t^{-1}}, \lexp{(g,t)}E))=m_{\lexp{g}P,i_g\pi i_{t^{-1}}, \lexp{(g,t)}E} = m_{P,\pi,E}\in \FFTD(G,L\langle u\rangle)\,.
\end{align*} 
Hence {\rm (i)} follows.

\smallskip
{\rm (ii)} We have
\begin{align*}
\Theta((P,\pi,E)\lambda)=\Theta(P,\pi,E_\lambda)=m_{P,\pi,E_\lambda} = \lambda(u)m_{P,\pi,E}=\lambda(u)\Theta(P,\pi,E)\,. 
\end{align*}

\smallskip
{\rm (iii)} Finally, we have
\begin{align*}
\Theta(\varphi (P,\pi,E))=\Theta(P,\pi\varphi^{-1},\lexp{\varphi}E)=m_{P,\pi\varphi^{-1},\lexp{\varphi}E}= m_{P,\pi,E} \varphi = \Theta(P,\pi,E)\varphi\,.
\end{align*}
\end{proof}

\begin{nothing}\label{noth permutationsigma}
Let $\Omega = [\calZ / \widehat{\langle u\rangle}]$ be a set of representatives of the right orbits of $\widehat{\langle u\rangle}$ on $\calZ$. As $\calZ$ is an $(S,\widehat{\langle u\rangle})$-biset, we can choose $\Omega$ to be left invariant by the action of $S$. Then, since $G\times L\langle u\rangle$ is a normal subgroup of $S$, the set $\Sigma=(G\times L\langle u\rangle)\backslash \Omega$ is a left $\Aut(L,u)$-set.  Now by Lemma \ref{lem propofTheta}{\rm (i)}, the map $\Theta$ induces a map
\begin{align*}
\overline{\Theta}: \FF\Sigma \to b\FFTD(G,L\langle u\rangle)\widetilde{F^{L\langle u\rangle}_{L,u}}
\end{align*}
sending the orbit $(G\times L\langle u\rangle) (P,\pi,E)$, for $(P,\pi,E)\in\Omega$, to $m_{P,\pi,E}$. 

\smallskip
{\rm (a)} Every element in $b\FFTD(G,L\langle u\rangle)\widetilde{F^{L\langle u\rangle}_{L,u}}$ is a linear combination of the elements $m_{P,\pi,E}$. Moreover, if $(P,\pi,E)$ and $(P',\pi',E')$ are in the same $((G\times L\langle u\rangle)\times \widehat{\langle u\rangle})$-orbit, then the elements $m_{P,\pi,E}$ and $m_{P',\pi',E'}$ differ by a constant. Therefore, the map $\overline{\Theta}$ is surjective. 

Now assume that $\alpha_1 m_{P_1,\pi_1,E_1} +\cdots +\alpha_l m_{P_l,\pi_l,E_l} =0$ for pairwise distinct elements $(P_i,\pi_i,E_i)$ of $\Sigma$ and $\alpha_i\in\FF$.  Fix $i\in\{1,\cdots,l\}$.  Then the $(P_i,\pi_i)$-component of the image of the sum $\alpha_1 m_{P_1,\pi_1,E_1} +\cdots +\alpha_l m_{P_l,\pi_l,E_l}$ under the isomorphism in Theorem \ref{thm isomprojective} is also zero.  This component is equal to
\begin{equation}\label{eqn image sum}
\sum_j \alpha_j S_{E_j}
\end{equation}
up to a non-zero scalar, where the sum runs over $j\in \{1,\cdots, l\}$ with the property that $(P_j,\pi_j)=(P_i,\pi_i)$. Now if $\alpha_i$ is non-zero, then since the sum in (\ref{eqn image sum}) is zero there exists an index $j\neq i$ such that $E_i = (E_j)_\lambda$ for some $\lambda\in\widehat{\langle u\rangle}$. But this implies that $(P_i,\pi_i,E_i)$ and $(P_j,\pi_j,E_j)$ are in the same $(G\times L\langle u\rangle)\times \widehat{\langle u\rangle}$-orbit. This is a contradiction.  Therefore we conclude that the map $\overline{\Theta}$ is injective and hence an isomorphism.

\smallskip
{\rm (b)} By Lemma \ref{lem propofTheta}{\rm (iii)}, we have $\overline{\Theta}(\varphi f)=\overline{\Theta}(f)\varphi^{-1}$ for any $\varphi\in\Aut(L,u)$ and $f\in \FF\Sigma$.  Together with Part{\rm (a)}, this implies that the map $\overline{\Theta}:\FF\Sigma\to b\FFTD(G,L\langle u\rangle)\widetilde{F^{L\langle u\rangle}_{L,u}}$ is an isomorphism of right $\FF\Aut(L,u)$-modules. 
\end{nothing}

\begin{nothing}
Let $\calY=\calY(G,L,u)$ be the set of triples $(P,\pi,F)$ where 
\begin{itemize}
\item $P$ is a $p$-subgroup of $G$.
\item $\pi:L\to P$ is a group isomorphism such that there exists a $p'$-element $s\in G$ with $\pi(\lexp{u}l)=\lexp{s}\pi(l)$ for all $l\in L$.  
\item $F$ is an $u$-invariant projective indecomposable $k\Br_P(b)C_G(P)$-module.
\end{itemize}

\smallskip
{\rm (a)} The group $G\times L\langle u\rangle$ acts on $\calY$ by
\begin{align*}
(g,t)\cdot (P,\pi,F):=(\lexp{g}P, i_g\pi i_{t^{-1}}, \lexp{g}F)
\end{align*}
for $(g,t)\in G\times L\langle u\rangle$ and $(P,\pi,F)\in\calY$. Here $\lexp{g}F$ is the $kC_G(\lexp{g}P)$-module equal to $F$ as a $k$-vector space and on which $c\in C_G(\lexp{g}P)$ acts by
\begin{align*}
c\cdot \lexp{g}f:= c^g f \,.
\end{align*}
To show that this action is well-defined, we only need to show that $\lexp{g}F$ is $u$-invariant projective indecomposable $k\Br_{\lexp{g}P}(b)C_G(\lexp{g}P)$-module.  Since $F$ is projective indecomposable, it follows that $\lexp{g}F$ is also projective indecomposable.  Moreover, by Theorem~\ref{thm cliffordthm} the module $F$ extends to a projective indecomposable $k\overline{N}_{P,\pi}$-module $E$. By \ref{thesetZ}{\rm (b)} the module $\lexp{(g,t)}E$ is a projective indecomposable $k\overline{N}_{\lexp{g}P,i_g\pi i_{t^{-1}}}$-module whose restriction to $C_G(\lexp{g}P)$ is indecomposable.  But the restriction of $\lexp{(g,t)}E$ to $C_G(\lexp{g}P)$ is $\lexp{g}F$ and hence it is $u$-invariant. 

\smallskip
{\rm (b)} The group $\Aut(L,u)$ acts on $\calY$ by
\begin{align*}
\varphi\cdot (P,\pi,F):=(P,\pi\varphi^{-1}, F)
\end{align*}
for $\varphi\in\Aut(L,u)$ and $(P,\pi,F)\in\calY$.  The well-definedness of this action is proved similar to Part{\rm(a)}, using \ref{thesetZ}{\rm(c)}. 

\smallskip
{\rm (c)}  Let $\left((g,t),\varphi\right)\in S=(G\times L\langle u\rangle)\rtimes \Aut(L,u)$ and $(P,\pi,E)\in\calZ$.  Then, as before, for any $l\in L$, we have
\begin{align*}
(i_g\pi i_{t^{-1}} \varphi^{-1})(l)=\left(i_g\pi\varphi^{-1}i_{\varphi(t^{-1})} \right)(l)\,.
\end{align*}
Moreover, for $c\in C_G(\lexp{g}P)$ and $f\in F$, we have
\begin{align*}
c\cdot \lexp{\varphi}{\left(\lexp{(g,t)} f \right)} & = c^g f = c\cdot \lexp{(g,\varphi(t))}{\left(\lexp{\varphi} f \right)}\,.
\end{align*}
These show that
\begin{align*}
\varphi\cdot\left((g,t)\cdot (P,\pi,F)\right) = (g,\varphi(t))\cdot \left(\varphi\cdot (P,\pi,F)\right)\,.
\end{align*}
Therefore the group $S$ acts on $\calY$.
\end{nothing}

\begin{nothing}

\smallskip
{\rm (a)} We view $\calY$ as an $(S,\widehat{\langle u\rangle})$-biset with trivial right action.  Consider the map
\begin{align*}
\Psi:\calZ&\to\calY\\
(P,\pi,E)&\mapsto (P,\pi,\Res^{\overline{N}_{P,\pi}}_{C_G(P)}E) \,.
\end{align*}
Note that the map is well-defined.  Now let $\left((g,t),\varphi\right)\in S$, $\lambda\in\widehat{\langle u\rangle}$ and $(P,\pi,E)\in\calZ$. Then we have
\begin{align*}
\Psi\left( \left(\left(g,t\right),\varphi\right)\cdot \left(P,\pi,E\right)\right)&=\Psi\left(\lexp{g}P, i_g\pi\varphi^{-1}i_{t^{-1}}, \lexp{(g,t)}{(\lexp{\varphi}E)} \right)\\&
=\left(\lexp{g}P, i_g\pi\varphi^{-1}i_{t^{-1}}, \Res^{\overline{N}_{\lexp{g}P, i_g\pi\varphi^{-1}i_{t^{-1}}}}_{C_G(\lexp{g}P)}\left( \lexp{(g,t)}{(\lexp{\varphi}E)}\right)\right)\\&
=\left(\lexp{g}P, i_g\pi\varphi^{-1}i_{t^{-1}},  \lexp{(g,t)}{\left( \Res^{\overline{N}_{P, \pi\varphi^{-1}}}_{C_G(P)} (\lexp{\varphi}E)\right)}\right)\\&
=\left((g,t),\varphi \right)\cdot \left(P,\pi,  \Res^{\overline{N}_{P,\pi}}_{C_G(P)} E \right)\\&
=\left((g,t),\varphi \right)\cdot \Psi(P,\pi, E)\,.
\end{align*}
Hence $\Psi$ is a map of $S$-sets.  Moreover,
\begin{align*}
\Psi\left((P,\pi,E)\cdot \lambda\right)&=\Psi(P,\pi,E_\lambda)=(P,\pi,\Res^{\overline{N}_{P,\pi}}_{C_G(P)} E_\lambda)=(P,\pi,\Res^{\overline{N}_{P,\pi}}_{C_G(P)})=\Psi\left(P,\pi,E\right)\\&=\Psi\left(P,\pi,E\right)\cdot \lambda\,.
\end{align*}
Hence $\Psi$ is also a map of right $\widehat{\langle u\rangle}$-sets and therefore a map of $(S,\widehat{\langle u\rangle})$-bisets. 

\smallskip
{\rm (b)} The map $\Psi$ induces a map $\overline{\Psi}:\Omega = [\calZ /\widehat{\langle u\rangle}]\to \calY$ of $S$-sets. The map $\overline{\Psi}$ is an isomorphism by Theorem~\ref{thm cliffordthm}. Indeed,  given $(P,\pi)\in \calP(G,L,u)$, any $\langle u\rangle$-invariant projective indecomposable $kC_G(P)$-module $F$ extends to a projective indecomposable $k\overline{N}_{P,\pi}$-module $E$ and any such extension differ by an element $\lambda$ of $\widehat{\langle u\rangle}$. 

Now $\overline{\Psi}$ induces an isomorphism of left $\Aut(L,u)$-sets
\begin{align*}
\psi: \left(G\times L\langle u\rangle\right)\setminus \Omega \to \Xi:= \left(G\times L\langle u\rangle\right)\setminus \calY
\end{align*}
and hence \ref{noth permutationsigma}{\rm (b)} implies that
\begin{align}\label{eqn permxi}
b\FFTD(G,L\langle u\rangle)\widetilde{F^{L\langle u\rangle}_{L,u}}\cong \FF\Xi
\end{align}
as right $\FF\Aut(L,u)$-modules where now $\Aut(L,u)$ acts on the right on $\Xi$. 

\smallskip
{\rm (c)} Let $\calU$ be a set of representatives of $\Aut(L,u)$-orbits on $\Xi$. For $(P,\pi,F)\in\calY$, we write $\overline{(P,\pi,F)}:= \left(G\times L\langle u\rangle\right) (P,\pi,F)$. Also for $\overline{(P,\pi,F)}\in\Xi$, we denote by $\Aut(L,u)_{\overline{(P,\pi,F)}}$ the stabilizer of $\overline{(P,\pi,F)}$ in $\Aut(L,u)$.  The isomorphism in (\ref{eqn permxi}) can be written as 
\begin{align}\label{eqn isomsigmainduction}
b\FFTD(G,L\langle u\rangle)\widetilde{F^{L\langle u\rangle}_{L,u}}\cong \bigoplus_{\overline{(P,\pi,F)}\in\calU} \Ind_{\Aut(L,u)_{\overline{(P,\pi,F)}}}^{\Aut(L,u)} \FF\,.
\end{align}
Note that $\overline{(P,\pi,F)}$ and $\overline{(P',\pi',F')}$ lie in the same $\Aut(L,u)$-orbit if and only if there exist $\varphi\in\Aut(L,u)$ and $(g,t)\in G\times L\langle u\rangle$ such that 
\begin{align*}
(P',\pi'\varphi,F')=(\lexp{g}P,i_g\pi i_{t^{-1}}, \lexp{g}F)\,.
\end{align*}
Furthermore, the element $\varphi\in\Aut(L,u)$ belongs to $\Aut(L,u)_{\overline{(P,\pi,F)}}$ if and only if there exists $(g,t)\in G\times L\langle u\rangle$ such that 
\begin{align*}
(P,\pi\varphi,F)=(\lexp{g}P,i_g\pi i_{t^{-1}}, \lexp{g}F)\,,
\end{align*}
i.e., $g\in N_G(P)$, $\varphi\pi=i_g\pi i_{t^{-1}}$ and $\lexp{g}F=F$. 
\end{nothing}

\begin{nothing}\label{noth isomwithD}
{\rm (a)} Let $\calV$ be a set of representatives of $\Aut(L,u)$-orbits on $(G\times L\langle u\rangle)\setminus \calP(G,L,u)$. For $(P,\pi)\in \calP(G,L,u)$, let $\overline{(P,\pi)}$ denote the orbit $(G\times L\langle u\rangle)(P,\pi)$ and let $\Aut(L,u)_{\overline{(P,\pi)}}$ denote its stabilizer in $\Aut(L,u)$. Then for $(P,\pi,F)\in\calY$, the stabilizer $\Aut(L,u)_{\overline{(P,\pi,F)}}$ is a subgroup $\Aut(L,u)_{\overline{(P,\pi)}}$, and we can rewrite (\ref{eqn isomsigmainduction}) as
\begin{align}
b\FFTD(G,L\langle u\rangle)\widetilde{F^{L\langle u\rangle}_{L,u}}\cong \bigoplus_{\overline{(P,\pi)}\in\calV} \Ind_{\Aut(L,u)_{\overline{(P,\pi)}}}^{\Aut(L,u)} \left( \bigoplus_{F\in \calW_{\overline{(P,\pi)}}} \Ind_{\Aut(L,u)_{\overline{(P,\pi,F)}}}^{\Aut(L,u)_{\overline{(P,\pi)}}} \FF\right)\,,
\end{align}
where $\calW_{\overline{(P,\pi)}}$ is a set of representatives of $\Aut(L,u)_{\overline{(P,\pi)}}$-orbits on the set of $u$-invariant projective indecomposable $k\Br_P(b)C_G(P)$-modules. Note that we have
\begin{align*}
\bigoplus_{F\in \calW_{\overline{(P,\pi)}}} \Ind_{\Aut(L,u)_{\overline{(P,\pi,F)}}}^{\Aut(L,u)_{\overline{(P,\pi)}}} \FF \cong \FF\Proj(k\Br_P(b)C_G(P),u)
\end{align*}
as right $\FF\Aut(L,u)_{\overline{(P,\pi)}}$-modules where the action of $\Aut(L,u)_{\overline{(P,\pi)}}$ is given as follows:
Let $F\in \calW_{\overline{(P,\pi)}}$ and let $\varphi\in \Aut(L,u)_{\overline{(P,\pi)}}$. Then there exists $(x,t)\in G\times L\langle u\rangle$ such that 
\begin{align*}
(\lexp{x}P,i_x\pi i_{t^{-1}})=(P,\pi\varphi)\,.
\end{align*}
This is equivalent to the existence of an element $g\in G$ such that 
\begin{align*}
(\lexp{g}P,i_g\pi)=(P,\pi\varphi)\,.
\end{align*}
Indeed, since $\pi\in\calP(G,L,u)$, there exists $s\in G$ such that $\pi(\lexp{u}l)=\lexp{s}\pi(l)$ for all $l\in L$. Now if $t=l_0\cdot u^i$, then $g:=xs^{-i}\pi(l_0^{-1})$ satisfies the desired equation. Moreover, one can show that the element $g$ is well-defined up to multiplication by an element of $C_G(P)$. Now since we have
\begin{align*}
(P,\pi,F)\cdot\varphi = (P,\pi\varphi,F)=(\lexp{g}P,i_g\pi, F)= g\cdot (P,\pi, \lexp{g^{-1}}F)\,,
\end{align*}
the element $\varphi$ maps $F$ to $\lexp{g^{-1}}F$.  

\smallskip
{\rm (b)} All these imply that we have an isomorphism 
\begin{align}\label{eqn isomwithV}
b\FFTD(G,L\langle u\rangle)\widetilde{F^{L\langle u\rangle}_{L,u}}\cong \bigoplus_{\overline{(P,\pi)}\in\calV} \Ind_{\Aut(L,u)_{\overline{(P,\pi)}}}^{\Aut(L,u)} \FF\Proj(k\Br_P(b)C_G(P),u) \,,
\end{align}
of right $\FF\Aut(L,u)$-modules. Arguments in Part{\rm (a)} implies that the orbits of $G$ and of $G\times L\langle u\rangle$ on the set $\calP(G,L,u)$ are the same.  Hence we have a bijection
\begin{align*}
\calD:=[G\setminus \calP(G,L,u)/ \Aut(L,u)]\to \calV=[(G\times L\langle u\rangle)\setminus \calP(G,L,u)/ \Aut(L,u)]
\end{align*}
and the isomorphism in (\ref{eqn isomwithV}) can be written as 
\begin{align}\label{eqn isomwithD}
b\FFTD(G,L\langle u\rangle)\widetilde{F^{L\langle u\rangle}_{L,u}}\cong \bigoplus_{\overline{(P,\pi)}\in\calD} \Ind_{\Aut(L,u)_{\overline{(P,\pi)}}}^{\Aut(L,u)} \FF\Proj(k\Br_P(b)C_G(P),u) \,.
\end{align}
\end{nothing}

\begin{nothing}
Let $\hat{\calF}_b$ denote the category whose objects are $b$-Brauer pairs $(P,e)$ and whose morphisms from $(P,e)$ to $(Q,f)$ are group homomorphisms $\psi: P\to Q$ for which there exists $g\in G$ such that $\psi(x)=\lexp{g}x$ for any $x\in P$ and such that $\lexp{g}{(P,e)}\le (Q,f)$ (see, for instance, \cite[Section~6.3]{linckelmann2018} for more details on Brauer pairs). 

Let $\hat{\calP}_b(G,L,u)$ be the set of triples $(P,e,\pi)$ where
\begin{itemize}
\item $(P,e)\in\hat{\calF}_b$
\item $\pi:L\to P$ is a group isomorphism such that $\pi i_u \pi^{-1} \in \Aut_{\hat{\calF}_b}(P,e)$.
\end{itemize}

\smallskip
{\rm (a)} The set $\hat{\calP}_b(G,L,u)$ is a $(G,\Aut(L,u))$-biset via
\begin{align*}
g\cdot (P,e,\pi)\cdot\varphi = (\lexp{g}P,\lexp{g}e, i_g\pi\varphi)
\end{align*}
for $g\in G$, $(P,e,\pi)\in\hat{\calP}_b(G,L,u)$ and $\varphi\in\Aut(L,u)$.  Indeed, we have $(\lexp{g}P,\lexp{g}e)\in\hat{\calP}_b(G,L,u)$ and one shows that
\begin{align*}
i_g\pi\varphi i_u \varphi^{-1}\pi^{-1} i_{g^{-1}} = i_{s'}
\end{align*}
where $s'=g\pi(l_0)s\in N_G(\lexp{g}P,\lexp{g}e)$, $l_0\in L$ with $\pi(u)=l_0\cdot u$ and $s\in N_G(P,e)$ such that $\pi i_u\pi^{-1}=i_s$. 

\smallskip
{\rm (b)} Let $\overline{(P,\pi)} \in \calD$. We have 
\begin{align*}
\FF\Proj(k\Br_P(b)C_G(P),u)\cong \bigoplus_{\substack{e\in\mathrm{Bl}(C_G(P))\\ \Br_P(b)e=e}} \FF\Proj(keC_G(P),u)
\end{align*}
as right $\FF\Aut(L,u)_{\overline{(P,\pi)}}$-modules.  Note that the group $\Aut(L,u)_{\overline{(P,\pi)}}$ permutes the summands on the right-hand side as follows: Let $\varphi\in \Aut(L,u)_{\overline{(P,\pi)}}$ and let $e\in \mathrm{Bl}(C_G(P))$ with $\Br_P(b)e=e$. Then there exists $g\in N_G(P)$ such that $i_g\pi=\pi \varphi$ and by \ref{noth isomwithD}(a), $\varphi$ sends an $u$-invariant projective indecomposable $keC_G(P)$-module $S$ to $\lexp{g^{-1}}S$.  Hence $\varphi$ maps the $b$-Brauer pair $(P,e)$ to $(P,\lexp{g^{-1}}e)$. 

Let $\Bl(C_G(P),u)$ denote the set of block idempotents $e$ of $C_G(P)$ for which there exists an indecomposable projective $u$-invariant $keC_G(P)$-module. Let $[\Bl(C_G(P),u)]$ denote a set of representatives of $\Aut(L,u)_{\overline{(P,\pi)}}$-orbits of $\Bl(C_G(P),u)$. Then as right $\Aut(L,u)_{\overline{(P,\pi)}}$-modules, we have
\begin{align*}
\FF\Proj(k\Br_P(b)C_G(P),u)\cong \bigoplus_{e\in [\Bl(C_G(P),u)]} \Ind_{\Aut(L,u)_{\overline{(P,e,\pi)}}}^{\Aut(L,u)} \FF\Proj(keC_G(P),u)
\end{align*}
where $\Aut(L,u)_{\overline{(P,e,\pi)}}$ is the stabilizer in $\Aut(L,u)_{\overline{(P,\pi)}}$ of $e$. 

Note that $\varphi\in\Aut(L,u)_{\overline{(P,\pi)}}$ fixes $e$ if and only if there exists $g\in N_G(P)$ with $i_g\pi=\pi \varphi$ and $\lexp{g}e=e$.  Hence $\Aut(L,u)_{\overline{(P,e,\pi)}}$ is equal to the stabilizer $\Aut(L,u)_{P,e,\pi}$ of the $G$-orbit of $(P,e)$ in $\Aut(L,u)$.  These imply that we have 
\begin{align}\label{eqn isomwithgeneralizedfusion}
b\FFTD(G,L\langle u\rangle)\widetilde{F^{L\langle u\rangle}_{L,u}}\cong \bigoplus_{(P,e,\pi)\in [G \backslash \hat{\calP}_b(G,L,u) / \Aut(L,u) ]} \Ind_{\Aut(L,u)_{P,e,\pi}}^{\Aut(L,u)}  \FF\Proj(keC_G(P),u)
\end{align}
as right $\FF\Aut(L,u)$-modules. 
\end{nothing}

\begin{nothing}\label{noth isomwithfusion}
Let $\calF_b$ denote the fusion system of $kGb$ with respect to a maximal $b$-Brauer pair $(D,e_D)$.  For each subgroup $P\le D$ let $e_P$ denote the unique block of $kC_G(P)$ with $(P,e_P)\le (D,e_D)$.  Note that every $G$-orbit in $\hat{\calF}_b$ contains an element in $\calF_b$. 

For $(P,e_P)\in \calF_b$, let $\calP_{(P,e_P)}(L,u)$ denote the set of group isomorphisms $\pi:L\to P$ with $\pi i_u \pi^{-1}\in \Aut_{\calF_b}(P,e_P)$.  The set $\calP_{(P,e_P)}(L,u)$ is an $(N_G(P,e_P), \Aut(L,u))$-biset via
\begin{align*}
g\cdot \pi\cdot\varphi = i_g\pi\varphi
\end{align*}
for $g\in N_G(P,e_P)$, $\pi\in \calP_{(P,e_P)}(L,u)$ and $\varphi\in\Aut(L,u)$. Let $[\calP_{(P,e_P)}(L,u)]$ denote a set of representatives of $N_G(P,e_P)\times \Aut(L,u)$-orbits of $\calP_{(P,e_P)}(L,u)$.  Then the isomorphism in (\ref{eqn isomwithgeneralizedfusion}) can be written as 
\begin{align}\label{eqn isomwithfusion}
b\FFTD(G,L\langle u\rangle)\widetilde{F^{L\langle u\rangle}_{L,u}}\cong \bigoplus_{(P,e_P) \in [\calF_b]} \bigoplus_{\pi \in [\calP_{(P,e_P)}(L,u)]} \Ind_{\Aut(L,u)_{P,e,\pi}}^{\Aut(L,u)}  \FF\Proj(keC_G(P),u)\,.
\end{align}
\end{nothing}

\begin{nothing}\label{noth isomwithlocalpoints}
Let $\calL_b(G,L,u)$ denote the set of pairs $(P_\gamma, \pi)$ where
\begin{itemize}
\item $P_\gamma$ is a local pointed point group on $kGb$,
\item $\pi:L\to P$ is a group isomorphism such that $\pi i_u \pi^{-1} = \Res(i_s)$ for some $s\in N_G(P_\gamma)$. 
\end{itemize}
The set $\calL_b(G,L,u)$ is a $(G, \Aut(L,u))$-biset via
\begin{align*}
g\cdot (P_\gamma, \pi)\cdot\varphi = (\lexp{g}P_{\lexp{g}\gamma}, i_g\pi\varphi)
\end{align*}
for $g\in G$ and $\varphi\in\Aut(L,u)$.  For $(P_\gamma, \pi)\in \calL_b(G,L,u)$, we write $\Aut(L,u)_{(P_\gamma,\pi)}$ for the stabilizer of the $G$-orbit of $(P_\gamma, \pi)$ in $\Aut(L,u)$.

Let $\overline{(P,\pi)}\in\calD$. A projective indecomposable $k\Br_P(b)C_G(P)$-module $S$ determines a conjugacy class of a primitive idempotent of $k\Br_P(b)C_G(P)$ and hence via the Brauer morphism
\begin{align*}
\Br_P: (kGb)^P\to kC_G(P)
\end{align*}
a local point $\gamma$ of $P$ on $kGb$. This in fact induces a bijection between the sets of local points of $P$ on $kGb$ and projective indecomposable $k\Br_P(b)C_G(P)$-modules. See for instance \cite[Corollary 37.6]{Thevenaz1995} for more details.  This bijection induces an isomorphism of right $\Aut(L,u)_{\overline{(P,\pi)}}$-modules
\begin{align*}
\FF\Proj(k\Br_P(b)C_G(P),u)\cong \FF \calL_{(P,\pi)}
\end{align*}
where $\calL_{(P,\pi)}$ is the set of local points $P_\gamma$ with $\pi i_u\pi^{-1} = i_s$ for some $s\in N_G(P_\gamma)$.  The action of $\varphi\in\Aut(L,u)_{\overline{(P,\pi)}}$ on $\FF \calL_{(P,\pi)}$ is given as follows: There exists $g\in G$ such that $(\lexp{g}P, i_g\pi)=(P,\pi\varphi)$. Then $\varphi$ maps $P_\gamma\in \calL_{(P,\pi)}$ to $P_{\lexp{g^{-1}}\gamma}$.  Hence the isomorphism in (\ref{eqn isomwithD}) implies that 
\begin{align*}
b\FFTD(G,L\langle u\rangle)\widetilde{F^{L\langle u\rangle}_{L,u}}\cong \bigoplus_{\overline{(P,\pi)}\in\calD} \Ind_{\Aut(L,u)_{\overline{(P,\pi)}}}^{\Aut(L,u)} \left(\bigoplus_{P_\gamma\in [\calL_{(P,\pi)} / \Aut(L,u)_{\overline{(P,\pi)}}]} \Ind^{\Aut(L,u)_{\overline{(P,\pi)}}}_{\Aut(L,u)_{\overline{(P,\pi)}_\gamma}} \FF \right)
\end{align*}
as right $\FF\Aut(L,u)$-modules where $\Aut(L,u)_{\overline{(P,\pi)}_\gamma}$ is the stabilizer of $P_\gamma$ in $\Aut(L,u)_{\overline{(P,\pi)}}$.  But one shows that $\Aut(L,u)_{\overline{(P,\pi)}_\gamma}$ is equal to $\Aut(L,u)_{(P_\gamma,\pi)}$ and it follows that
\begin{align}\label{eqn isomwithlocalpoints}
b\FFTD(G,L\langle u\rangle)\widetilde{F^{L\langle u\rangle}_{L,u}}\cong \bigoplus_{(P_\gamma,\pi)\in [G \backslash \calL_b(G,L,u) / \Aut(L,u) ]} \Ind_{\Aut(L,u)_{(P_\gamma,\pi)}}^{\Aut(L,u)}  \FF
\end{align}
as right $\FF\Aut(L,u)$-modules. 
\end{nothing}

\begin{theorem}\label{thm multiplicityformulas}
Let $S_{L,u,V}$ be a simple diagonal $p$-permutation functor and let $b$ be a block idempotent of $kG$.  The multiplicity of $S_{L,u,V}$ in $\FFTD_{G,b}$ is equal to the $\FF$-dimensions of any of the following vector spaces.

\smallskip
{\rm (a)} $\bigoplus_{\overline{(P,\pi,F)}\in\calU} V^{\Aut(L,u)_{\overline{(P,\pi,F)}}}$

\smallskip
{\rm (b)} $\bigoplus_{(P,e_P) \in [\calF_b]} \bigoplus_{\pi \in [\calP_{(P,e_P)}(L,u)]} \FF\Proj(keC_G(P),u)\otimes_{\Aut(L,u)_{P,e,\pi}} V$

\smallskip
{\rm (c)} $\bigoplus_{(P_\gamma,\pi)\in [G \backslash \calL_b(G,L,u) / \Aut(L,u) ]} V^{\Aut(L,u)_{(P_\gamma,\pi)}}$
\end{theorem}
\begin{proof}
Recall from \ref{noth multiplicityofsimple} that the multiplicity of $S_{L,u,V}$ in $\FFTD_{G,b}$ is equal to the $\FF$-dimension of $b\FFTD(G,L\langle u\rangle)\widetilde{F^{L\langle u\rangle}_{L,u}} e_V$ where $e_V$ is an idempotent of $\FF\Out(L,u)$ such that $V\cong \FF\Out(L,u)e_V$.  Hence part {\rm (a)} follows from the isomorphism in (\ref{eqn isomsigmainduction}), part {\rm (b)} follows from the isomorphism in (\ref{eqn isomwithfusion}), and part {\rm (c)} follows from the isomorphism in (\ref{eqn isomwithlocalpoints}).
\end{proof}

\begin{corollary}\label{cor multiplicityoftrivial and defect}
{\rm (i)} The multiplicity of the simple functor $S_{1,1,\FF}$ in the functor $\FFTD_{G,b}$ is equal to the number of isomorphism classes of simple $kGb$-modules.

\smallskip
{\rm (ii)} The multiplicity of the simple functor $S_{L,u,V}$ at $\FFTD_{G,b}$ is zero unless $L$ is isomorphic to a subgroup of a defect group of $b$. 
\end{corollary}
\begin{proof}
Both statements follow immediately from Theorem~\ref{thm multiplicityformulas}.
\end{proof}

\section{Nilpotent blocks}\label{sec nilpotentblocks}

Nilpotent blocks were introduced by Brou{\'e} and Puig in \cite{BrouePuig1980}. In this section we give a characterization of nilpotent blocks in terms of diagonal $p$-permutation functors.  Let $G$ be a finite group.

\begin{definition}[\cite{BrouePuig1980}]
A block idempotent $b$ of $kG$ is called {\em nilpotent} if for any $b$-Brauer pair $(P,e)$ the quotient $N_G(P,e)/C_G(P)$ is a $p$-group.
\end{definition}

\begin{theorem}\label{thm nilpotentcharacterization}
Let $b$ be a block idempotent of $kG$ with a defect group $D$. The following are equivalent.

\smallskip
{\rm (i)} The block $b$ is nilpotent.

\smallskip
{\rm (ii)} If $S_{L,u,V}$ is a simple summand of $\FFTD_{G,b}$, then $u=1$. 

\smallskip
{\rm (iii)} The functor $\FFTD_{G,b}$ is isomorphic to the functor $\FFTD_D$.
\end{theorem}
\begin{proof}
{\rm (i)} $\Rightarrow$ {\rm (ii)}: 
Suppose that $b$ is nilpotent and that $S_{L,u,V}$ is a simple summand of $\FFTD_{G,b}$.  Then by Theorem~\ref{thm multiplicityformulas}(b),  there exists a triple $(P,e,\pi)\in\hat{\calP}_b(G,L,u)$ such that
\begin{align*}
\FF \Proj(keC_G(P),u) \neq 0\,.
\end{align*}
Let $s\in N_G(P,e)$ be an element with the property $\pi i_u\pi^{-1}=i_s:P\to P$. By \ref{noth introductionofcalP}(d), there exists a $p'$-element $s'\in N_G(P,e)$ with $\pi i_u\pi^{-1}=i_{s'}$. Since the block idempotent $b$ is nilpotent, the quotient $N_G(P,e)/C_G(P)$ is a $p$-group. It follows that $s'\in C_G(P)$ and hence
\begin{align*}
\pi(\lexp{u}l)=\lexp{s'}\pi(l)=\pi(l)
\end{align*}
for any $l\in L$. Since $(L,u)$ is a $\DD$-pair, this means that $u=1$. Hence {\rm (i)} implies {\rm (ii)}.

\smallskip
{\rm (ii)} $\Rightarrow$ {\rm (i)}: 
Assume that {\rm (ii)} holds. Then for any $\DD$-pair $(L,u)$ with $u\neq 1$, by Theorem~\ref{thm multiplicityformulas}(a), the set $\calY(G,L,u)$ is empty.  This is equivalent to the following statement:

$(B)_{G,b}$: If $P$ is a $p$-subgroup of $G$ and if $s\in N_G(P)$ induces a non-trivial $p'$-automorphism of~$P$, then there is no $s$-invariant simple $k\Br_P(b)C_G(P)$-module.

Indeed, if $s\in N_G(P)$ induces a nontrivial $p'$-automorphism $u$ of $P$, setting $L=P$, $\pi=\id$, then $(L,u)$ is a $\DD$-pair and $(P,\pi)\in\calP(G,L,u)$. Hence if $S$ is an $s$-invariant simple $k\Br_P(b)C_G(P)$-module, then $(P,\pi,\hat{S})\in\calY(G,L,u)$ where $\hat{S}$ is a projective cover of $S$.  This contradicts our assumption.

Now we claim that the statement $(B)_{G,b}$ is equivalent to the following statement:

$(C)_{G,b}$: If $(P,e)$ is a $b$-Brauer pair and if $s\in N_G(P,e)$ induces a nontrivial $p'$-automorphism of $P$, then there is no $s$-invariant simple $kC_G(P)e$-module.

Indeed, $(B)_{G,b}\Rightarrow (C)_{G,b}$ is clear. Now assume that $(C)_{G,b}$ holds and suppose that $s\in N_G(P)$ induces a non-trivial $p'$-automorphism of $P$ and $S$ is an $s$-invariant simple $k\Br_P(b)C_G(P)$-module. Then $S$ belongs to a unique block $e$ of $kC_G(P)$. Since $S$ is $s$-invariant, it follows that $e$ is also $s$-invariant. This is a contradiction. 

We will prove that the statement $(C)_{G,b}$ implies that the block $b$ is nilpotent.  We use induction on the order of $G$. 

If $G$ is the trivial group, then the block $kG=k$ is obviously nilpotent. Now assume that the statement $(C)_{H,c}$ implies that the block $c$ is nilpotent whenever $c$ is a block idempotent of a group $H$ with $|H|< |G|$, and assume that the statement $(C)_{G,b}$ holds. 

Let $(P,e)$ be a $b$-Brauer pair. We will show that $N_G(P,e)/C_G(P)$ is a $p$-group.  Set $H=C_G(P)$.  If $H=G$, then the quotient $N_G(P,e)/C_G(P)$ is trivial hence a $p$-group.  

So we can assume that $|H|<|G|$. Let $(Q,f)$ be an $e$-Brauer pair of $H$. Then one can show that the pair $(QP,f)$ is a $b$-Brauer pair of $G$.  Let $s\in N_H(Q,f)$ be an element which induces a nontrivial $p'$-automorphism $u$ of $Q$. Then $s\in C_G(P)\cap N_G(Q,f)\subseteq N_G(QP,f)$ induces a nontrivial $p'$-automorphism of of $QP$. Since $(C)_{G,b}$ holds and since $(QP,f)$ is a $b$-Brauer pair, it follows that there is no $s$-invariant simple $kC_G(QP)f$-module. Therefore there is no $s$-invariant simple $kC_H(Q)f$-module.  This proves that the statement $(C)_{H,e}$ holds. Since $|H|<|G|$, by induction hypothesis, the block $e$ of $kH=kC_G(P)$ is nilpotent.  So there is a unique simple module $S$ of $kC_G(P)e$. Now if $t\in N_G(P,e)$ induces a nontrivial $p'$-automorphism $v$ of $P$, then $S$ is invariant by $t$ and hence $v=1$ since $(C)_{G,b}$ holds.  In other words $t\in C_G(P)$ and so the quotient $N_G(P,e)/C_G(P)$ is a $p$-group.  This shows that {\rm (ii)} implies {\rm (i)}.

\smallskip
{\rm (iii)} $\Rightarrow$ {\rm (ii)}: This is clear. 

\smallskip
{\rm (i)} $\Rightarrow$ {\rm (iii)}: Assume that $b$ is nilpotent. By the first part of the proof, if $S_{L,u,V}$ is a simple summand of $\FFTD_{G,b}$, then $u=1$.  So let $S_{L,1,V}$ be a simple functor. We will show that
\begin{align*}
b\FFTD(G,L)\widetilde{F^{L}_{L,1}} \cong \FFTD(D,L)\widetilde{F^{L}_{L,1}}
\end{align*}
as right $\FF\Aut(L)$-modules using the isomorphism in (\ref{eqn isomwithfusion}). Since $b$ is nilpotent, for any $b$-Brauer pair $(P,e_P)$, the block idempotent $e_P\in kC_G(P)$ is nilpotent.  So there is a unique simple $ke_PC_G(P)$-module, and hence $\FF\Proj(ke_PC_G(P),1)\cong \FF$.  Moreover, $\calF_b = \calF_D$, $\calP(P,e_P)(L,1)=\mathrm{Isom}(L,P)$ and $N_G(P,e_P)=N_D(P)$.  The result follows. 
\end{proof}

\section{Functorial equivalence of blocks}
In this section $G$ and $H$ denote finite groups. We come back to the case of an arbitrary commutative ring $R$ of coefficients.

\begin{definition}\label{functorially equivalent}
Let $b$ be a block idempotent of $kG$ and let $c$ be a block idempotent of~$kH$. We say that the pairs $(G,b)$ and $(H,c)$ are {\em functorially equivalent} over $R$, if the corresponding diagonal $p$-permutation functors $\RTD_{G,b}$ and $\RTD_{H,c}$ are isomorphic in $\Fppk{R}$. 
\end{definition}

\begin{lemma}
Let $(G,b)$ and $(H,c)$ be as in Definition~\ref{functorially equivalent}.

\smallskip
{\rm (i)} $(G,b)$ and $(H,c)$ are functorially equivalent over~$R$ if and only if there exists $\omega\in b\RTD(G,H)c$ and $\sigma\in c\RTD(H,G)b$ such that
\begin{align*}
\omega \cdot_G \sigma = [kGb] \quad \text{in} \quad b\RTD(G,G)b \quad \text{and} \quad \sigma \cdot_H \omega = [kHc] \quad \text{in} \quad c\RTD(H,H)c \,.
\end{align*}

\smallskip
{\rm (ii)} If $kGb$ and $kHc$ are $p$-permutation equivalent, then $(G,b)$ and $(H,c)$ are functorially equivalent over~$R$. 

\end{lemma}
\begin{proof}
The first statement follows from the Yoneda lemma. The second statement follows from the first one, and from the definition of $p$-permutation equivalence in \cite{BoltjePerepelitsky2020}.
\end{proof}
\begin{Remark} It follows that functorial equivalence over $\ZZ$ is almost the same notion as $p$-permutation equivalence, which only requires in addition that $\sigma$ be the opposite of $\omega$ in~(i).
\end{Remark}
\begin{proposition}\label{fun equivalence}
Let $b$ be a block idempotent of $kG$ and $c$ a block idempotent of $kH$. 

\smallskip
{\rm (i)} If $(G,b)$ and $(H,c)$ are functorially equivalent over~$\FF$, then we have $l(kGb)=l(kHc)$.  

\smallskip
{\rm (ii)} If $(G,b)$ and $(H,c)$ are functorially equivalent over~$\FF$, then $b$ and $c$ have isomorphic defect groups. 

\smallskip
{\rm (iii)} If $b$ has defect zero, then the functor $\FFTD_{G,b}$ is isomorphic to the simple functor $S_{1,1,\FF}$. In particular, all pairs $(G,b)$, where $b$ is a block of defect zero of $kG$, are functorially equivalent over~$\FF$.

\smallskip
{\rm (iv)} More generally, for any $p$-group $D$, all pairs $(G,b)$, where $b$ is a nilpotent block of $kG$ with defect isomorphic to $D$ are functorially equivalent over $\FF$. 
\end{proposition}
\begin{proof}
The first statement follows from Corollary \ref{cor multiplicityoftrivial and defect}(i).  For the second statement, let $D$ be a defect group of $b$. The multiplicity of the simple functor $S_{D,1,\FF}$ at $\FFTD_{G,b}$ is non-zero by Theorem~\ref{thm multiplicityformulas}.  Hence it is also nonzero at $\FFTD_{H,c}$. By Corollary \ref{cor multiplicityoftrivial and defect}(ii), it follows that $D$ is isomorphic to a subgroup of a defect group of the block $c$. Similarly, one can show that a defect group of $c$ is isomorphic to a subgroup of a defect group of $b$ whence {\rm (ii)} holds.  For the third statement assume that $b$ has defect zero. Then by Corollary \ref{cor multiplicityoftrivial and defect}, the functor  $\FFTD_{G,b}$ is isomorphic to $l(kGb)S_{1,1,\FF}=S_{1,1,\FF}$. The last statement follows from Theorem \ref{thm nilpotentcharacterization}. 
\end{proof}

\begin{theorem}
Let $b$ be a block idempotent of $kG$ and $c$ a block idempotent of $kH$. The following are equivalent:

\smallskip
{\rm (i)} $(G,b)$ and $(H,c)$ are functorially equivalent over $\FF$.

\smallskip
{\rm (ii)} For any $\DD$-pair $(L,u)$ and any $\varphi\in\Aut(L,u)$, one has
\begin{align*}
| (G\setminus \calL_b(G,L,u))^\varphi | = | (H\setminus \calL_c(H,L,u))^\varphi |\,.
\end{align*}
\end{theorem}
\begin{proof}
{\rm (i)} $\Rightarrow$ {\rm (ii)}: Suppose that $(G,b)$ and $(H,c)$ are functorially equivalent over $\FF$, and let $(L,u)$ be a $\DD$-pair.  For any simple $\FF\Out(L,u)$-module $V$, we have,  by Theorem~\ref{thm multiplicityformulas}(iii), 
\begin{align*}
\sum_{(P_\gamma,\pi)\in [G \backslash \calL_b(G,L,u) / \Aut(L,u) ]} \hspace{-3ex} \dim_{\FF}( V^{\Aut(L,u)_{(P_\gamma,\pi)}}) = \sum_{(Q_\delta,\rho)\in [H \backslash \calL_c(H,L,u) / \Aut(L,u) ]} \hspace{-3ex}\dim_{\FF}(V^{\Aut(L,u)_{(Q_\delta,\rho)}})\,.
\end{align*}
Note that since $\Inn(L,u)$ acts trivially on the $G$-orbits of $\calL_b(G,L,u)$, this means that we have
\begin{align*}
\sum_{(P_\gamma,\pi)\in [G \backslash \calL_b(G,L,u) / \Aut(L,u) ]} \hspace{-3ex} \dim_{\FF}( V^{\Out(L,u)_{(P_\gamma,\pi)}}) = \sum_{(Q_\delta,\rho)\in [H \backslash \calL_c(H,L,u) / \Aut(L,u) ]} \hspace{-3ex}\dim_{\FF}(V^{\Out(L,u)_{(Q_\delta,\rho)}})
\end{align*}
which implies that
\begin{align*}
\sum_{(P_\gamma,\pi)\in [G \backslash \calL_b(G,L,u) / \Aut(L,u) ]} \hspace{-3ex} \langle \chi_V,  1\rangle_{\Out(L,u)_{(P_\gamma,\pi)}} =  \sum_{(Q_\delta,\rho)\in [H \backslash \calL_c(H,L,u) / \Aut(L,u) ]} \hspace{-3ex}\langle \chi_V,  1\rangle_{\Out(L,u)_{(Q_\delta,\rho)}}\,.
\end{align*}
Since $\FF\Out(L,u)$ is semisimple, it follows that this equality holds for any class function on $\Out(L,u)$. So let $\calC$ be the conjugacy class of $\overline{\varphi}$ in $\Out(L,u)$ and $\chi_\calC$ denote the characteristic function of $\calC$. We have
\begin{align*}
&\sum_{(P_\gamma,\pi)\in [G \backslash \calL_b(G,L,u) / \Aut(L,u) ]} \hspace{-3ex} \langle \chi_\calC,  1\rangle_{\Out(L,u)_{(P_\gamma,\pi)}}\\&
=\sum_{(P_\gamma,\pi)\in [G \backslash \calL_b(G,L,u) / \Aut(L,u) ]} \frac{1}{|\Out(L,u)_{(P_\gamma,\pi)}|} \sum_{\overline{\psi}\in \Out(L,u)_{(P_\gamma,\pi)}} \chi_\calC(\overline{\psi})\\&
=\sum_{(P_\gamma,\pi)\in [G \backslash \calL_b(G,L,u)]} \frac{1}{|\Out(L,u)|} \sum_{\overline{\psi}\in \Out(L,u)_{(P_\gamma,\pi)} \cap \calC} 1\\&
=\frac{1}{|\Out(L,u)|} \sum_{\overline{\psi}\in \calC} \sum_{(P_\gamma,\pi)\in [(G \backslash \calL_b(G,L,u))^{\overline{\psi}}]} 1\\&
=\frac{1}{|\Out(L,u)|} |\calC| |[(G \backslash \calL_b(G,L,u))^{\overline{\varphi}}]| 
\end{align*}
Similarly, we have
\begin{align*}
\sum_{(Q_\delta,\rho)\in [H \backslash \calL_c(H,L,u) / \Aut(L,u) ]} \hspace{-3ex} \langle \chi_\calC,  1\rangle_{\Out(L,u)_{(Q_\delta,\rho)}}=\frac{1}{|\Out(L,u)|} |\calC| |[(H \backslash \calL_c(H,L,u))^{\overline{\varphi}}]|\,.
\end{align*}
Therefore, we have $|[(G \backslash \calL_b(G,L,u))^{\overline{\varphi}}]|= |[(H \backslash \calL_c(H,L,u))^{\overline{\varphi}}]|$ which implies {\rm (ii)}.

\smallskip
{\rm (ii)} $\Rightarrow$ {\rm (i)}: Suppose that {\rm (ii)} holds. Then the permutation $\FF\Aut(L,u)$-modules $\FF [G\backslash \calL_b(G,L,u)]$ and $\FF[H\backslash \calL_c(H,L,u)]$ have the same character and hence they are isomorphic. The result then follows from Theorem~\ref{thm multiplicityformulas}(iii). 
\end{proof}
\begin{theorem} \label{finiteness}Let $D$ be a finite $p$-group. Then there is only a finite number of pairs $(G,b)$, where $G$ is a finite group, and $b$ is a block idempotent of $kG$ with defect $D$, up to functorial equivalence over $\FF$.
\end{theorem}
\begin{proof}
We know that the functor $\FFTD_{G,b}$ splits as a direct sum of simple functors $S_{L,u,V}$. Using Theorem~\ref{thm multiplicityformulas} (b), we will show that only a finite number (depending only on $D$) of simple functors $S_{L,u,V}$, can appear in $\FFTD_{G,b}$, and that the multiplicity of $S_{L,u,V}$ as a summand of $\FFTD_{G,b}$ is bounded by a constant depending only on $D$. This will imply that, up to isomorphism, there is only a finite number of possibilities for the functor $\FFTD_{G,b}$, once the defect $D$ of $b$ is fixed.\par
Recall that the simple functors $S_{L,u,V}$ are parametrized by triples $(L,u,V)$, where $L$ is a finite $p$-group, $u$ is a faithful $p'$-automorphism of $L$, and $V$ is a simple $\FF \Out(L,u)$-module, where $\Out(L,u)$ is a quotient of $\Aut(L,u)$, itself a subgroup of the automorphism group $\Aut(L)$ of $L$. By Theorem~\ref{thm multiplicityformulas} (b), the multiplicity of $S_{L,u,V}$ as a summand of~$\FFTD_{G,b}$ is equal to the $\FF$-dimension of
$$m_{l,u,V}(b)=\bigoplus_{(P,e_P) \in [\calF_b]} \bigoplus_{\pi \in [\calP_{(P,e_P)}(L,u)]}  \FF\Proj(keC_G(P),u)\otimes_{\Aut(L,u)_{P,e,\pi}} V,$$
where $\calF_b$ is the fusion system of $b$ and $[\calP_{(P,e_P)}(L,u)]$ is a set of $N_G(P,e_P)\times\Aut(L,u)$-orbits of the set $\calP_{(P,e_P)}(L,u)$ of group isomorphisms $\pi:L\to P$ such that $\pi\circ i_u\circ \pi^{-1}$ is an automorphism of $(P,e)$ in the fusion system $\calF_b$. Moreover $\Proj\big(keC_G(P),u\big)$ is a subgroup of the group of projective $keC_G(P)$-modules.\par
It follows that if $S_{L,u,V}$ appears in $\FFTD_{G,b}$ with non zero multiplicity, then $L$ is isomorphic to a subgroup of $D$. Hence there is only a finite number of such groups $L$, up to isomorphism. For each~$L$, there is only a finite number - at most $|\Aut(L)|$ - of faithful $p'$-automorphisms $u$ of $L$, and for each such $u$, there is only a finite number - at most $|\Aut(L)|$ again - of simple $\FF\Aut(L,u)$-modules, up to isomorphism. Hence the number of simple summands of $\FFTD_{G,b}$ is bounded by a number $c_D$ depending only on $D$.\par
Now for such a summand $S_{L,u,V}$, the dimension of the $\FF$-vector space $\FF\Proj\big(keC_G(P),u\big)\otimes_{\Aut(L,u)_{P,e,\pi}} V$ is less than or equal to $\dim \FF\Proj\big(keC_G(P)\big)\dim V$, and moreover $\dim V\le|\Out(L,u)|\le |\Aut(L)|$. Now $\dim \FF\Proj\big(keC_G(P)\big)$ is equal to the number $l\big(keC_G(P)\big)$ of simple $keC_G(P)$-modules, which is equal to the number $l\big(kePC_G(P)\big))$ of simple $kePC_G(P)$-modules, as $P$ acts trivially on each simple $kePC_G(P)$-module. Now $e$ is a block idempotent of $kPC_G(P)$, and by Corollary~4.5 of~\cite{alperin-broue1979}, we can assume that $P\le D$, and that $e$ has defect $PC_D(P)$, which is a subgroup of $D$. By Theorem~1 of~\cite{brauer-feit1959}, the number of ordinary  irreducible characters in a block $c$ of a finite group $H$ is at most $\frac{1}{4}p^{2d}+1$, where $p^d$ is the order of the defect group of $c$. This number is smaller than $2p^{2d}$. It follows that $l\big(keC_G(P)\big)$, which is at most equal to the number of ordinary irreducible characters in the block $e$ of $PC_G(P)$, is smaller than $2|PC_D(P)|^2\le 2|D|^2$. \par
Finally, we get that 
$$m_{L,u,V}(b)\le n_{D,L}2|D|^2|\Aut(L)|^2,$$
where $n_{D,L}$ is the number of subgroups of $D$ isomorphic to $L$. So there is a constant $m_D$, depending only on $D$, such that $m_{L,u,V}(b)\le m_D$ (for example $m_D=n_D2|D|^2M_D$, where $n_D$ is the number of subgroups of $D$, and $M_D$ is the sup of $|\Aut(L)|$ over subgroups $L$ of~$D$). This bound only depends on $D$, as was to be shown. This completes the proof.
\end{proof} 

\section{Blocks with abelian defect groups}

Let $b$ be a block idempotent of $kG$ with an abelian defect group $D$. Let $c$ be a block idempotent of $kH$ which in Brauer correspondence with $b$ where $H=N_G(D)$. Let $(L,u)$ be a $\DD$-pair with the property that the multiplicity of $S_{L,u,V}$ in $\FFTD_{G,b}$ is non-zero for some $V\in \FF\Out(L,u)$-mod. Then $L$ is also abelian.  

Let $(D,e_D)$ be a maximal $b$-Brauer pair and note that $(D,e_D)$ is also a maximal $c$-Brauer pair. For every $P\le D$, let $e_P\in \mathrm{Bl}(kC_G(P))$ and $f_P\in \mathrm{Bl}(kC_H(P))$ such that $(P,e_P)\le_G (D,e_D)$ and $(P,f_P)\le_H (D,e_D)$.  One can show that the block idempotents $e_P$ and $f_P$ are Brauer correspondents. Let $\calF_b$ be the fusion system of $b$ associated to $(D,e_D)$ and let $\calF_c$ be the fusion system of $c$ associated to $(D,e_D)$.  

Recall that by (\ref{eqn isomwithfusion}), we have isomorphisms
\begin{align*}
b\FFTD(G,L\langle u\rangle)\widetilde{F^{L\langle u\rangle}_{L,u}}\cong \bigoplus_{(P,e_P) \in [\calF_b]} \bigoplus_{\pi \in [\calP_{(P,e_P)}^G(L,u)]} \Ind_{\Aut(L,u)_{P,e_P,\pi}}^{\Aut(L,u)}  \FF\Proj(ke_PC_G(P),u)
\end{align*}
and
\begin{align*}
c\FFTD(H,L\langle u\rangle)\widetilde{F^{L\langle u\rangle}_{L,u}}\cong \bigoplus_{(P,f_P) \in [\calF_c]} \bigoplus_{\pi \in [\calP_{(P,f_P)}^H(L,u)]} \Ind_{\Aut(L,u)_{P,f_P,\pi}}^{\Aut(L,u)}  \FF\Proj(kf_PC_H(P),u)
\end{align*}
of right $\FF\Aut(L,u)$-modules. 

\smallskip
{\rm (a)} Let $(P,e_P)\in\calF_b$.  Let also $s\in N_G(P,e_P)$. Then since $D$ is abelian, we have $N_G(P,e_P)\subseteq N_G(D,e_D)C_G(P)$. Hence we have $s=ht$ for some $h\in N_G(D,e_D)\le H$ and $t\in C_G(P)$.  Thus $i_s = i_h:P\to P$ and $h\in N_H(P,e_P)$.

\smallskip
{\rm (b)} We have $(P,e_P)\equiv_{\calF_b} (Q, e_Q)$ if and only if $(P,f_P)\equiv_{\calF_c} (Q, f_Q)$. Indeed, let $\phi: (P,e_P)\to (Q,e_Q)$ be an isomorphism in $\calF_b$.  Then by Alperin's fusion theorem, there exists an automorphism $\psi:(D,e_D)\to (D,e_D)$ such that $\psi|_P=\phi$.  There exists an element $g\in G$ such that $\psi(x)=\lexp{g}x$ for any $x\in D$ and such that $\lexp{g}(D,e_D)=(D,e_D)$.  This implies in particular that $g\in N_G(D,e_D)\le H$.  This shows that $\psi\in\calF_c$ and $\lexp{g}f_P=\lexp{g}{(\Br_D^{C_G(P)}(e_P))}=\Br_D^{C_G(Q)}(\lexp{g}e_P)=f_Q$.  Hence $\psi |_P:(P,f_P)\to (Q,f_Q)$ is an isomorphism in $\calF_c$. 

Conversely, let $\phi: (P,f_P)\to (Q,f_Q)$ be an isomorphism in $\calF_c$. Then again by Alperin's fusion theorem, there exists $\psi: (D,e_D)\to (D,e_D)$ such that $\psi |_P = \phi$. Let $h\in N_H(D,e_D)$ with $\psi=i_h$. We have $f_Q=\lexp{h}f_P=\lexp{h}{(\Br_D^{C_G(P)}(e_P))}=\Br_D^{C_G(Q)}(\lexp{h}e_P)$ which implies that $\lexp{h}e_P$ and $f_Q$ are Brauer correspondents. Hence $\lexp{h}e_P=e_Q$ and $\psi |_P:(P,e_P)\to (Q,e_Q)$ is an isomorphism in $\calF_b$. 

\smallskip
{\rm (c)} We have $\calP_{(P,e_P)}^G(L,u) = \calP_{(P,f_P)}^H(L,u)$ for any $P\le D$.  Let $\pi\in \calP_{(P,e_P)}^G(L,u)$. Then by definition $\pi i_u\pi^{-1} \in\Aut_{\calF_b}(P,e_P)$. So there exists $s\in N_G(P,e_P)$ such that $\pi i_u\pi^{-1} = i_s$. By part~{\rm (a)}, there exists $h\in N_H(P,e_P)$ with $i_s=i_h$.  But then $h\in N_H(P,f_P)$ and hence $\pi\in  \calP_{(P,e_P)}^H(L,u)$ as desired. 

Conversely, let $\pi\in \calP^H_{(P,f_P)}(L,u)$. Then $\pi i_u\pi^{-1}=i_h$ for some $h\in N_H(P,f_P)$. By part {\rm (a)} we can choose $h\in N_H(D,e_D)$.  Then one shows that $\lexp{h}e_P = e_P$ which implies that $h\in N_H(P,e_P)\subseteq N_G(P,e_P)$.  Therefore, $\pi\in  \calP_{(P,e_P)}^G(L,u)$.

\smallskip
{\rm (d)} Let $P\le D$ and let $\pi, \rho \in \calP_{(P,f_P)}^H(L,u)=\calP_{(P,e_P)}^G(L,u)$. Then $\rho \in N_G(P,e_P) \cdot \pi \cdot \Aut(L,u)$ if and only if $\rho\in N_H(P,f_P) \cdot \pi \cdot \Aut(L,u)$.  Indeed, suppose that for some $g\in N_G(P,e_P)$ and $\varphi \in \Aut(L,u)$ we have $\rho =i_g \pi\varphi$.  By part {\rm (a)}, there exists $h\in N_H(P,e_P)$ with $i_g=i_h$.  But then $h\in N_H(P,f_P)$ and $\rho=i_h\pi \varphi$ which proves the claim. Converse is proved similarly.

\smallskip
{\rm (e)} Let $P\le D$ and let $\pi\in [\calP_{(P,e_P)}^G(L,u)] = [\calP_{(P,f_P)}^H(L,u)]$. We have $\Aut(L,u)_{P,e_P,\pi}^G=\Aut(L,u)_{P,f_P,\pi}^H$.  Indeed, let $\varphi\in \Aut(L,u)_{P,e,\pi}^G$.  Then there exists $g\in N_G(P,e_P)$ such that $\pi \varphi \pi^{-1} = i_g$. As above this means that there exists $h\in N_H(P,f_P)$ such that $\pi \varphi \pi^{-1} = i_g = i_h$. This proves that $\varphi \in \Aut(L,u)_{P,f_P,\pi}^H$. 

Conversely, if $\varphi \in \Aut(L,u)_{P,f_P,\pi}^H$, then there exists $h\in N_H(P,f_P)$ such that $\pi \varphi \pi^{-1}=i_h$. Again as above, we can choose $h\in N_H(D,e_D)$ and hence it follows that $h\in N_H(P,e_P)\subseteq N_G(P,e_P)$. Therefore, $\varphi\in \Aut(L,u)_{P,e,\pi}^G$.

All these imply the following.
\begin{theorem}\label{abelian defect}
Let $b$ be a block idempotent of $kG$ with an abelian defect group $D$. Let $c$ be a block idempotent of $kH$ which is in Brauer correspondence with $b$ where $H=N_G(D)$. Let $(D,e_D)$ be a maximal $b$-Brauer pair. For every $P\le D$, let $e_P\in \mathrm{Bl}(kC_G(P))$ and $f_P\in \mathrm{Bl}(kC_H(P))$ such that $(P,e_P)\le_G (D,e_D)$ and $(P,f_P)\le_H (D,e_D)$.

\smallskip
{\rm (a)} The multiplicity of $S_{D,u,V}$ in $\FFTD_{G,b}$ and $\FFTD_{H,c}$ are the same.

\smallskip
{\rm (b)} If for every $P\le D$ and $s\in N_H(P,f_P)_{p'}$ we have an isomorphism
\begin{align*}
\FF\Proj(ke_P C_G(P),s)\cong \FF\Proj(kf_PC_H(P),s)
\end{align*}
of $\FF N_H(P,f_P)$-modules, then $(G,b)$ and $(H,c)$ are functorially equivalent over $\FF$. 
\end{theorem}

\section*{Acknowledgement}

The second author is supported by the Scientific and Technological Research Council of Turkey (T{\"u}bitak) through 2219 -- International Postdoctoral Research Fellowship Program.

\centerline{\rule{5ex}{.1ex}}
\begin{flushleft}
Serge Bouc, CNRS-LAMFA, Universit\'e de Picardie, 33 rue St Leu, 80039, Amiens, France.\\
{\tt serge.bouc@u-picardie.fr}\vspace{1ex}\\
Deniz Y\i lmaz, LAMFA, Universit\'e de Picardie, 33 rue St Leu, 80039, Amiens, France.\\
{\tt deyilmaz@ucsc.edu}
\end{flushleft}
\end{document}